\documentclass[12pt]{amsart}
\usepackage{amsmath,amssymb,amsfonts,latexsym}
\usepackage{mathrsfs}
\usepackage[dvips]{geometry}
\usepackage{array}
\usepackage{epsf}
\usepackage{epsfig}
\newtheorem*{lemma}{Lemma}
\newtheorem*{prop}{Proposition}
\newtheorem*{thm}{Theorem}
\newcommand{\iso}{\overset{\sim}{\rightarrow}}

\newcommand{\ad}{\operatorname{ad}}
\newcommand{\im}{\operatorname{Im}}
\newtheorem*{cor}{Corollary}

\newcommand{\nc}{\newcommand}
\nc{\Spec}{\operatorname{Spec}}
\nc{\Index}{\operatorname{index}}
\nc{\GKdim}{\operatorname{GKdim}}
\nc{\gr}{\operatorname{gr}}
\nc{\aut}{\operatorname {Aut}}
\nc{\Ker}{\operatorname {Ker}}
\nc{\rk}{\operatorname {rk}}
\nc{\tr}{\operatorname {tr}}
\nc{\supp}{\operatorname{supp}}
\nc{\Aut}{\operatorname{Aut}}
\nc{\Id}{\operatorname{Id}}

\begin{document}

\title[Adapted Pairs]{The Integrality of an Adapted Pair}
\author[Anthony Joseph]{Anthony Joseph}\footnote {Work supported in part by Israel Science Foundation Grant, no. 710724.}

\date{\today}
\maketitle

\vspace{-.9cm}\begin{center}
Donald Frey Professional Chair\\
Department of Mathematics\\
The Weizmann Institute of Science\\
Rehovot, 76100, Israel\\
anthony.joseph@weizmann.ac.il
\end{center}\



 \

\textbf{Abstract}

Let $\mathfrak a$ be an algebraic Lie algebra.  An adapted pair for $\mathfrak a$ is pair $(h,\eta)$ consisting of an ad-semisimple element of $h \in \mathfrak a$ and a regular element of $\eta \in \mathfrak a^*$ satisfying $(\ad h)\eta=-\eta$.  In general such pairs are not easy to find and even more difficult to classify.  A natural question is whether $\ad h$ has integer eigenvalues on $\mathfrak a$, a property called the integrality of the adapted pair.  In  general this fails even for a Frobenius subalgebra of $\mathfrak {sl}(4)$ and rather seriously in the sense that any rational number may serve as an eigenvalue.  Nevertheless integrality is shown to hold for any Frobenius Lie algebra which is a biparabolic subalgebra of a semisimple Lie algebra.

Call $\mathfrak a$ regular if there are no proper semi-invariant polynomial functions on $\mathfrak a^*$ and if the subalgebra of invariant functions is polynomial. In this case there are no known counter-examples to integrality.  It is shown that if $\mathfrak a$ is the canonical truncation of a biparabolic subalgebra of a simple Lie algebra $\mathfrak g$ which is regular and admits an adpated pair $(h,\eta)$,  then the eigenvalues of $\ad h$ on $\mathfrak a$ lie in $\frac{1}{m}\mathbb Z$, where $m$ is a coefficient of a simple root in the highest root of $\mathfrak g$.

Let $\mathfrak a$ be a regular Lie algebra admitting an adapted pair $(h,\eta)$.  Let $\mathfrak a_\mathbb Z$ be the subalgebra spanned by the eigensubspaces of $\ad h$ with integer eigenvalue. It is shown that the canonical truncation of $\mathfrak a_\mathbb Z$ is regular.  Sufficient knowledge of the relation between the generators for the invariant polynomial functions on  $\mathfrak a^*$ and on $\mathfrak a^*_\mathbb Z$ can then lead to establishing that $\mathfrak a=\mathfrak a_\mathbb Z$.  A particular interesting case is when $\mathfrak a$ is the canonical truncation of a biparabolic subalgebra of a simple Lie algebra $\mathfrak g$. If $\mathfrak g$ is of type $A$, then integrality already holds by the paragraph above. If $\mathfrak g$ is of type $C$ and $\mathfrak a$ is a truncated \textit{parabolic} subalgebra then a rather refined analysis shows that $\mathfrak a= \mathfrak a_\mathbb Z$.   In principle this method can also be applied to biparabolic subalgebras in type $C$ but there are some difficult combinatorial questions involving meanders to be resolved.  Outside types $A$ and $C$ further technical complications arise out of an insufficient knowledge of the subalgebra of invariant polynomial functions on the dual of a biparabolic.


\section{Introduction}

The base field $\textbf{k}$ is assumed algebraically closed and of characteristic zero throughout.

\subsection{Adapted pairs}\label{1.1}

Let $\mathfrak a$ be an algebraic Lie algebra and denote its algebraic adjoint group by a corresponding boldface Roman letter that is by $\bf{A}$.   Let $S(\mathfrak a)$ be the symmetric algebra of $\mathfrak a$ and set $Y(\mathfrak a)=S(\mathfrak a)^{\bf {A}}$.  Let $Y(\mathfrak a)_+$ denote the augmentation of $Y(\mathfrak a)$ and let $\mathscr N(\mathfrak a)$ denote the zero locus of $S(\mathfrak a)Y(\mathfrak a)_+$.

Given $a \in \mathfrak a$, define $\ad a$ to be the linear endomorphism of $\mathfrak a$ mapping $b$ to $[a,b]$, for all $b \in \mathfrak b$.  This defines the adjoint action of $\mathfrak a$ on itself. We also use $\ad a$ to denote the coadjoint action of $a$ on the dual space $\mathfrak a^*$ obtained by transport of structure.  Then $Y(\mathfrak a)$ identifies with the algebra of $\textbf{A}$ invariant polynomial functions on $\mathfrak a^*$.

Set $\mathfrak a^\xi: =\{a \in \mathfrak a|(\ad a)\xi=0\}$.  An element $\xi \in \mathfrak a^*$ is called regular if $\dim \mathfrak a^\xi$ takes its minimal possible value denoted $\ell(\mathfrak a)$ and called the index of $\mathfrak a$.
Let $\mathfrak a^*_{reg}$ denote the set of regular elements of $\mathfrak a^*$.
Let $I(\mathfrak a)$ (or simply, $I$) be the set $\{1,2,\ldots, \ell(\mathfrak a)\}$.

An adapted pair for $\mathfrak a$ is a pair $(h,\eta) \in \mathfrak a \times \mathfrak a^*_{reg}$ such that $(\ad h)\eta= -\eta$.  Since $\mathfrak a$ is algebraic we can and do assume that $\ad_\mathfrak ah$ is a semisimple endomorphism.   We refer to $h$ as the ad-semisimple element of the adapted pair $(h,\eta)$.

The main interest of an adapted pair is to construct a Weierstrass section (\ref {2.1}) for the $\textbf{A}$ orbits in $\mathfrak a^*$.  This provides a canonical form for ``most" orbits.  As was pointed out by Popov \cite [Example 2.2.2] {P} it extends Weierstrass canonical form for elliptic curves.  This is worked out in detail in [Sect. 2]{J9}.

For example if $\mathfrak g$ is simple and $(x,h,y)$ is a principal s-triple for $\mathfrak g$, then identifying $\mathfrak g^*$ with $\mathfrak g$ through the Killing form makes the subset $\{h,y\}$ into an adapted pair.  Notice here we do not take the usual normalisation; but rather the one for which $[h,y]=-y$.  The latter is more convenient because the principal nilpotent orbit is ``even" with the consequence that the eigenvalues of $\ad h$ on $\mathfrak g$ are still integer-valued.  Moreover the eigenvalues of $\ad h$ on $\mathfrak g^x$ which is a complement of $(\ad \mathfrak g)y$ are just the ``exponents" $\{m_i\}_{i\in I} $ of $\mathfrak g$.  Here we recall that $Y(\mathfrak g)$ is polynomial and its homogeneous generators have degrees $m_i+1:i\in I$. After Kostant \cite {K1}, the linear subvariety $y+\mathfrak g^x$ is Weierstrass section for $\textbf{G}$ orbits in $\mathfrak g^*$.  Rather exceptionally $\mathscr N(\mathfrak g)$ is irreducible.  As a consequence $\textbf{G}(y+\mathfrak g^x) = \mathfrak g^*_{reg}$, which is also rather exceptional.

\subsection{Semi-invariants}\label{1.2}

Retain the above notation.  Let $F(\mathfrak a)$ denote the fraction field of $S(\mathfrak a)$ and set $C(\mathfrak a)=F(\mathfrak a)^{\textbf {A}}$. A general result of Chevalley-Rosenlicht asserts that the transcendence degree of $C(\mathfrak a)$ equals $\ell(\mathfrak a)$.

A vector spanning a one dimensional module for $S(\mathfrak a)$ under adjoint action is called a semi-invariant. The set of semi-invariants of $S(\mathfrak a)$ is stable under taking products and factors.  Consequently the space $Sy(\mathfrak a)$ spanned by the set of all semi-invariants of $S(\mathfrak a)$ is a subalgebra of $S(\mathfrak a)$.  It is an easy consequence of \cite [3.5]{J1} that $Sy(\mathfrak a)$ is a unique factorisation domain (for further details and references - see \cite [1.3]{J1}).

An argument of Chevalley-Dixmier shows that if one writes $\xi \in C(\mathfrak a)$ in the form $\xi=a^{-1}b$ with $a,b \in S(\mathfrak a)$ coprime, then both $a$ and $b$ are semi-invariants.  Thus $Sy(\mathfrak a)$ is in general quite big, its Gelfand-Kirillov dimension, or growth rate, being at least $\ell(\mathfrak a)$.

On the other hand $Y(\mathfrak a)$ may be rather small and can even be reduced to scalars.

It is sometimes convenient to rectify the above situation using an observation of Borho-Chevalley.  A semi-invariant $a \in S(\mathfrak a)$  comes with a character $\lambda$ of $\mathfrak a$, defined through the relation $(\ad x)a=\lambda(x)a, \forall x \in \mathfrak a$.  Let $\Lambda$ be the set of all characters of $\mathfrak a$ which appear in the decomposition of $Sy(\mathfrak a)$.  Then  $\mathfrak a_\Lambda: =\cap_{\lambda \in \Lambda} \Ker \lambda $ is called the canonical truncation of $\mathfrak a$.  It is an algebraic Lie algebra and an ideal of $\mathfrak a$.  After Borho-Chevalley one has
$$Sy(\mathfrak a)=Sy(\mathfrak a_\Lambda)=Y(\mathfrak a_\Lambda).$$

Moreover one readily obtains the identity
$$\dim \mathfrak a +\ell(\mathfrak a)= \dim \mathfrak a_\Lambda +\ell(\mathfrak a_\Lambda). \eqno {(*)}$$

For all this and the modifications needed in the not necessarily algebraic case see \cite {OV}. One may remark that $(*)$ also holds with $\mathfrak a_\Lambda$ on the right hand side replaced by any subalgebra of $\mathfrak a$ containing $\mathfrak a_\Lambda$.

A disadvantage of passing to the canonical truncation is that adapted pairs become more rare and may easily fail to exist.

Define the reduced index $r\ell(\mathfrak a)$ of $\mathfrak a$ to be the index of its canonical truncation $\mathfrak a_\Lambda$.  By the above it is the Gelfand-Kirillov dimension (that is growth rate) of $Sy(\mathfrak a)$.  In particular if $Sy(\mathfrak a)$ is polynomial then $r\ell(\mathfrak a)$ is just the number of algebraically independent generators of $Sy(\mathfrak a)$.

\subsection{Truncated Biparabolics and Centralizers}\label{1.3}

Fix a root system $\Delta$ and $\pi$ a choice of simple roots. Let $\mathfrak g_\pi$ (or simply $\mathfrak g$) be the corresponding semisimple Lie algebra.  Unless otherwise stated we shall assume $\pi$ indecomposable, so then $\mathfrak g_\pi$ is simple and we can speak of its type. For all $\alpha \in \Delta$, let $x_\alpha$ denote the element of a Chevalley basis for $\mathfrak g$ of weight $\alpha$ and let $\kappa$ denote the corresponding Chevalley involution for $\mathfrak g$.  Let $s_\alpha$ be the reflection corresponding to the root $\alpha \in \Delta$ and $\varpi_\alpha$ the fundamental weight corresponding to the root $\alpha \in \pi$.

 By definition a Borel subalgebra of $\mathfrak g$ is a maximal solvable subalgebra. The set of all Borel subalgebras of $\mathfrak g$ is a single $\bf{G}$ orbit. Let $\mathfrak b_\pi$ be the Borel subalgebra of $\mathfrak g_\pi$ with roots in $\Delta^+:=\Delta \cap \mathbb N \pi$.

 By definition a parabolic subalgebra $\mathfrak p$ contains a Borel subalgebra.  Up to conjugation by $\textbf{G}$ we may assume the latter to be $\mathfrak b_\pi$ and then $\mathfrak p$ is determined by a subset $\pi_1$ of $\pi$, being the simple roots of its Levi factor. In this case we write $\mathfrak p =\mathfrak p_{\pi_1}$. Set $\mathfrak p^-_{\pi_1}:=\kappa(\mathfrak p_{\pi_1})$.

 By definition a biparabolic subalgebra $\mathfrak q$ is the intersection of two parabolic subalgebras whose sum is $\mathfrak g$. A biparabolic subalgebra can be conjugated by $\textbf{G}$ into one of the form $\mathfrak q_{\pi_1,\pi_2} :=\mathfrak p_{\pi_1} \cap \mathfrak p_{\pi_2}^-$.  Here may assume $\pi_1\cup \pi_2 = \pi$ without loss generality. This together with our assumption that $\pi$ is indecomposable and that $\pi_1\cap \pi_2 \varsubsetneq \pi$ (which excludes $\mathfrak g_\pi$ itself), will be called our standing hypothesis. They imply that $\mathfrak q_{\pi_1,\pi_2}$ admits no semisimple ideal.

 One checks that the set of roots of $\mathfrak q_{\pi_1,\pi_2}$ is $\mathbb N\pi_2 \cup -\mathbb N\pi_1$.

For a biparabolic $\mathfrak q$ subalgebra admitting no semisimple ideal, $Y(\mathfrak q)$ is reduced to scalars \cite [proof of Lemma 7.9]{J2}. On the other hand $Sy(\mathfrak q)$ is often polynomial, for example if $\pi$ is of type $A$ or type $C$.  Thus it is of interest to consider its canonical truncation.  If $\pi$ is of type $A$, then the canonical truncation of $\mathfrak q$ admits an adapted pair \cite {J4}. Outside type $A$ adapted pairs sometimes but not always exist.

Again one may consider a centralizer $\mathfrak g^x$ in a simple Lie algebra $\mathfrak g$.  One may reduce to the case when $\ad x$ is a nilpotent derivation.  Then $Y(\mathfrak g^x)$ is often polynomial \cite {PPY}, for example if $\pi$ is of type $A$ or type $C$. Here it is less appropriate to consider the canonical truncation of $\mathfrak g^x$ for the reasons discussed in \cite {JS}.

\subsection{Integrality}\label{1.4}

Let $\mathfrak a$ be an algebraic Lie algebra and $h \in \mathfrak a$, which is ad-semisimple.

For all $i \in \textbf{k}$, set $\mathfrak a_i= \{a \in \mathfrak a| (\ad h)a=ia\}$, $\mathbb I := \{i \in \textbf{k}| \mathfrak a_i\neq 0\}$ and $\mathfrak a_\mathbb Z:=\oplus_{i \in \mathbb Z}\mathfrak a_i$.

Now let $(h,\eta)$ be an adapted pair for $\mathfrak a$ such that $\ad h$ is semisimple.

On a recent visit to the Weizmann Institute, Elashvili asked if one always has $\mathbb I \subset \mathbb Z$, equivalently if $\mathfrak a = \mathfrak a_\mathbb Z$. In this case we say that the adapted pair $(h,\eta)$ has the integrality property.  Though I had considered this question before and indeed it was the reason that in \cite [2.7]{J5} I had chosen the normalization $(\ad h)\eta=-\eta$, I previously had not take this question too seriously.  Already in \cite [2.7]{J5} it did not even seem obvious that the eigenvalues on the stabilizer $\mathfrak a^\eta$ are integer.

The integrality property can fail (see \ref {1.5}) but one can ask if it holds for regular Lie algebras (see \ref {2.1}) which includes most truncated biparabolics.  One advantage of this restriction is that the eigenvalues of $\ad h$ are integer on the stabilizer $\mathfrak a^\eta$ \cite [Cor. 2.3]{JS}.

If $\mathfrak q$ is a truncated biparabolic subalgebra of a simple Lie algebra of type $A$, then the semi-simple element $h$ of an adapted pair satisfies the integrality property \cite [9.10]{FJ2}.  Already generalizing this to truncated parabolics in the symplectic case is very difficult, let alone for all truncated biparabolics and indeed all regular Lie algebras. Here we develop some general techniques to handle this question and apply it to proving that the integrality property holds for adapted pairs in all truncated parabolic subalgebras of $\mathfrak {sp}(2n)$.   Even this rather special case occupies the whole of the rather long Section \ref {8}.  The difficulty is inherent in the fact that we have very little idea of how to classify adapted pairs.

One may remark that $\mathfrak g^x$ admits an adapted pair \cite {J6} if $\pi$ is of type $A$ and this pair satisfies integrality by \cite [3.3,3.4]{JS} combined with \cite [Cor. 4.15]{J6}, or by direct computation.

\subsection{The Frobenius Case}\label{1.5}

Recall that an algebraic Lie algebra $\mathfrak a$ is called Frobenius if $\ell(\mathfrak a)=0$.  If $\mathfrak a$ is Frobenius, it admits an adapted pair $(h,\eta)$ which in addition is unique up to conjugation by $\textbf{A}$.  In this case $\mathfrak a^\eta=0$. Yet the integrality property generally fails. Indeed \cite [9.11]{FJ2} for every rational number $c$, there is a four dimensional algebraic subalgebra of the Borel subalgebra $\mathfrak b$ of $\mathfrak {sl}(3)$ of index $0$ for which the semisimple element of the adapted pair admits $c$ as an eigenvalue.  In the language of \cite {J8} these algebras are the possible ablations of $\mathfrak b$ which is itself an example of an almost-Frobenius biparabolic.  It is general phenomenon that almost Frobenius biparabolics admit a family of ablations which are Frobenius having a unique up to equivalence adapted pair which generally does not satisfy integrality.

On the other hand the structure of Frobenius biparabolic subalgebras is more rigid, and for them the integrality property holds (Theorem \ref {5.10}).  Outside type $A$ this is not entirely trivial.

\subsection{}\label{1.6}

Let $\mathfrak a$ be an algebraic Lie algebra.   Here we recall that $\mathfrak a$ is algebraic if and only if we can write $\mathfrak a = \mathfrak r \oplus \mathfrak m$, with $\mathfrak m$ the nilpotent radical and $\mathfrak r$ reductive with the property that any Cartan subalgebra $\mathfrak h$ of $\mathfrak r$ admits a basis $\{h_k\}_{k \in K}$ (with $K$ an appropriate index set) such that $\ad h_k$ has integer eigenvalues on $\mathfrak a$ for all $k \in K$.  We call $\mathfrak r$ the reductive part of $\mathfrak a$.  Its isomorphism class is canonically determined as $\mathfrak a/\mathfrak m$.

 In addition we recall that $\mathfrak r$ is reductive if and only if it admits a non-degenerate $\ad \mathfrak r$ invariant pairing $\varphi :\mathfrak r \times \mathfrak r \rightarrow \textbf{k}$.

Let $(h,\eta)$ an adapted pair for $\mathfrak a$.  Define $\mathfrak a_\mathbb Z$ as in \ref {1.4}.  Obviously the integrality of the pair is equivalent to the assertion that $\mathfrak a = \mathfrak a_\mathbb Z$.  Thus it is useful to note the following

\begin {lemma}  $\mathfrak a_\mathbb Z $ is an algebraic Lie algebra.
\end {lemma}

\begin {proof}  This is rather obvious but we give the details anyway.

Let $(h,\eta)$ be an adapted pair for $\mathfrak a$.  Recall notation of the first part of \ref {2.1}.  We can assume that $h \in \mathfrak h$.  Then $\varphi$ restricts to a non-degenerate pairing $\mathfrak r_i \times \mathfrak r_{-i} \mapsto \textbf{k}$.  Thus  $\mathfrak r_\mathbb Z$ is reductive with Cartan subalgebra $\mathfrak h$. On the other hand $\mathfrak m_\mathbb Z \subset \mathfrak m$ so is nilpotent and an ideal in $\mathfrak a_\mathbb Z$, hence must be its nilpotent radical.  We conclude that $\mathfrak a_\mathbb Z$ is algebraic.
\end {proof}
\section{The Regular Case}\label{2}

\subsection{Regular Lie algebras and Weierstrass Sections}\label{2.1}

\

Let $\mathfrak a$ be an algebraic Lie algebra.  Throughout this section fix a decomposition $\mathfrak a = \mathfrak r \oplus \mathfrak m$ as in \ref {1.6} let $\mathfrak h$ be a Cartan subalgebra of $\mathfrak r$. Define $\mathfrak a_\mathbb Z$ as in \ref {1.4} and observe that $\mathfrak h$ is also a Cartan subalgebra for $\mathfrak r_\mathbb Z=\mathfrak a_\mathbb Z \cap \mathfrak r$.

Take $\eta \in \mathfrak a^*$ (not necessarily regular) and $V \in \mathfrak a^*$ a subspace. One calls $\eta+V$ a linear subvariety of $\mathfrak a^*$.

  There are two different senses for a linear subvariety $\eta+V$ to be a Weierstrass section for the action of $\bf{A}$ on $\mathfrak a$.

 The algebraic sense given by Popov [2.2]{P}.  This means that restriction of functions induces an isomorphism of $Y(\mathfrak a)$ onto the algebra $\textbf{k}[\eta +V]$ of regular functions on $\eta +V$.

 The geometric sense given in \cite [7.3]{J7}.  This means that $\textbf{A}(\eta+V)$ is dense in $\mathfrak a^*$ and that every orbit in $\textbf{A}(\eta +V)$ cuts $\eta +V$ at exactly one point and furthermore transversally.

 (In \cite [7.3]{J7}, a Weierstrass section was called a slice; but we have now adopted the terminology of the Russian school (cf \cite[7.1]{J7}, \cite [2.2]{P}).)

 Suppose that $Sy(\mathfrak a)=Y(\mathfrak a)$. Then the algebraic sense implies the geometric sense \cite [12.11.2]{FJ2}.  The converse holds if $(\eta+V)\setminus (\eta+V)_{reg}$ has codimension $\geq 2$.

 We say $\mathfrak a$ admits no proper semi-invariants if $Sy(\mathfrak a)=Y(\mathfrak a)$.  We say that $\mathfrak a$ is regular if it admits no proper semi-invariants and $Y(\mathfrak a)$ is polynomial.  As noted above many truncated biparabolics and centralizers are regular.  Together these give a hugely varied number of examples.

\

 \textbf{In the remainder of this section we assume that $\mathfrak a$ is regular and admits an adapted pair $(h,\eta)$.}

Under the above assumption we may choose $V$ as an $h$-stable complement to $(\ad \mathfrak a)\eta$ in $\mathfrak a^*$.  Indeed by \cite [Cor. 2.3]{JS} $\eta+V$ is a Weierstrass section in the algebraic sense and furthermore through \cite [Prop. 7.8]{J7} all its elements are regular.

\subsection{}\label{2.2}

Let  $V$ be an $\ad h$ stable complement to $(\ad \mathfrak a)\eta$ in $\mathfrak a^*$.  The eigenvalues of $\ad h$ on $V$ are called the exponents $m_i:i\in I(\mathfrak a)$, of the adapted pair.

Let $d_i: i \in I(\mathfrak a)$ be the degrees of the homogeneous generators of $Y(\mathfrak a)$.

By \cite [Cor. 2.3]{JS} one has $m_i=d_i-1, \forall i\in I(\mathfrak a)$ and moreover $\eta+V$ is a Weierstrass section for the action $\textbf{A}$ on $\mathfrak a^*$.

Through the non-degenerate pairing $V\times\mathfrak a^\eta \rightarrow \textbf{k}$ (cf \cite [2.1]{JS}) it follows that eigenvalues of $\ad h$ on $\mathfrak a^\eta$ are the $-m_i:i \in I(\mathfrak a)$.  In particular $\mathfrak a^\eta \subset \mathfrak a_\mathbb Z$.

\begin {lemma}  The zero $\ad h$ eigensubspace $\mathfrak a^\eta_0$ of $\mathfrak a^\eta$ coincides with the centre $\mathfrak z$ of $\mathfrak a$.
\end {lemma}

\begin {proof} Obviously $\mathfrak z \subset \mathfrak a^\eta_0$.  On the other hand by the above $\dim \mathfrak a^\eta_0$ is just the dimension of the space of linear invariants in $Y(\mathfrak a)$ which itself is just $\mathfrak z$.
\end {proof}

\subsection{}\label{2.3}

  Following the convention in \ref {1.1} we let $\textbf{A}_\mathbb Z$ denotes the adjoint group of $\mathfrak a_\mathbb Z$.  Since the eigenvalues of $\ad h$ on $\eta$ and on $V$ are integer, we can identify $\eta+V$ with a linear subvariety of $\mathfrak a_\mathbb Z^*$.  Recall that $\eta+V$ is a Weierstrass section for the action $\textbf{A}$ on $\mathfrak a^*$ and that $\text {index} \ \mathfrak a = \dim V$.

\begin {prop}

\

(i)  $\eta +V$ is a Weierstrass section in the geometric sense for the action of $\textbf{A}_\mathbb Z$ on $\mathfrak a^*_\mathbb Z$.

\

(ii) $\text {index} \ \mathfrak a_\mathbb Z = \dim V$.

\

(iii)  $\eta+V \subset (\mathfrak a_\mathbb Z^*)_{reg}$.

\end {prop}

\begin {proof} Since $\mathfrak a^\eta \subset \mathfrak a_\mathbb Z$ it follows that $\mathfrak a^\eta \subset \mathfrak a_\mathbb Z^\eta$.  Conversely $\ad h$ has integer eigenvalues on $(\ad \mathfrak a_\mathbb Z) \eta$ and so the latter must vanish on any $\ad h$ eigenvector of $\mathfrak a$ which does not have an integer eigenvalue.  Thus if $x \in \mathfrak a_\mathbb Z$ satisfies $(\ad x)\eta=0$ on $\im(\mathfrak a^* \rightarrow \mathfrak a^*_\mathbb Z)$, then $(\ad x)\eta=0$ on $\mathfrak a^*$.  Hence the opposite inclusion.   In particular $\dim \mathfrak a^\eta_\mathbb Z=\dim \mathfrak a^\eta=\dim V$.

On the other hand the eigenvalue of $\ad h$ on $\eta$ equals $-1$ whilst those on $V$ are the $m_i$ and in particular non-negative.  Then by a standard deformation argument (cf \cite [7.8]{J7}), it follows that $\dim \mathfrak a_\mathbb Z^\xi \leq \dim \mathfrak a_\mathbb Z^\eta$, for all $\xi \in \eta+V$.

The hypothesis of transversality means that $T_{\xi,\textbf{A}\xi} \cap T_{\xi, \eta+V}=0$ for all $\xi \in \eta+V$.  Since $T_{\xi,\textbf{A}_\mathbb Z\xi}\subset T_{\xi,\textbf{A}\xi}$, it follows that $\textbf{A}_\mathbb Z \xi$ cuts $\eta +V$ transversally at each point $\xi \in \eta+V$.  Again $\textbf{A}_\mathbb Z \subset \textbf{A}$ and so every $\textbf{A}_\mathbb Z$ orbit through $\eta+V$ meets the latter at exactly one point.

In particular we have a fibration $\textbf{A}_\mathbb Z(\eta+V) \rightarrow\eta+V$ with base $\eta+V$ of dimension $\dim V$ in which the fibre over $\xi:=\eta +v:v \in V$ is just the orbit $\textbf{A}_\mathbb Z\xi$ and has dimension $\dim \mathfrak a^*_\mathbb Z - \dim \mathfrak a^\xi_\mathbb Z\geq \dim \mathfrak a^*_\mathbb Z - \dim V$. Then by \cite [Thm. 7, Sect. 3, Chap. I]{Sh} one obtains $\dim \textbf{A}_\mathbb Z(\eta+V) \geq \dim \mathfrak a^*_\mathbb Z$. Thus equality holds and $\textbf{A}_\mathbb Z(\eta+V)$ is dense in $\mathfrak a_\mathbb Z^*$.


This proves (i).

Since $(\mathfrak a_\mathbb Z^*)_{reg}$ is dense in $\mathfrak a_\mathbb Z^*$, the generic fibre of the above fibration has dimension $\dim \mathfrak a^*_\mathbb Z - \text{index} \ \mathfrak a_\mathbb Z$, forcing $\text{index} \ \mathfrak a_\mathbb Z = \dim V$. Hence (ii).

  Finally $\dim \mathfrak a_\mathbb Z^\xi \leq \dim \mathfrak a_\mathbb Z^\eta =\dim V = \text{index} \ \mathfrak a_\mathbb Z \leq \dim \mathfrak a_\mathbb Z^\xi$. Hence (iii).

 \end {proof}

 \subsection{}\label{2.4}

Retain the above hypotheses.  Let $\psi:S(\mathfrak a)\rightarrow \textbf{k}[\eta+V]$ be restriction of functions.

\begin{cor}

\

(i)  $\textbf{A}_\mathbb Z(\eta+V)$ is open dense in $\mathfrak a^*_\mathbb Z$.

\

(ii) Let $q \in Sy(\mathfrak a_\mathbb Z)$ be a non-zero semi-invariant.   Then $\psi(q)\neq 0$.

\

(iii) $\psi$ restricts to an embedding of $Y(\mathfrak a_\mathbb Z)$ into $\textbf{k}[\eta+V]$.

\end{cor}

\begin {proof} Recall that a Weierstrass section in the geometric sense is just a slice in the language of \cite [7.3]{J7}.  Then (i) follows from \ref {2.3}(ii) and the third paragraph of \cite [7.3]{J7}. (ii) follows from (i) and (iii) from (ii).
\end {proof}

\textbf{Remarks}. In general it is false that $\psi(q)\neq 0$, for all $q \in Sy(\mathfrak a)$.  Actually by \cite [Prop. 12.3]{J7} restriction of functions induces an isomorphism $\textbf{k}[\textbf{A}_\mathbb Z(\eta +V)]^{\textbf{A}_\mathbb Z} \iso \textbf{k}[\eta+V]$; but this will not be needed.  Again we shall not need to know (the deeper fact) that $\mathfrak a^*_\mathbb Z \setminus \textbf{A}_\mathbb Z(\eta+V)$ is a \textit{closed} subvariety.

 \subsection{}\label{2.5}

 Let $\mathfrak c$ be the subspace of $\mathfrak a$ spanned by the $\ad h$ eigenvectors having eigenvalues in $\textbf{k}\setminus \mathbb Z$.

 \begin {lemma}  $\varphi(x,y):=\eta[x,y], \forall x,y \in \mathfrak c$ defines a non-degenerate anti-symmetric bilinear form on $\mathfrak c$.
 \end {lemma}

 \begin {proof}  Clearly $\eta[x,y]=0, \forall x \in \mathfrak c, y \in \mathfrak a_\mathbb Z$.  Thus if $x \in \Ker \varphi$, then $(\ad x)\eta =0$.  This contradicts the fact that $\mathfrak a^\eta \subset \mathfrak a_\mathbb Z$.
 \end {proof}

 \subsection{}\label{2.6}

  By the PBW theorem we can write $S(\mathfrak a)=S(\mathfrak a_\mathbb Z )\oplus \mathfrak c S(\mathfrak a)$.  Let $\mathscr P$ denote the projection onto the first factor.  It is an algebra map and may be viewed as restriction of functions to $\mathfrak a^*_\mathbb Z$.

Choose an $\ad h$ stable complement $V'$ to $\textbf{k}\eta$ in $(\ad \mathfrak a_\mathbb Z)\eta$ and choose $y\in S(\mathfrak a_\mathbb Z)$ vanishing on $V+V'$ and taking the value $1$ on $\eta$. One may remark that $(\ad h)y =y$.

 \begin {lemma} $\mathscr P:a \mapsto \mathscr P(a)$ is an $\mathfrak a_\mathbb Z$ module map and an algebra isomorphism of  $Y(\mathfrak a)$ onto $Y(\mathfrak a_\mathbb Z)$. Moreover $\psi$ induces an algebra isomorphism of $Y(\mathfrak a_\mathbb Z)$ onto $\textbf{k}[\eta +V]$.
 \end {lemma}

 \begin {proof}  It is clear that $[\mathfrak a_\mathbb Z,\mathfrak c] \subset \mathfrak c$. Thus $\mathscr P$ is an $\mathfrak a_\mathbb Z$ module map.  Again $\mathfrak c \subset V'$.  Thus the composed map $Y(\mathfrak a) \rightarrow Y(\mathfrak a_\mathbb Z) \rightarrow \textbf{k}[\eta +V]$ is defined.  Moreover it is just the restriction map of invariant functions on $\mathfrak a^*$ to $\eta+V$.  Since by hypothesis $\eta+V$ is a Weierstrass section with respect to the algebra $Y(\mathfrak a)$ of invariant functions on $\mathfrak a^*$ - see \ref {2.1} - this composed map is an isomorphism.   In particular the first map $Y(\mathfrak a_\mathbb Z) \rightarrow \textbf{k}[\eta +V]$ is surjective, whilst it is injective by Corollary \ref {2.4}(iii).  Hence the assertions.
 \end {proof}

  \subsection{}\label{2.7}

  The aim of the remainder of Section \ref {2}, is to show that $Sy(\mathfrak a_\mathbb Z)$ is a polynomial algebra.  The main idea is to combine the polynomiality of $Y(\mathfrak a_\mathbb Z)$ implied by Lemma \ref {2.6} with polynomiality enforced by torus action as in the case of Frobenius Lie algebras.

  Recall that $Sy(\mathfrak a_\mathbb Z)$ is a unique factorization domain whose invertible elements are scalars.

  Now let $p \in  S(\mathfrak a_\mathbb Z)$ be a non-zero semi-invariant such that $\psi(p)$ is a scalar (necessarily non-zero by Corollary \ref {2.4}(ii)).  Then every irreducible factor of $p$ has this property.  Let $\{q_j\}_{j \in J}$ be the set of all the distinct irreducible factors so obtained, that is $\{q_j\}_{j \in J}$ is the set of all the irreducible non-scalar semi-invariants such that $\psi(q_j)$ is a non-zero scalar.  For all $j \in J$ let $\Lambda_j$ be the $\mathfrak h$ weight of $q_j$.

   Recall the Cartan subalgebra $\mathfrak h$ of the reductive part of $\mathfrak a$ and let $\mathfrak h_\mathbb Z$ be the $\mathbb Z$ linear span of the basis $\{h_k\}_{k \in K}$ defined in \ref {1.6}.  The matrix with entries $\{h_k(\Lambda_j\}_{j \in J, k \in K}$ has integer entries. The $\Lambda_j:j \in J$ are $\textbf{k}$-linearly independent if and only this matrix has rank $<|J|$ and in this case they are $\mathbb Z$ linearly dependent.

%

  \begin {lemma} The $\Lambda_j:j \in J$ are linearly independent.  In particular they freely generate a semigroup $\Lambda^+ \subset \mathfrak h^*$.   Moreover the $\ad h$ eigenvalue of $q_j$ equals $\deg q_j$ and so in particular lies in $\mathbb N^+$.
  \end {lemma}

  \begin {proof}  Recall the comments above.  For the first assertion it is enough to show that the $\Lambda_j:j \in J$ cannot be $\mathbb Z$ linearly dependent.  Suppose by way of contradiction that there exist $m_j\in \mathbb Z$ such that $\sum_{j \in J}m_j\Lambda_j=0$. Set $J_\pm=\{j \in J|m_j \in \pm\mathbb N\}$ and $q_\pm:= \prod_{j \in J_\pm} q_j^{\pm m_j}$.   The $\psi(q_\pm)$ are non-zero scalars and so there exists a non-zero scalar $c$ such that $\psi(q_+ -cq_-)=0$. Yet $q_\pm$ have the same weight, so $(q_+ -cq_-)$ is again a weight vector and then $q_+=cq_-$, by Corollary \ref {2.4}(ii).  This contradicts unique factorization, hence the first assertion.

  The last assertion follows from the fact that $\eta$ is an $\ad h$ eigenvector of eigenvalue $-1$, whilst $\eta(q_j)$ is a non-zero scalar.
  \end {proof}

   \subsection{}\label{2.8}

Choose a basis $\{x_i\}_{i\in I}$ of $\mathfrak a^\eta$ consisting of $\ad h$ eigenvectors.

Recall \ref {2.2} and fix $i \in I$. By Lemma \ref {2.6}, there exists a unique element $p_i \in Y(\mathfrak a_\mathbb Z)$ such that $\psi(p_i) = x_i$.
The unique up to scalars irreducible factors of $p_i$ are again semi-invariants and so their images under $\psi$ are all non-zero.  Thus there is a unique irreducible factor $\hat{p_i}$ of $p_i$ with image $x_i$, the remaining factors having non-zero scalar images.

\begin {lemma}   The $\{\hat{p_i},q_j\}_{i \in I, j\in J}$ are algebraically independent.
\end {lemma}

\begin {proof}

Since the $\{\hat{p_i},q_j\}_{i \in I, j\in J}$ are all weight vectors, to prove their algebraic independence it is enough to consider a sum of distinct monomials having all the same weight and to show that the vanishing of the sum implies the vanishing of each summand.  Applying $\psi$ to such a sum,  the monomials in the $q_j:j \in J$ become non-zero scalars, whilst the monomials in the $\hat{p_i}$ become  the corresponding monomials in the algebraically independent elements $x_i:i \in I$. Thus it suffices to consider just sums of monomials in the $q_j:j \in J$.  However these monomials are linearly independent since their weights are pairwise distinct through Lemma \ref {2.7}.

Hence the required assertion.

\end {proof}

\textbf{Remark}.  Since $p_i$ has weight zero, it follows that the weight of $\hat{p}_i$ lies in $-\Lambda^+$.  Again by the last part of Lemma \ref {2.7}, the $\ad h$ eigenvalue of $\hat{p}_i$ lies in $-\mathbb N$.

  \subsection{}\label{2.9}

  By Lemma \ref {2.7} the $\mathbb Z$ module $\Lambda$ generated by the weights of the semi-invariants in $S(\mathfrak a_\mathbb Z)$ is the free module with generators $\Lambda_j:j \in J$.
  Then in the notation of \ref {2.7}, there exists a subset $J'$ of $K$ of cardinality $|J|$ such that $\det h_{j'}(\Lambda_j)_{j' \in J', j \in J}$ is non-zero and hence a non-zero integer $n$.  Then one can find $h'_{j'}:j' \in J'$ lying in $\frac {1}{n}\mathfrak h_\mathbb Z$ such that for all $j' \in J', j \in J$, and in terms of the Kronecker delta one has $h'_{j'}(\Lambda_j)=\delta_{j',j}$.  Let $\mathfrak h_\Lambda$ denote their linear span. One can further find $|K\setminus J'|$ linearly independent elements in $\frac {1}{n}\mathfrak h_\mathbb Z$ vanishing on $\Lambda$ hence forming a basis for the kernel $\Ker \Lambda$ of $\lambda$ in $\mathfrak h$.  It is clear that $\mathfrak h_\Lambda$ is a complement to $\Ker \Lambda$ in $\mathfrak h$.  Again let $\mathfrak a_{\mathbb Z, \Lambda}$  be the common kernel of the $\Lambda_j:j \in J$ viewed as characters on $\mathfrak a$. Then $\mathfrak h_\Lambda$ is a complement to $\mathfrak a_{\mathbb Z,\Lambda}$ in $\mathfrak a_\mathbb Z$.

Let $\Lambda_{\mathbb Q}$ denote the $\mathbb Q$ module generated by the $\Lambda_j:j \in J$.  A weight of $S(\mathfrak a_\mathbb Z)$ lying in $(\Ker \Lambda)^\perp=\sum_{j \in J}\textbf{k}\Lambda_j$ must be integer-valued on the $h_k:k \in K$ and hence it must \emph{}lie in $\Lambda_{\mathbb Q}$.

%


Let $P$ be an irreducible polynomial in the $x_i:i \in I$.  Let $Q$ be the polynomial in the $p_i:i \in I$ obtained by replacing $x_i$ by $p_i$.  By construction $Q \in Y(\mathfrak a_\mathbb Z)$ and $\psi(Q)=P$.  As in \ref {2.8} there exists unique irreducible factor $\hat{Q} \in Sy(\mathfrak a_\mathbb Z)$ of $Q$ such that $\psi (\hat{Q})=P$, the remaining factors having non-zero scalar images, in particular the weight of $\hat{Q}$ lies in $-\Lambda^+ \subset \Lambda$.

By definition of the $\hat{p_i}:i \in I$, we can also write $Q$ as a sum of monomials in the $\hat{p_i}:i \in I$ with each coefficients being a semi-invariant whose image under $\psi$ are non-zero scalars.   In this it is clear that $Q/\hat{Q}$ is just the common divisor of these coefficients.

\begin {lemma}  Let $q_\lambda$ be a non-scalar irreducible semi-invariant in $S(\mathfrak a_\mathbb Z)$ with weight $\lambda \in \Lambda_{\mathbb Q}$.  Then either $q_\lambda \in \textbf{k} q_j:j \in J$ or $q_\lambda$ is some $\hat{Q}$.  In particular $\lambda \in \Lambda$.  Moreover either $\psi(q_\lambda)$ is a non-zero scalar or an irreducible polynomial in the $x_i:i \in I$.

\end {lemma}

\begin {proof}  Choose $n \in \mathbb N^+$ such that $n\lambda \in \Lambda$ and set $q=q^n_\lambda \in Sy(\mathfrak a_\mathbb Z)$.  For some finite index set $L$ we can write $\psi(q)$ as a product $P'$ of irreducible polynomials $P_\ell: \ell \in L$ in the $x_i:i \in I$. For all $\ell \in L$, let $\hat{Q}_\ell \in Sy(\mathfrak a_\mathbb Z)$ be the unique irreducible factor of $q$ such that $\psi(\hat{Q}_\ell)=P_\ell$ and let $\hat{Q}'\in Sy(\mathfrak a_\mathbb Z)$ denote their product. Then $\psi(q)=\psi(\hat {Q}')$.

By the hypothesis and the choice of $n$, the weight of semi-invariant $q$ can be written as $\sum_{j \in J} m_j\Lambda_j:m_j \in \mathbb Z$.  Set $q'=\prod_{j\in J}q_j^{m_j}$.  The weight of $\hat{Q}'$ admits a similar description and we let $q''$ denote the corresponding product.  Then by construction $qq''$ and $\hat{Q}'q'$ have the same weight and the same image under $\psi$.  Hence by Corollary \ref {1.4}(ii), $qq''=\hat{Q}'q'$, up to a non-zero scalar.  Clearing denominators and common factors in $q',q''$ we can assume that $q',q'' \in Sy(\mathfrak a_\mathbb Z)$ and have no common factors.  Then $\hat{Q}'q'$ cannot be scalar by the assumption that $q_\lambda$ is not scalar.

By unique factorization in $Sy(\mathfrak a_\mathbb Z)$, every irreducible factor in $\hat{Q}'q'$ must divide $q$ and hence must divide $q_\lambda$. Since $q_\lambda$ was assumed irreducible, it is proportional to some $q_j$ or to some $\hat{Q}_\ell$.  The assertions of the lemma result.
\end {proof}

\subsection{}\label{2.10}

We would like to show that $Sy(\mathfrak a_\mathbb Z)$ is generated by the irreducible elements $\{\hat{p_i},q_j\}_{i \in I, j\in J}$ described in Lemma \ref {2.8}.  However this might be false, the trouble being that $\Lambda$ may not exhaust the set of weights of $Sy(\mathfrak a_\mathbb Z)$.

To proceed further recall the notation of \ref {2.9}
The argument of Borho-Chevalley shows that $Y(\mathfrak a_{\mathbb Z, \Lambda})=Sy(\mathfrak a_{\mathbb Z, \Lambda})^{\Ker \Lambda}=Sy(\mathfrak a_{\mathbb Z})^{\Ker \Lambda}$.  Again it is immediate from Lemma \ref {2.9} that

\

$(F)$ \ $Y(\mathfrak a_{\mathbb Z,\Lambda})$ is the polynomial algebra generated by the set $\{\hat{p_i},q_j\}_{i \in I, j\in J}$.

\

 Of course
 $$Y(\mathfrak a_\mathbb Z) \subset Y(\mathfrak a_{\mathbb Z, \Lambda}). \eqno{(*)}$$
 
 Again (\ref {2.7}, \ref {2.8}) the weights of $Y(\mathfrak a_{\mathbb Z,\Lambda}$ contain $\Lambda^+$ and generate $\Lambda$ as an additive group.

 One may further observe that the number of generators of $Y(\mathfrak a_{\mathbb Z,\Lambda})$ is just $\rk \Lambda + \ell (\mathfrak a_\mathbb Z) = \ell(\mathfrak a_{\mathbb Z, \Lambda})$, where the last equality obtains from \ref {1.2}$(*)$.  In particular the transcendence degree of Fract $Y(\mathfrak a_{\mathbb Z, \Lambda})$ equals $\ell(\mathfrak a_{\mathbb Z, \Lambda})$.

 On the other hand the transcendence degree of $C(\mathfrak a_{\mathbb Z, \Lambda})$ is again $\ell(\mathfrak a_{\mathbb Z, \Lambda})$.  Hence  $C(\mathfrak a_{\mathbb Z, \Lambda})$ is algebraic over Fract $Y(\mathfrak a_{\mathbb Z, \Lambda})$.  One can then ask if these fields coincide.  
 
%


\begin {lemma} Suppose $\xi \in C(\mathfrak a_{\mathbb Z, \Lambda})$ is an $\mathfrak h$ eigenvector of weight $\lambda \in \Lambda$.  Then $\xi \in \text {Fract} \ Y(\mathfrak a_{\mathbb Z, \Lambda})$.
\end {lemma}

\begin {proof}  To simplify notation we set $\mathfrak d=\mathfrak a_{\mathbb Z, \Lambda}$.  


 We may write $\xi = a_1^{-1}a_2$ with $a_1,a_2 \in S(\mathfrak d)$ and coprime.  Since $\mathfrak d$ is an ideal in $\mathfrak a_\mathbb Z$, the Chevalley-Dixmier argument shows that $a_1,a_2$ are semi-invariants with respect to the adjoint action of $\mathfrak a_\mathbb Z$.


By definition of $\Lambda$ and of $\mathfrak d$, there exist non-zero semi-invariants $q_1,q_2 \in Y(\mathfrak d)$ such that $\zeta:=q_1^{-1}q_2$ has weight $-\lambda$. Since $\zeta \in F_\mathfrak h(\text {Fract} \ Y(\mathfrak d))$, we may replace $\xi$ by its product with $\zeta$ which then has zero weight.

 
 Thus by Corollary \ref {2.4}(ii) $\psi(a_1)$ and $\psi(a_2)$ are both non-zero and of course elements of $\textbf{k}[\eta+V]$.  By surjectivity in Lemma \ref {2.6}, there exist non-zero elements $p_1,p_2 \in Y(\mathfrak a_\mathbb Z)\subset Y(\mathfrak d)$, such that $\psi(p_1)^{-1}\psi(p_2)=\psi(a_1)^{-1}\psi(a_2)$, with both sides of zero $\mathfrak h$ weight.  We conclude that $a_1p_2-a_2p_1$ is a semi-invariant of $S(\mathfrak d)$ mapped to zero under $\psi$, hence is zero by Corollary \ref {2.4}(ii).  The conclusion of the lemma results.

\end {proof}

\subsection{}\label{2.11}

%

For a suitable index set $\mathscr J$, the set $\{\gamma_j \in \mathfrak h^*\}_{j \in \mathscr J}$
of $\mathfrak h$ weights of the irreducible semi-invariants of $S(\mathfrak a_\mathbb Z)$ generate an additive subgroup $\Gamma \subset \mathfrak h^*$ containing $\Lambda$ and all the weights of $Sy(\mathfrak a_\mathbb Z)$.  It is clear that we may take $\mathscr J \supset J$ and that $\gamma_j=\lambda_j$, for all $j \in J$.

\begin {lemma} $\Gamma/\Lambda$ is a free additive group.  Moreover $\Gamma$ is the additive subgroup of $\mathfrak h^*$ freely generated by the $\{\gamma_j\}_{j \in \mathscr J}$.
\end {lemma}

\begin {proof} If $\Gamma/\Lambda$ were not free, there would exist a finite subset $F \in \mathscr J \setminus J$, non-zero integers $m_j:j \in F$ and $\gamma_j:j \in F$ such that
$$\sum_{j \in F}m_j\gamma_j \in \Lambda. \eqno {(*)}$$.

 As $j$ runs over $F$, by definition of $\gamma_j$, there exist pairwise distinct irreducible semi-invariants $q_j \in Sy(\mathfrak a_{\mathbb Z, \Lambda})$ not lying in $Y(\mathfrak a_{\mathbb Z, \Lambda})$ of weight $\gamma_j$. Then $\prod_{j \in F}q_j^{m_j} \in C(\mathfrak a_{\mathbb Z, \Lambda})$ which by $(*)$ is an $\mathfrak h$ weight vector of weight $\lambda \in \Lambda$.  Through Lemma \ref {2.10}, there exist a finite set $F' \subset J$, integers $n_k:k \in F'$ and pairwise distinct irreducible elements $p_k \in Y(\mathfrak a_{\mathbb Z, \Lambda}): k \in F'$ such that
$$\prod_{j \in F}q_j^{m_j} = \prod_{k \in F'}p_k^{n_k}.$$

Yet all the irreducible elements in this expression are pairwise distinct so this contradicts unique factorization. Hence the first assertion.  Moreover $\Gamma/\Lambda$ is freely generated by the $\{\gamma_j\}_{j \in \mathscr J \setminus J}$, whilst by Lemma \ref {2.8}, $\Lambda$ is freely generated by the $\lambda_j=\gamma_j:j \in J$.  Hence the second assertion.



\end {proof}

\subsection{}\label{2.12}

Let $\mathfrak a_{\mathbb Z,\Gamma}$ be the subalgebra of $\mathfrak a_\mathbb Z$ obtained as the common kernel of the set of characters $\{\gamma \in \Gamma\}$. It is the canonical truncation of $\mathfrak a_\mathbb Z$.  In particular there is a subalgebra $\mathfrak h_\Gamma$ of $\mathfrak h$ (not in general unique) such that $\mathfrak a_\mathbb Z=\mathfrak a_{\mathbb Z, \Gamma} \oplus\mathfrak h_\Gamma$ and such that the map $(h,\gamma)\rightarrow h(\gamma)$ defines a non-degenerate pairing $\mathfrak h_\Gamma \times \Gamma \rightarrow \textbf k$. Of course $\mathfrak a_{\mathbb Z,\Gamma}$ is also a subalgebra of $\mathfrak a_{\mathbb Z, \Lambda}$.

\begin {thm} $Sy(\mathfrak a_\mathbb Z)$ is the polynomial algebra generated by the $\{\hat{p}_i, q_j\}_{i \in I,j \in \mathscr J}$.
\end {thm}

\begin {proof} Every element of $Sy(\mathfrak a_\mathbb Z)$ can be written as a sum of its semi-invariants and in just one fashion. Every semi-invariant can be uniquely factored into a product of the $q_j: j \in \mathscr J \setminus J$ and an element in $Y(\mathfrak a_{\mathbb Z,\Lambda})$.  Then the assertion follows from $(F)$ of \ref {2.10}.
\end {proof}

\subsection{}\label{2.13}

Recall (\ref {2.8}) that $Y(\mathfrak a_\mathbb Z)$ is polynomial on generators $p_i:i \in I$ and recall the definition of the $\hat{p_i}:i \in I$. Let $-\delta_i$ be the $\mathfrak h$ weight of $\hat{p_i}$.  Then $\delta_i \in \Lambda^+$ and by Lemma \ref {2.8}, there is a unique product $Q_i$ of the $q_j: j \in J$ such that $p_i=\hat{p_i}Q_i, \forall i \in I$.

 Consider the subdivision of the generators of $Sy(\mathfrak a_\mathbb Z)$ into the three sets $\{\hat{p_i}\}_i\in I, \{q_j\}_{j\in J},\{q_j\}_{j\in \mathscr J\setminus J}$. The factorisation of $p_i$ entails exactly one factor in the first set which is moreover $\hat{p_i}$ and possibly several factors in the second set and none in the third set. Notice that we do not need to know a priori this subdivision to test this property.  Indeed we need only compute $Sy(\mathfrak a_\mathbb Z)^{\mathfrak h_\Gamma}$ in terms of the full set of generators  $\{\hat{p}_i, q_j\}_{i \in I,j \in \mathscr J}$.  Then the $p_i:i \in I$ are up to linear combinations the homogeneous generators of the resulting algebra $Y(\mathfrak a_\mathbb Z)$ which in addition must be polynomial.

\subsection{}\label{2.14}

The rather special factorisation property described in \ref {2.13} can in principle be used to show that a contraction is reached unless $\mathfrak a= \mathfrak a_\mathbb Z$, that is to say the adapted pair $(h,\eta)$ satisfies the integrality property.  For it to be useful we need to be able to describe rather explicitly the weights of the generators of $Y(\mathfrak a_\mathbb Z)$ and of $Y(\mathfrak a_{\mathbb Z, \Gamma})$.  This is possible if $\mathfrak a$ is a truncated biparabolic subalgebra of a simple Lie algebra in most cases, though even in the case of a parabolic subalgebra in type $C$ the combinatorics is rather formidable (see Section 8).

We may conclude that the question of integrality of an adapted pair $(h,\eta)$ of a regular Lie algebra $\mathfrak a$ is a rather delicate one made difficult by the fact that it is seemingly very hard to describe all adapted pairs even for a a truncated biparabolic in type $A$.

\section{Equivalence Classes of Adapted Pairs}\label {3}

 Assume that $\mathfrak a$ is regular and admits an adapted pair $(h,\eta)$ throughout this section.

\subsection{}\label{3.4}

Let us first extend slightly the analysis in \cite [Sect. 9]{FJ2}. Let $\mathfrak z$ be the centre of $\mathfrak a$. As before let $\mathfrak h$ be a Cartan subalgebra in the reductive part $\mathfrak r$ of $\mathfrak a$.

Recall the notation of \ref {1.4} and set $\mathfrak a_i^\eta= \mathfrak a_i\cap \mathfrak a^\eta$.

Recall that we assuming $\mathfrak a$ is regular. Let $m \in \mathbb N$ be the largest exponent of the $\mathfrak a$ and set $M:=\{0,1,2,\ldots,m\}$. Then one has $\mathfrak a_{-i}^\eta =0$ unless $i \in M$.  In particular $\mathfrak n: =\sum_{i\neq 0}\mathfrak a_i^\eta$ is a unipotent Lie subalgebra of $\mathfrak a$, that is to say a finite dimensional Lie subalgebra such that $\ad_\mathfrak a x$ is a nilpotent derivation of $\mathfrak a$ for all $x \in \mathfrak n$. Let $\bf{N}$ denote the connected nilpotent algebraic subgroup of $\bf{A}$ with Lie algebra $\mathfrak n$.

\begin {lemma} For all $a \in \mathfrak n$ there exists $n \in \bf{N}$ such that $nh-h=a$.
\end {lemma}

\begin {proof}  The proof follows an inductive argument of Kostant given in \cite [3.6]{K1}.  The details are briefly sketched.  Assume for $k \in \mathbb N$ we have found an element $w_k \in \sum_{j=1}^k \mathfrak n_{-j}$ such that $(\exp w_k)h-h=a-\sum_{i=k+1}^m b_{-i}$, for some $b_{-i} \in \mathfrak a_{-i}$.  For $k=0$, we take $w_k=0$ and the sum to be $a$.  To pass to a subsequent step we replace $w_k$ by $w_k+\frac{1}{k+1}b_{-(k+1)}$.  As noted in \cite [3.6]{K1} and as easily checked, this process eliminates $b_{-(k+1)}$ and can at most change the remaining $b_{-i}: i >k+1$.

\end {proof}

\textbf{Remark}.  Actually since $\eta$ is regular $\mathfrak a^\eta$ is commutative \cite [1.11.7]{D}.  This can be used to further simplify the argument.

\subsection{}\label{3.5}

Since $\mathfrak a$ is algebraic and we have assumed $\ad_\mathfrak ah$ to be a semisimple derivation, we may assume that $h \in \mathfrak h$, without loss of generality.

\

 Recall that by \ref {2.2} one has $\mathfrak z=\mathfrak a^\eta_0$. Thus $\mathfrak a^\eta$ is also a unipotent (and commutative) Lie algebra subalgebra of $\mathfrak a$.

 \

Again $\mathfrak a$ splits as a direct product of $\mathfrak z \cap \mathfrak h$ and an ideal. Thus we may assume \textit{and do assume} that $\mathfrak z \cap \mathfrak h =0$. Under this assumption, $\mathfrak z$ belongs to the nilradical $\mathfrak m$ of $\mathfrak a$.  We cannot eliminate $\mathfrak z$ entirely since it may not be a direct summand of $\mathfrak a$ as an $\ad \mathfrak a$ module.


\begin {lemma}  Suppose $h' \in \mathfrak a$ satisfies $(\ad h')\eta=-\eta$ and is $\ad$-semisimple.  Then there exists $n \in \bf{N}$ such that $h'-nh \in \mathfrak z$.
\end {lemma}

\begin {proof}   Clearly $h'-h \in \mathfrak a^\eta$.  Thus we may write $h'-h=a_0+a$, with $a_0 \in \mathfrak z, a \in \mathfrak n$.  By Lemma \ref {3.4}, there exists $n \in \textbf{N}$ such that $nh-h=a$.  Then $h'-nh=a_0$, as required.


\end {proof}

\textbf{Remark}.  In the case when $\mathfrak a$ is a truncated biparabolic subalgebra of a semisimple Lie algebra $\mathfrak g$, we could do better \cite [9.8]{FJ2}.  First we could assume that $\mathfrak h$ was a Cartan subalgebra for $\mathfrak g$.  Secondly we could show that $\mathfrak a^\eta_0$ is spanned by non-zero root vectors (nevertheless of course commuting with $h$).  Thus the condition $\mathfrak h \cap \mathfrak z=0$ is automatic. Again $h',nh$ are ad-semisimple and commute (since they differ by an element of $\mathfrak z$), hence $h'-nh$ is ad-semsimple as a derivation of $\mathfrak g$.  On the other hand if $a_0$ is non-zero it defines a non-zero nilpotent derivation of $\mathfrak g$.  Thus we were able to conclude that $h'=nh$.

\subsection{}\label{3.6}

We say that adapted pairs $h,\eta, h',\eta'$ are equivalent if there exists $a \in \bf{A}$ such that $ah-h'\in \mathfrak z, a\eta=\eta'$.  Taking account of \cite [Prop. 9.6]{FJ2} in which equivalence is defined by only requiring that $a\eta=\eta'$, we obtain (as in \cite [Prop. 9.8]{FJ2}, where only truncated biparabolics are considered) the following Corollary.

Recall that we are assuming $\mathfrak a$ to be regular.

\begin {cor}   The map $(h,\eta)\rightarrow \eta$ induces a bijection of the  set of equivalence classes of adapted pairs for $\mathfrak a$ onto the irreducible components of $\mathscr N(\mathfrak a)$ of codimension $\ell(\mathfrak a)$.
\end {cor}

\textbf{Remark}.  In the biparabolic case we could require $ah-h'=0$ in the definition of equivalence.  Yet the weaker condition $ah-h' \in \mathfrak z$ is still quite satisfactory since then $ah$ and $h'$ induce the same derivations of $\mathfrak a$.

\section{Rationality and Further Consequences}\label{4}

Let $\mathfrak a$ be a regular Lie algebra with an adapted pair $(h,\eta)$.  We first show that the eigenvalues of $\ad_\mathfrak a h$ are rational. Again by Lemma \ref {3.5}, we may conclude that $\ad h$ is determined up to conjugation by $\eta$.  Here we show the converse, that is up to conjugation $\eta$ is determined by $\ad h$.  This will follow from Proposition \ref {4.4}(iv).

Recall \ref {2.1} that we may write $\mathfrak a = \mathfrak r \oplus \mathfrak m$ with $\mathfrak r$ reductive and that $h$ can be assumed to belong to a Cartan subalgebra $\mathfrak h$ of $\mathfrak r$.  Moreover the centre $\mathfrak z$ of $\mathfrak a$ coincides with $\mathfrak a^\eta_0$ and can be assumed to have null intersection with $\mathfrak h$.

\subsection{}\label{4.1}

For all $\alpha \in \mathfrak h^*$, set $\mathfrak a_\alpha  = \{a \in \mathfrak a| [(\ad h)a=\alpha(h)a, \forall h \in \mathfrak h\}, \mathfrak a^*_\alpha  = \{\xi \in \mathfrak a^*| (\ad h)\xi=\alpha(h)\xi, \forall h \in \mathfrak h\}$.  To avoid confusion with the notation introduced in \ref {2.1} we denote the zero root by $\textbf{0}$. Set $\mathfrak m_{\textbf{0}}=\mathfrak a_{\textbf{0}}\cap \mathfrak m$.  Clearly $\mathfrak a_{\textbf{0}}= \mathfrak m_{\textbf{0}} \oplus \mathfrak h$. Set $\Delta := \{\alpha \in \mathfrak h^* \setminus \{\textbf{0}\}| \mathfrak a_\alpha \neq 0 \}$. By construction
$$\mathfrak a =\mathfrak a_{\textbf{0}} \oplus \oplus_{\alpha \in \Delta}\mathfrak a_\alpha.$$

Through the $\mathfrak a$ invariant non-degenerate pairing $\mathfrak a^* \times \mathfrak a \rightarrow \bf{k}$ we obtain
$$\mathfrak a^* =\mathfrak a^*_{\textbf{0}} \oplus \oplus_{\alpha \in \Delta}\mathfrak a^*_{-\alpha}.$$

Set $\Delta_1:=\{\alpha \in \Delta|h(\alpha)=1\}$.

 By definition of an adapted pair, there exist $S \subset \Delta_1$ and non-zero elements $\xi_{-\alpha} \in \mathfrak a^*_{-\alpha}:\alpha \in S$ such that $\eta= \sum_{\alpha \in S} \xi_{-\alpha}$.  Recall the index set $K$ defined in \ref {2.1} and the basis $\{h_i\}_{i\in K}$ of $\mathfrak h$.

 \begin {lemma} $S$ spans $\mathfrak h^*$.  In particular $|S|\geq|K|$.
 \end {lemma}

 \begin {proof} Otherwise there exists $h' \in \mathfrak h$ such that $h' \in \mathfrak a^\eta$.  This contradicts our assumption that $\mathfrak z \cap \mathfrak h=0$.
 \end {proof}

\subsection{}\label{4.2}

Recall the notation in the first part of \ref {2.1}.

\begin {cor} $\Delta \subset \mathbb QS$.  In particular the ad-semisimple element of an adapted pair has rational eigenvalues.
\end {cor}

\begin {proof}  By Lemma \ref {4.1} we may choose a subset $S_\#$ of $S$ which is a basis for $\mathfrak h^*$.  Then we may write $S_\#=\{\alpha_i\}_{i \in K}$, for some $\alpha_i \in \Delta$.  Moreover the matrix with integer entries $h_i(\alpha_j) : i, j \in K$ has a non-zero determinant whose value is some non-zero integer $d$.  Take $\alpha \in \Delta$.  By definition of $S_\#$ we may write $\alpha =\sum_{k \in K} c_k\alpha_k$, for some $c_k \in \textbf{k}$. Then $\sum_{k \in K}c_k h_j(\alpha_k)= h_j(\alpha) \in \mathbb Z$.  Thus $c_k \in d^{-1}\mathbb Z$, as required.
\end {proof}

\subsection{}\label{4.3}

 Let $h$ be the ad-semisimple element of an adapted pair.  In the notation of \ref {1.4}, we may write $\mathfrak a_\mathbb Q= \oplus_{i \in \mathbb Q}\mathfrak a_i$. Then by Corollary \ref {4.2} one has $\mathfrak a =\mathfrak a_\mathbb Q$.

 For all $c \in \mathbb Q$, set $\mathfrak a_{\geq c} = \oplus_{i\geq c}\mathfrak a_i, \mathfrak a_{> c}= \oplus_{i> c}\mathfrak a_i$, and more generally for any subset $\mathscr I \subset \mathbb Q$ set $\mathfrak a_\mathscr I = \oplus_{i \in \mathscr I}\mathfrak a_i$.  A similar definition is given for $\mathfrak a^*_{\mathscr I}$.  Observe that duality gives a non-degenerate pairing of $\mathfrak a_\mathscr I$ with $\mathfrak a^*_{-\mathscr I}$.

\subsection{}\label{4.4}

A natural question that arises from Lemma \ref {4.1} is whether we can choose $\eta$ in minimal form in the sense that $S$ is a basis for $\mathfrak h^*$.  The following result gives some preliminary information.

\begin {prop}  Set $\mathscr I = ]-1,0[$.

\

(i) $\frac {1}{2}(\dim \mathfrak a + \Index \mathfrak a) = \dim \mathfrak a^*_{\geq 0} +\frac {1}{2}\dim \mathfrak a^*_{\mathscr I}$.

\

(ii) $\frac {1}{2}(\dim \mathfrak a - \Index \mathfrak a) = \dim \mathfrak a^*_{< 0} -\frac {1}{2}\dim \mathfrak a^*_{\mathscr I}$.

\

(iii)  $(\ad \mathfrak a_{\leq 0})\eta =\mathfrak a_{\leq -1}^*$, equivalently $\bf{A}_{\leq 0}\eta$ is dense in $\mathfrak a_{\leq -1}^*$.

\

(iv)  $(\ad \mathfrak a_{0})\eta =\mathfrak a_{ -1}^*$, equivalently $\bf{A}_{0}\eta$ is dense in $\mathfrak a_{ -1}^*$.

\

(v) $\dim \mathfrak a_0= \dim \mathfrak z +\dim \mathfrak a_{-1}^*$.
\end {prop}

\begin {proof}  Define $V\subset \mathfrak a^*$ as in \ref {2.2}.  One has  $(\ad \mathfrak a)\eta \oplus V= \mathfrak a^*$ and in addition the eigenvalues of $h$ on $V$ are non-negative integers, whilst the eigenvalues of $h$ on $\mathfrak a^\eta$ are non-positive integers.  Consequently
$$\dim \mathfrak a^*_{\geq 0}=\dim \mathfrak a_{\geq 1} + \dim \mathfrak a^\eta_{\geq 1} + \dim V=\dim \mathfrak a^*_{\leq -1}  + \Index \mathfrak a.$$

Adding $\dim \mathfrak a^*_{\geq 0}$ to both sides gives (i).  (ii) follows from (i).

Again since the eigenvalues of $h$ on $\mathfrak a^\eta$ are non-positive integers we obtain
$$\dim (\ad \mathfrak a_{\leq 0})\eta = \dim \mathfrak a_{\leq 0}- \dim \mathfrak a^\eta= \dim \mathfrak a^*_{\geq 0}- \Index \mathfrak a =\frac{1}{2}(\dim \mathfrak a -\Index \mathfrak a-\dim \mathfrak a^*_{\mathscr I}),$$
where the last identity obtains from (i).

On the other hand
$$\dim \mathfrak a_{\leq -1}^* = \dim \mathfrak a_{< 0}^* - \dim \mathfrak a_\mathscr I=\frac{1}{2}(\dim \mathfrak a -\Index \mathfrak a-\dim \mathfrak a^*_{\mathscr I}),$$
where the last identity obtains from (ii).

We conclude that the inclusion $(\ad \mathfrak a_{\leq 0})\eta \subset \mathfrak a_{\leq -1}^*$ is an equality giving (iii).  Finally (iv) follows from (iii).  Since $\mathfrak a^\eta_0=\mathfrak z$ by Lemma \ref {2.2}, we obtain (v) from (iv).
\end {proof}

\textbf{Remark}.  Recall that the existence of an adapted pair for $\mathfrak a$ is just condition (H1) of \cite [3,1]{JL}.  Moreover since $\mathfrak a$ is assumed regular, condition (H2) of \cite [3.1]{JL} results from \cite [Cor. 2.3]{JS}. Nevertheless \cite [Lemma 9]{JL} is incorrect and must be replaced by (i) above.  Yet when condition (H4) of \cite {JL} also holds then by \cite [3.3,3.5]{JS} the eigenvalues of $\ad h$ on $\mathfrak a$ are all integer. Thus $\mathfrak a^*_{\mathscr I}=0$ in Proposition \ref {4.4} above.  This recovers the conclusion of  \cite [Lemma 9]{JL} when (H4) of \cite {JL} holds. Consequently the main results of \cite {JL} are unaffected by this error in \cite [Lemma 9]{JL}.

\subsection{}\label{4.5}

 Proposition \ref {4.4}(iv) implies that the semisimple element $h$ of the adapted pair $(h,\eta)$ determines $\eta$ as an element in general position in $\mathfrak a^*_{-1}$.

Again to show that $\eta$ can be chosen in minimal form it is enough to find a regular element in $\mathfrak a^*_{-1}$ which has minimal form.

Recall \ref {2.1}.  One can ask if $h$ is regular in the reductive part $\mathfrak r$ of $\mathfrak r$, equivalently that $\mathfrak r^h=\mathfrak h$.

Since $\mathfrak a$ is an algebraic Lie algebra so is $\mathfrak a_0=\mathfrak a^h$. Moreover $\mathfrak r_0:=\mathfrak r \cap \mathfrak a_0$ is the reductive part of $\mathfrak a_0$. To answer the above question it is enough to show that $h$ is regular in $\mathfrak r_0$, equivalently that $\mathfrak r_0^h=\mathfrak h$.

Recall that $\mathfrak a_0^\eta$ coincides with the centre $\mathfrak z$ of $\mathfrak a$ and so in particular is an ideal in $\mathfrak a_0$. Set $\overline{\mathfrak a}_0= \mathfrak a_0/\mathfrak z$.  Clearly $\mathfrak a^*_{-1}$ is an $\overline{\mathfrak a}_0$ module.  By Proposition \ref {4.4}(iv) it has the same dimension as $\overline{\mathfrak a}_0$ and further admits a (unique) dense orbit.  This in itself is not sufficient to imply the required conclusion, for we could take $\overline{\mathfrak a}_0$ to be a non-solvable Frobenius Lie algebra with $\mathfrak a^*_{-1}$ it co-adjoint module.

 Below we give a counterexample to the above question.

   Consider $\mathfrak a$ to be the canonical truncation of the parabolic $\mathfrak p_{\pi'}$ in $\mathfrak {sl}(n)$ defined by two blocks of sizes p,q respectively with p,q being coprime. In this case $Y(\mathfrak a)$ is polynomial on one generator of degree $\frac{1}{2}(p^2+q^2 +pq-1)$ and moreover $\mathfrak a$ admits an adapted pair $(h,\eta)$ unique up to equivalence \cite {J3}. Here we can assume that $h \in \mathfrak h$ and is the unique dominant element with respect to the Weyl group of the Levi factor (defined by $\pi'$). Then $h$ is regular in the Levi factor if and only if $h(\alpha) >0, \forall \alpha \in \pi'$.  In partial answer to the above conjecture we showed \cite [Thm. 4.8]{J3} that $h(\alpha):\alpha \in \pi'$ is large (of the order of $n^2$) for exactly two simple roots (determined by the solution to the Bezout equation equation $sq-rp=1$) whilst $h(\alpha)$ is otherwise small (of the order of $n$). Conjecture 3 of \cite [7.20]{J1.2} further suggested that these small values lie in $\{0,1\}$.

Let us give an example for which $h(\alpha)$ can take the value $0$ for some $\alpha \in \pi'$.  This is provided by taking $p=5,q=8, n=13$ in the above. The invariant in question has degree $64$. With respect to the Bourbaki order of the simple roots, the values of $h(\alpha)$ are found using \cite [Cor. 4.3]{J3}) to be $(1,61,1,1)$ in the small block and $(1,1,60,1,0,1,1)$ in the large block.

 Thus $\mathfrak r_0^h \iso \mathfrak {sl}(2)+\mathfrak h$ and in particular $h$ is not regular in the reductive part of $\mathfrak a$.

\subsection{}\label{4.6}

It can happen that $\mathfrak a$ is the canonical truncation of an algebraic Lie algebra $\hat{\mathfrak a}$.  As a consequence $Sy(\hat{\mathfrak a})=Sy(\mathfrak a)=Y(\mathfrak a)$.  Moreover we can choose a Cartan subalgebra $\hat{\mathfrak h}$ containing a Cartan subalgebra $\mathfrak h$ and indeed $\mathfrak a$ is obtained from $\hat{\mathfrak a}$ by replacing $\hat{\mathfrak h}$ by $\mathfrak h$.  In this, duality induces a non-degenerate pairing
$$\Lambda \times \hat{\mathfrak h}/\mathfrak h \rightarrow \textbf {k}, \eqno {(*)}$$
where $\Lambda$ is the set of weights of $Sy(\hat{\mathfrak a})$.

Typically $\hat{\mathfrak a}$ may be a biparabolic and in this case $Y(\hat{\mathfrak a})$ is reduced to scalars under our standing hypothesis. However we need not assume this.

Define the set $\hat{\Delta}$ of (non-zero) roots of $\hat{\mathfrak a}$ as in \ref {4.1}.  Clearly $\Delta=\hat{\Delta}|_\mathfrak h$.

Assume that $Y(\mathfrak a)$ is polynomial and that $\mathfrak a$ admits an adapted pair $(h,\eta)$.  Then Lemma \ref {4.1} translates to imply that $S|_\mathfrak h $ spans $\mathfrak h^*$.  Then by \cite [9.3]{FJ2} a complement $V$ to $(\ad \mathfrak a)\eta$ in $\mathfrak a^*$ may be chosen in the form of a sum of root subspaces defined by a subset $T \subset \hat{\Delta}$, the only difference here being that root subspaces need not be one-dimensional.  Then as in \cite [Prop. 9.4]{FJ2}, we obtain from $(*)$ and the isomorphism $Sy(\hat{\mathfrak a}) \iso \textbf{k}[\eta +V]$ (giving notably \cite [9.4$(*)$]{FJ2}) the following

\begin {lemma}  The subset $S\cup T$ of $\hat{\Delta}$ spans $\hat{\mathfrak h}^*$.  In particular $\hat{\Delta} \subset \mathbb Q(S\cup T)$.
\end {lemma}

\begin {proof}  It remains to note that the last part follows from the first part through the argument given in the proof of Corollary \ref {4.2}.
\end {proof}

\subsection{}\label{4.6}

Suppose $\mathfrak a$ is a simple Lie algebra $\mathfrak g$ with an adapted pair $(h,\eta)$.  Then $\mathfrak g_\mathbb Z$ is reductive by the argument of Lemma \ref {1.6} and so both $\mathfrak g$ and $\mathfrak g_\mathbb Z$ are unimodular.  This is obviously incompatible with the conclusion of Lemma \ref {2.5} unless $\mathfrak c=0$.  Hence this pair satisfies integrality.   On the other hand the condition $\mathfrak a^\eta_{-i}=0$ unless $i \in \mathbb N$ implies that $(h,\eta)$ is a good pair in the sense of Elashvili-Kac \cite {EK} for the regular nilpotent orbit generated by $\eta$.  Since these authors show that a good pair for the regular nilpotent orbit is equivalent to the pair extracted from an s-triple (\ref {1.1}) this gives nothing new.  However it means that we can ignore the case $\mathfrak a$ simple from now on.

\section{Frobenius Biparabolics}

\subsection{}\label{5.1}

Let $\mathfrak q_{\pi_1,\pi_2}$ be a biparabolic subalgebra of a simple Lie algebra $\mathfrak g_\pi$.  Suppose further that this algebra is Frobenius.  Then it admits a unique up to conjugation adapted pair $(h,\eta)$ and one can ask if this pair has the integrality property.  Here we show that this is true.  Considering that it is false for even a rather small Frobenius Lie (cf. \ref {1.4}), this result should not be considered self-evident.

\subsection{}\label{5.2}

Recall the notion \cite [Sect.2]{J1} of the Kostant cascade $B_\pi$ defined for $\mathfrak g_\pi$ as follows. First the unique highest root $\beta \in \Delta$ belongs to $B_\pi$.  Since $\Delta_\beta:=\{\gamma \in \Delta|(\gamma,\beta)=0\}$ is a root system, the construction is continued by decomposing $\pi_\beta:=\Delta_\beta\cap \pi$ into connected components.  This makes $B_\pi$ a maximal set of strongly orthogonal roots (that is the sum and difference of distinct elements of $B_\pi$ is not a root) with an order relation given in the obvious manner by the construction.  This set is described explicitly in \cite [Table III]{J1}.  Its cardinality is $|\pi|$ if and only if $-1$ belongs to the Weyl group $W$, more explicitly if $\mathfrak g$ has no factors of type $A_n, D_{2n+1}:n\geq 2, E_6$. However even if $-1 \in W$, it is generally false that the $\mathbb Z$ linear span of $B_\pi$ contains the set $\Delta$ of all (non-zero) roots.

To an arbitrary biparabolic $\mathfrak q_{\pi_1,\pi_2}$  we assign the set $B:=-B_{\pi_1}\sqcup B_{\pi_2}$. We noted in \cite [4.3]{J8} that $\mathfrak q_{\pi_1,\pi_2}$ is Frobenius if and only if $B$ forms a $\bf{k}$ basis for $\bf{k}\pi$.  Moreover with respect to the unique up to equivalence adapted pair $(h.\eta)$, we can choose $\eta$ in the form
$$\eta = \sum_{\beta \in B}x_{-\beta}, \eqno{(*)}$$
and hence $h\in \mathfrak h$ is uniquely determined by the condition that $h(\beta)=1, \forall \beta \in B$.

Here we should add that all this presentation of the adapted pair is an immediate consequence of a result of Tauvel and Yu \cite {TY}.

The integrality property of $(h,\eta)$ is immediate if $\mathbb ZB \supset \Delta$.  Though the latter does hold if $\pi$ is of  type $A$, it is generally false.   Thus a little more work is needed and moreover one needs the criterion for $\mathfrak q_{\pi_1,\pi_2}$ to be Frobenius described in \cite [Lemma 4.2]{J8}.

\subsection{}\label{5.3}

The first case to consider is when $\mathfrak q_{\pi_1,\pi_2}$ is a Borel subalgebra (and Frobenius).  This is when $\pi_2 =\pi, \pi_1=\phi$ (and furthermore $\pi$ is not of type $A_n, D_{2n+1}:n\geq 2, E_6$).  A priori we would not expect the integrality property to hold, however it does and the result is rather curious.  Remarkably we even find that $h(\alpha) \in \{1,0,-1\}, \forall \alpha \in \pi$.  Even more curiously $h(\alpha) \in \{1,-1\}, \forall \alpha \in \pi$ exactly when $\pi$ is of type $B_{2n+1}, D_{2n+2}:n\geq 1, E_7,E_8$ which are exactly those cases for which the Borel is Frobenius and the Kostant cascade does not pass through a simple root system of type $C_2$.  In these cases we were able to show \cite {J7} that the truncated Borel admits a Weierstrass section (in fact \textit{not} given by an adapted pair).  Finally those simple roots $\alpha$ on which $h(\alpha)=-1$ are exactly those which are intractable in the sense of \cite [5.1]{J8}.

Although all these facts can be read off from Table I below, it turns out that one can give them an intrinsic proof which results from the following two lemmas.

Assume that $\pi$ is connected.  Call a root $\beta$ long if $(\beta,\beta) \geq (\gamma,\gamma), \forall \gamma \in \Delta$ and short otherwise. Then by convention there are short roots in $\Delta$ if and only if $\Delta$ is not simply-laced, that is in types $B,C,F,G$.  Let $\beta$ denote the unique highest root in $\Delta^+:=\Delta\cap \mathbb N\pi$ and recall the notation of \ref {5.2}.  Set $\pi^\beta=\pi\setminus \pi_\beta, \Gamma_\beta = \Delta^+ \setminus (\Delta^+\cap \Delta_\beta)$.  The following is an unpublished result of Kostant with some details given in \cite [2.2]{J1}.  We give a proof for completeness.

Let $W_\pi$ denote the Weyl defined by $\pi$ (that is to say generated by the $s_\alpha: \alpha \in \pi$) and $w_\pi$ its unique longest element.  Set $i_\pi =-w_\pi$.  It may be viewed as a Dynkin diagram involution which is often trivial.

\begin {lemma}

\

(i) $(\beta,\gamma) = \frac {1}{2}(\beta, \beta), \forall \gamma \in \Gamma_\beta \setminus \{\beta\}$.  In particular $\beta$ is a long root.

\

(ii) Given $\gamma, \delta \in \Gamma_\beta$ such that $\gamma + \delta \in \Delta$, then $\gamma+ \delta=\beta$.

\

(iii) $\prod_{\beta \in B_\pi}s_\beta = w_\pi$.

\

(iv) $|\pi^\beta|\leq 2$.  Moreover $\pi^\beta$ is a single $i_\pi$ orbit.

\end {lemma}

\begin {proof} Take $\gamma \in \Gamma_\beta \setminus \{\beta\}$.  Then $(\beta,\gamma)>0$ by definition.  Thus $\gamma-\beta$ is a root but obviously $\gamma - 2\beta$ cannot be a root.  Hence (i). Clearly (ii) follows from (i).

It is clear that $s_\beta \Gamma_\beta = -\Gamma_\beta, s_\beta|_{ \Delta_\beta} = \Id|_{\Delta_\beta}$. On the other hand it is also clear that $s_{\beta'} \Gamma_\beta = \Gamma_\beta, \forall \beta' \in B_\pi \setminus \{\beta\}$.  We conclude that $\prod_{\beta \in B_\pi}s_\beta$ sends $\Delta^+$ to $\Delta^-$ and so coincides with $w_\pi$. Hence (iii).

Compute $(\beta,\beta)$ by writing $\beta=\sum_{\alpha \in \pi} n_\alpha\alpha$.   Then by (i), since $(\alpha,\beta)\geq 0$, we obtain $\sum_{\alpha \in \pi^\beta} n_\alpha =2$.  Thus $|\pi^\beta|\leq 2$.  Let us write $\pi^\beta =\{\alpha_1,\alpha_2\}$.  Then $s_\beta \alpha_1=\alpha_1-\sum_{\alpha \in \pi}n_\alpha \alpha$, has a coefficient $-1$ of $\alpha_2$.  Then by (iii) so has $w_\pi\alpha_1$.   Yet by the uniqueness of $\beta$ it follows that $\pi^\beta$ is $i_\pi$ stable, hence $i_\pi\alpha_1=\alpha_2$.  Hence (iv).

\end {proof}

\textbf{Remarks}.   The completed Dynkin diagram \cite [Planches I-IX]{B} describes $\pi^\beta$. From this $B_\pi$ may be described as an ordered set \cite [Table III]{J1}. One finds that $|\pi^\beta|=2$ exactly in type $A_n:n \geq 2$.

\textbf{N.B.}  Not all the roots of $B_\pi$ need be long since they are only long relative to the root subsystem for which they are the unique highest root.

\subsection{}\label{5.4}

Assume that $\pi$ is connected.

Clearly $\mathfrak b_\pi$ is Frobenius if and only if $|B_\pi| = |\pi|$. By Lemma \ref {5.3}(iii) this condition is equivalent to $i_\pi$ being the identity.  Moreover this is in turn equivalent to the condition that for all $\beta \in B_\pi$ there is a unique $\alpha \in \pi$ belonging to the indecomposable root system for which $\beta$ is its highest root such that $(\alpha,\beta) >0$.

Fix $h \in \mathfrak h$ such that $h(\beta)=1, \forall \beta \in B_\pi$. It is unique if these three equivalent conditions given in the paragraph above hold.

\begin {lemma}  Let $\beta_*$ denote the unique highest root and assume $\alpha \in \pi$ is the unique element belonging to $\pi^{\beta_*}$.

\

(i)  Suppose that $\alpha$ is a short root.  Then $h(\alpha)=0$.

\

(ii) Suppose $\pi$ is of type $A_1$, then $h(\alpha)=1$,

\

(iii)  If neither (i),(ii) hold then $h(\alpha)=-1$.  In this case $\varpi_\alpha =\beta_*$, whereas if (i) or (ii) holds then $2\varpi_\alpha =\beta_*$.
\end {lemma}

\begin {proof}

 Suppose $\alpha$ is a short root.  Then $\beta_*-2\alpha$ is a root, necessarily belonging to the Kostant cascade. This forces $h(\alpha)=0$.  Hence (i).  (ii) is obvious.

 Consider (iii). By Lemma \ref {5.3}(iii) one has $w_\pi\beta_*=-\beta_*$, whilst by the hypothesis of the lemma and Lemma \ref {5.3}(iv), $w_\pi\alpha=-\alpha$.  Thus if we set $w^{\beta_\star}=\prod_{\beta \in B_\pi\setminus \{\beta_*\}}s_\beta$, then $w^{\beta_*} \alpha=-s_{\beta_*}\alpha=\beta_\star-\alpha$.

 Suppose that $(\beta,\alpha)\neq 0$, for some $\beta \in B_\pi \setminus \{\beta_*\}$.  Since $\alpha$ is assumed long one has $s_\beta\alpha=\alpha - 2\frac{(\beta,\alpha)}{(\beta,\beta)}\beta= \alpha + \frac{(\alpha,\alpha)}{(\beta, \beta)}\beta$, which gives
$$w^{\beta_\star}\alpha = \alpha + \sum_{\beta \in B_\pi \setminus \{\beta_*\}|(\beta,\alpha)\neq 0}   \frac{(\alpha,\alpha)}{(\beta,\beta)} \beta.$$

Combined with our previous formula we obtain
 $$\beta_*-2\alpha=\sum_{\beta \in B_\pi \setminus \{\beta_*\}|(\beta,\alpha)\neq 0}   \frac{(\alpha,\alpha)}{(\beta,\beta)} \beta.\eqno {(*)}$$

  Again since $\alpha$ and $\beta$ are both long the scalar product of the left hand side of $(*)$ with itself equals $3(\alpha,\alpha)$, by Lemma \ref {5.3}(i). Then computing the scalar product of the right hand side of $(*)$ with itself gives
 $$3(\alpha,\alpha)=\sum_{\beta \in B_\pi \setminus \{\beta_*\}|(\beta,\alpha)\neq 0}\frac{(\alpha,\alpha)^2}{(\beta,\beta)}.\eqno {(**)}$$

 Computing $h$ on both sides of $(*)$ and using $(**)$ gives $h(\alpha)=-1$.

 Finally by Lemma \ref {5.3}(i), $\varpi_\alpha = \beta$ if and only if $\alpha$ is a long simple root different to $\beta$.  Otherwise $2\varpi_\alpha=\beta$.  Hence (iii).

 \end {proof}

 \textbf{Remarks}.  Suppose $|B_\pi| =|\pi|$.  Then this result determines $h(\alpha)$ for all $\alpha \in \pi$ and to have values in $\{-1,0,1\}$.

 One may easily check from (iii) that $2\varpi_\alpha \in \mathbb NB_\pi$ and $\varpi_\alpha \in \mathbb NB_\pi$ if and only if $h(\alpha)=-1$.  In the latter case we called  $\alpha$ intractable \cite [5.1]{J7}.

\subsection{}\label{5.5}

 In Table I below we give the values of $h(\alpha):\alpha \in \pi$ for an adapted pair $h \in \mathfrak h, \eta \in \mathfrak b^*_{reg}$, when the Borel subalgebra $\mathfrak b$ of a simple Lie algebra $\mathfrak g_\pi$ is Frobenius, computed via Lemma \ref {5.4} and \cite [Table III]{J1}.

Here and below the Bourbaki convention \cite [Planches I-IX]{B} is used to label the simple roots.

\bigskip
\begin {center}

\begin{tabular}{|l|l|l|}
\hline
Type&$h(\alpha_i)$&$h(\alpha_i)$\\

\hline
$B_{2n-1}:n \geq 2$& $(-1)^{i-1}$&\\
$B_{2n}:n \geq 2$&$(-1)^{i-1}:i \leq 2n-1$&$0:i=2n$\\
$C_n:n \geq 2$& $0:i<n$&$1:i=n$\\
$D_{2n}:n \geq 2$& $(-1)^{i-1}:i \leq 2n-2$&$ 1:i=2n-1,2n$\\
 \hline
 $E_7$&$-1:i=1,4,6$ &$1:i=2,3,5,7$\\

 $E_8$&$-1:i=1,4,6,8$ &$1:i=2,3,5,7$\\
$F_4$&$(-1)^i:i=1,2$&$0:i=3,4$\\
$G_2$& $-1:i=2$&$1:i=1$\\
\hline

\end{tabular}

\end {center}

\bigskip

\begin {center} Table I

\end {center}

\bigskip

\subsection{}\label{5.6}

Suppose $|B_\pi|< |\pi|$.   Then $h(\pi)$ is not determined. However the only difficulty arises if at some point in the Kostant cascade one reaches a system of type $A_n:n \geq 2$.  (This occurs exactly in types $A_n, D_{2n+1}:n \geq 2, E_6$.)   Nevertheless we may still apply Lemmas \ref {5.3}, \ref {5.4} to obtain the following table.

\bigskip

\begin {center}

\begin{tabular}{|l|l|l|l|}


\hline
Type&$h(\alpha_i)$&$h(\alpha_i)$&$h(\alpha_i)+h(i_\pi(\alpha_i))$\\

\hline

$D_{2n+1}:n \geq 2$& $(-1)^{i-1}:i \leq 2n-1$&&$ 0:i=2n$\\
 \hline
 $E_6$&$-1:i=4$ &  $1: i=2$ &$0:i=1,3$\\

\hline

\end{tabular}

\end {center}

\bigskip

\begin {center} Table II

\end {center}

\subsection{}\label{5.7}

For the case of an arbitrary Frobenius biparabolic, we must recall \cite [Lemma 4.2]{J7} and some background theory. We shall do this very briefly, more details may be found in \cite [3.4]{J8} and references therein.

 We remark that in general the elements of $B_\pi$ are linearly independent and invariant under $i_\pi$.   On the other hand $|B_\pi|=|\pi/<i_\pi>|$.

 \

Thus
$$\textbf{k}B_\pi= \bigoplus_{\alpha \in \pi}\textbf{k}(\alpha+i_\pi(\alpha)). \eqno{(*)}$$

 For $j=1,2$, set $i_j=i_{\pi_j}$, which is a Dynkin diagram involution of $\pi_j$. (This shorthand notation brings us into line with the notation of \cite [4.5]{J2}. However it is not appropriate for Sections \ref {7}, \ref {8} and will be confined to this section and the introduction to Section \ref {6}.) They may be extended to involutions of an overset $\tilde{\pi}$, for which it may be necessary to adjoin ``fictitious roots" (see \cite [4.5]{J2}) however the precise details will not be needed here.  Set $\pi_\cap =\pi_1\cap \pi_2, \pi_{\cup} = \pi \setminus \pi_\cap$.

\begin {lemma} $\mathfrak q_{\pi_1,\pi_2}$ is Frobenius if and only if no $<i_1i_2>$ orbit lies entirely in $\pi_\cap$ and if every $<i_1i_2>$ orbit lies in $\pi$ and meets $\pi_\cup$ at exactly one point.
\end {lemma}

\begin {proof} See \cite [Lemma 4.2]{J7}.
\end {proof}

\textbf{Remark}.  The meaning of the first condition is clear without having to know how to extend $i_1,i_2$.  The second statement means that the orbit ``starts" at fixed point of $i_1$ in $\pi_1$ (or of $i_2$ in $\pi_2$) and then these two involutions are sequentially applied until a point in $\pi_\cup$ is reached.  The latter is then deemed to be a fixed point of the involution whose action on the given point in $\pi_\cup$ is not defined.

\subsection{}\label{5.8}

Suppose $\pi$ is indecomposable of type $A_n$.  Then the Kostant cascade takes a particularly simple form.  Indeed the $i^{th}$ element $\beta_i$ is just the sum of simple roots $\alpha_i+\alpha_{i+1}+\ldots +\alpha_{n+1-i}$. Moreover $\alpha:=\alpha_i, \alpha':=\alpha_{n+1-i}$ form the two elements of $\pi$ having a positive scalar product on $\beta_i$.  Thus the condition $h(\beta_i)=1:i=1,2,\ldots, [\frac{n+1}{2}]$ immediately implies that

\begin {lemma}

$$
\begin{array}{ccc}
 h(\alpha)=h(\alpha')=1& :& \alpha=\alpha', \\
 h(\alpha)+h(\alpha')=1 & :&  (\alpha,\alpha')<0, \\
 h(\alpha)=h(\alpha')=0& :&  (\alpha,\alpha')=0. \\
\end{array}
$$
\end {lemma}

\subsection{}\label{5.9}

In the above Lemma one may also view $\alpha$ and $\alpha'$ as the end-points of an arc defined by the involution $i_\pi$ sending $\alpha$ to $\alpha'$. Thus the first line means that $h(\alpha) =1$ if $\alpha$ is a fixed point.  Of course for an arbitrary biparabolic one must replace $i_\pi$ by $i_{\pi_1}$ or $i_{\pi_2}$ as appropriate.  Again if $\pi_1$, or $\pi_2$ is not of type $A$, then one must use Tables I,II to determine the value of $h(\alpha)$.  If $\alpha$ is not a fixed point, then the conclusion of the first two sentences of this paragraph apply since the corresponding root \textit{sub}system is still of type $A$.

\subsection{}\label{5.10}

Continue to assume that $\mathfrak q_{\pi_1,\pi_2}$ is Frobenius. The above facts give rise to a very simple algorithm for computing $h(\alpha):\alpha \in \pi$.  Recall Lemma \ref {5.7} and start at a fixed point $\alpha$  of $\pi_1$ or of $\pi_2$.  Then $h(\alpha)$ is given by the Table I, II or by Lemma \ref {5.8}.  Then the value of $h$ on the subsequent elements in $<i_1i_2>\alpha$ are determined by the conclusion of Lemma \ref {5.8} and the comments in \ref {5.9}.  Moreover it is easy to give a bound on $h(\alpha)$ through Lemma \ref {5.8}.   However this bound is really only meaningful if one can actually describe the orbit, which in general is essentially impossible. Instead we just state a weaker result.

\begin {thm} Suppose $\mathfrak q_{\pi_1,\pi_2}$ is Frobenius. Let $m_j$ be the number of components of type $A_{2m}:m >0$ in $\pi_j:j=1,2$ and set $m:=1+ \text {max} \{m_1,m_2\}$. Then the semisimple element of an adapted pair $(h,\eta)$ (defined by the condition $h \in \mathfrak h, h(\beta)=1, \forall \beta \in B=-B_{\pi_1}\sqcup B_{\pi_2}$) satisfies $h(\alpha) \in \{0,\pm 1, \pm 2, \ldots \pm m\}$.  In particular the unique (up to conjugation) adapted pair has the integrality property.
\end {thm}

 \subsection{}\label{5.11}

 It is not hard to find Frobenius biparabolic subalgebras $\mathfrak q$ of $\mathfrak g_\pi$ such in the above $h(\pi)\nsubseteq \{0,\pm 1\}$.  The simplest cases are when $\mathfrak q$ is a parabolic subalgebra $p_{\pi'}$ with $\pi'$ defining a Levi factor with two blocks of coprime sizes $p,q$.  If $p=1$, then $h(\alpha)\in \{0,\pm 1\}$.  However if $p=2$ and $q=4m-1: m\in \mathbb N^+$, then $h(\pi)=\{\pm 1,2\}$. If $p=3,q=4$, then $h(\pi)= \{\pm 1, \pm 2\}$. In order for $3$ to belong to $h(\pi)$ we would need both connected components of $\pi'$ to have even cardinality.  This translates to both $p,q$ being odd and so $p+q$ being even.  However in this case it is rather easy to see that $h(\pi)=\{0,\pm 1\}$.  This exhausts all biparabolics for which $Sy(\mathfrak q)$ is polynomial on one generator \cite [2.2]{J3}.

  If $\pi$ is of exceptional type it is not hard to classify all Frobenius biparabolics satifying our standing hypotheses.  In only one case $h(\pi) \nsubseteq \{0,\pm 1\}$.  This is when $\pi$ is of type $E_7$ with  $\pi_1: = \pi \setminus \{\alpha_5\}, \pi_2 =\pi \setminus \{\alpha_7\}$, in which case $h(\pi) \in \{0,\pm 1,2\}$.

 A biparabolic subalgebra is said to be almost-Frobenius \cite [4.2]{J8} if there are no $<i_1i_2>$ orbits lying entirely in $\pi_\cap$.  This is just a slight extension of a Frobenius biparabolic.  Both families seem to be quite unclassifiable.

 The description of all Frobenius biparabolics in general type is part of the description of almost-Frobenius biparabolics \cite [Sect. 8]{J8} in type $A$.

 \section {Truncated Biparabolics - Generalities}\label{6}

 In this section we suppose that $\hat{\mathfrak a}$ is a biparabolic subalgebra $\mathfrak q_{\pi_1,\pi_2}$ of a simple Lie algebra $\mathfrak g_\pi$, satisfying our standing hypotheses, with $\mathfrak a$ the resulting truncated biparabolic.

  As already noted in \ref {5.7}, in \cite [4.5]{J2} we described a set of $<i_1i_2>$ orbits in an overset $\tilde{\pi}$ expected to be in bijection with a set of generating weight vectors in $Sy(\mathfrak q_{\pi_1,\pi_2})$.   (In the parabolic case that is when $\pi_2=\pi$ one may take $\tilde{\pi}=\pi$.) This bijection is well-defined when the two bounds in \cite [Thm. 6.7]{J4} coincide.  Moreover the resulting algebra is polynomial and the weights of the generators are given by orbit sums. In particular this holds if $\pi$ is of type $A$ or of type $C$.   In this case $Y(\mathfrak a)$ is polynomial.

  For the moment we just assume that $Y(\mathfrak a)$ is polynomial and that $\mathfrak a$ admits an adapted pair $(h,\eta)$.  In this we can assume that $\ad h$ is a semisimple endomorphism of $\mathfrak g$ and hence belongs to a Cartan subalgebra $\mathfrak h$ of $\mathfrak g$.

\subsection{}\label{6.1}

 Risking a slight confusion (with \ref {4.1}, \ref {4.6}) we denote by $\Delta$ the set of roots of $\mathfrak g_\pi$ with $\pi$ being a choice of simple roots.  In this $\pi_1,\pi_2$ are subsets of $\pi$ satisfying $\pi_1\cup \pi_2=\pi$.  Under this hypothesis, Lemma \ref {4.6} translates to give
$$ S\cup T\subset \Delta \subset \mathbb Q (S\cup T). \eqno {(*)}$$

 For every root $\alpha \in \Delta$, let $\alpha^\vee$ denote the corresponding coroot in $\mathfrak h$.  Then $s_\alpha: \lambda \mapsto \lambda - \alpha^\vee(\lambda)\alpha$ is the corresponding reflection in $\Aut \mathfrak h^*$ and the $s_\alpha:\alpha \in \pi$ generate $W_\pi$.  Given a subset $\pi' \subset \pi$, let $W_{\pi'}$ denote the subgroup of $W_\pi$ generated by the $s_\alpha: \alpha \in \pi'$.

Now define $\mathfrak g_\mathbb Z$ with respect to $h$ as in \ref {1.4} replacing $\mathfrak a$ by $\mathfrak g$. Again similarly define $(\mathfrak q_{\pi_1,\pi_2})_\mathbb Z$.  The proof of Lemma \ref {1.6} shows that $\mathfrak g_\mathbb Z$ is a reductive Lie algebra having $\mathfrak h$ as a Cartan subalgebra. Set $\Delta^\mathbb Z :=\{\alpha \in \Delta|h(\alpha) \in \mathbb Z\}$.  It is the set of roots of $\mathfrak g_\mathbb Z$.  Let $\pi^\mathbb Z$ a choice of simple roots for $\Delta^\mathbb Z$. Obviously $\mathbb Z \pi^\mathbb Z \cap \Delta =\Delta^\mathbb Z$. Since $\Delta^\mathbb Z$ is a set of roots of a semisimple algebra
 $$\Delta^\mathbb Z\subset \mathbb N\pi^\mathbb Z \sqcup -\mathbb N\pi^\mathbb Z. \eqno {(**)}$$

 Since $h$ takes integer values on $S$ and on $T$, we have $S\cup T \subset \mathbb Z \pi^\mathbb Z$.  Then by $(*)$ one obtains $|\pi^\mathbb Z|=|\pi|$.  Consequently $\mathfrak g_\mathbb Z$ is semisimple with the same rank as $\mathfrak g$.

More generally let $\mathfrak g'$ be a semisimple subalgebra of $\mathfrak g$ with the same rank as $\mathfrak g$. Then a Cartan subalgebra of $\mathfrak g'$ serves as a Cartan subalgebra of $\mathfrak g$ and so can be taken to be our chosen Cartan subalgebra $\mathfrak h$.  Then we may write $\mathfrak g'$ as $\mathfrak g_{\pi'}$ where $\pi'$ is a choice of simple roots for $\mathfrak g'$ lying in $\mathfrak h^*$, where in addition the rank condition gives $|\pi'|=|\pi|$.  Such a pair $(\pi,\pi')$ is called regular.

The description of all possible regular pairs $\pi, \pi'$ obtains from classical work of Dynkin \cite {Dy}, by an inductive procedure using ``enhanced" Dynkin diagrams.  The latter are described in \cite [Planches I, IX] {B} where they are called completed Dynkin graphs.

The Dynkin theory gives a little too much since not all choices of $\pi'$ are described as some $\pi^\mathbb Z$. Again the Dynkin theory only lists the possible choices of $\pi^\mathbb Z$ up to conjugation by $W_\pi$. Since $\mathfrak q_{\pi_1,\pi_2}$ is not $W_\pi$ invariant, this is not nearly enough.  On the other hand if we set $\pi_\cap =\pi_1\cap \pi_2$, then the subgroup $W_{\pi_\cap}$ (or simply, $W_\cap$) of $W_\pi$ generated by the simple reflections defined by the elements $\pi_\cap$ does leave $\mathfrak q_{\pi_1,\pi_2}$ invariant.  This will be used in \ref {7.2}.

To summarize what we need we make the following definition.

Call $(\pi,\pi^\mathbb Z)$ a regular integral pair for a simple Lie algebra $\mathfrak g$ with Cartan subalgebra $\mathfrak h$ and root system $\Delta$, if $\pi\subset \Delta$ is a choice of simple roots, if there exists $h \in \mathfrak h$ such that $\pi^\mathbb Z$ is a choice of simple roots for $\Delta^\mathbb Z:=\{\gamma \in \Delta|h(\gamma) \in \mathbb Z\}$ and if $|\pi^\mathbb Z|=|\pi|$.

Given $\gamma \in \Delta$, define the order $o(\gamma)$ of $\gamma$ (relative to $\pi$) to be the sum of the coefficients of $\gamma$ written a sum of elements of $\pi$.

One may recall that $\Delta^+:=\Delta \cap \mathbb N \pi$ is exactly the set of elements of $\Delta$ of positive order relative to $\pi$.

\begin {lemma}  Let $(\pi,\pi^\mathbb Z)$ be a regular integral pair.  For a fixed choice of $\pi^\mathbb Z$ one may choose $\pi\in \Delta$ such that $\pi^\mathbb Z \subset \Delta \cap \mathbb N \pi$.  Then $\pi^\mathbb Z$ is exactly the set of elements $\Delta^\mathbb Z$ of positive order which cannot be written as a sum of elements of  $\Delta^\mathbb Z$ of positive order.
\end {lemma}

\begin {proof}  The proof is standard, but we repeat the details for completeness.

Since the elements of $\pi^\mathbb Z$ are linearly independent there exists for any choice of positive rational numbers $c_\alpha: \alpha \in \pi^\mathbb Z$ an element $h_* \in \mathfrak h$ such that $h_*(\alpha)=c_\alpha, \forall \alpha \in \pi^\mathbb Z$. Moreover $h_*$ is unique, since $\mathbb Q \pi^\mathbb Z \supset \Delta$.  Choosing the $c_\alpha$ in general position, $h_*$ becomes regular.  Then $\Delta^+:=\{\alpha \in \Delta| h_*(\alpha) >0\}$ is a choice of positive roots containing $\pi^\mathbb Z$.

Choosing $\pi^\mathbb Z$ as prescribed by the second part of the lemma gives $\Delta^\mathbb Z \cap \Delta^+ \subset \mathbb N \pi^\mathbb Z$ and $(\alpha,\beta) \leq 0$, for all $\alpha, \beta \in \pi^\mathbb Z$.  A standard computation using that $(\gamma,\gamma)>0$ for all $\gamma \in \mathbb Z \pi$ non-zero, implies that the elements of $\pi^\mathbb Z$ are linearly independent, as required.
\end {proof}

\textbf{Remark 1}.  One calls the elements of $\pi^\mathbb Z$ given as prescribed by the second part of the lemma, the extremal elements of $\Delta^+\cap \Delta^\mathbb Z$.

\

\textbf{Remark 2}.  From the last remark before the lemma and the last part of the lemma, $(**)$ of \ref {6.1} results.


\subsection{}\label{6.2}

Given a regular integral pair $(\pi,\pi^\mathbb Z)$ we shall always assume from now on that $\pi^\mathbb Z \subset \Delta^+ \subset \mathbb N \pi$.  Then $\pi^\mathbb Z$ is uniquely determined by the pair $(\Delta^\mathbb Z,\Delta^+)$ as the extremal elements of $\Delta^\mathbb Z$ in $\Delta^+$.

Since the possible choices for the set of simple roots in $\Delta$ are all conjugate under the Weyl group we may assume that the choice of $\pi$ described in Lemma \ref {6.1} coincides with that used to define the biparabolic subalgebra $\mathfrak q_{\pi_1,\pi_2}$ of $\mathfrak g$.

Observe that by \ref {6.1}$(**)$ and the choice of $\Delta^+$, one has
 $$\mathbb Z \pi^\mathbb Z \cap \Delta \cap \mathbb N \pi = \Delta^\mathbb Z \cap \Delta^+ = \mathbb N \pi^\mathbb Z \cap \Delta \cap \mathbb N \pi. \eqno {(*)}$$

 (By contrast $\mathbb Z \pi^\mathbb Z \cap \mathbb N \pi \neq  \mathbb N \pi^\mathbb Z  \cap \mathbb N \pi$, in general.  For example if we take $\pi=\{\alpha_1,\alpha_2\}$ of type $C_2$ with $\alpha_2$ long and $h=\frac{1}{4}\alpha_1^\vee$, then $\pi^\mathbb Z=\{2\alpha_1+\alpha_2,\alpha_2\}$ and only the left hand side contains $2\alpha_1$.)

 Set $\pi^\mathbb Z_i=\pi^\mathbb Z \cap \mathbb N\pi_i:i=1,2$ and $\pi^\mathbb Z_\cap =\pi^\mathbb Z_1 \cap \pi^\mathbb Z_2$.

\begin {lemma} $(\mathfrak q_{\pi_1,\pi_2})_\mathbb Z$ is the biparabolic subalgebra $\mathfrak q_{\pi^\mathbb Z_1,\pi^\mathbb Z_2}$ of $\mathfrak g_\mathbb Z=\mathfrak g_{\pi^\mathbb Z}$.
\end {lemma}

\begin {proof}  A biparabolic subalgebra containing a fixed Cartan subalgebra $\mathfrak h$ of $\mathfrak g$ is uniquely determined by its set of roots, which in the case of $\mathfrak q_{\pi_1,\pi_2}$ is just $(-\mathbb N \pi_1 \cup \mathbb N \pi_2)\cap \Delta$.  It remains to note that $\mathbb Z \pi^\mathbb Z \cap (-\mathbb N \pi_1 \cup \mathbb N \pi_2) \cap \Delta= (-(\mathbb N \pi^\mathbb Z \cap \mathbb N \pi_1) \cup (\mathbb N \pi^\mathbb Z \cap \mathbb N \pi_2))\cap \Delta  = (-\mathbb N \pi^\mathbb Z_1 \cup \mathbb N \pi^\mathbb Z_2)\cap \Delta^\mathbb Z$.  Here the first step follows from $(*)$ and for the last step we claim that $\mathbb N\pi^\mathbb Z \cap \mathbb N\pi_i= \mathbb N(\pi^\mathbb Z \cap \mathbb N \pi_i)=\mathbb N\pi_i^\mathbb Z$, for $i=1,2$. Now since $\pi^\mathbb Z \subset \mathbb N \pi$, it follows that  $(\mathbb N \pi^\mathbb Z \cap \mathbb N \pi_i)$ is just the orthogonal in $\mathbb N \pi^\mathbb Z$ of the set of fundamental weights $\varpi_\beta: \beta \in \pi \setminus \pi_i$.  Yet $(\varpi_\beta,\gamma) \geq 0$ for every positive root $\gamma$ and so if $\sum_{\alpha \in \pi^\mathbb Z} n_\alpha \alpha \in \mathbb N \pi_i$ for some $n_\alpha \in \mathbb N$, then
$\alpha \in \mathbb N \pi_i$ whenever $n_\alpha \neq 0$.  Hence the claim.

\end {proof}

\subsection{}\label{6.3}

To simplify notation we write $\hat{\mathfrak a}=\mathfrak q_{\pi_1,\pi_2}$ with $\mathfrak a$ its canonical truncation in the remainder of this section.

It is not immediate that $\hat{\mathfrak a}_\mathbb Z=\mathfrak q_{\pi^\mathbb Z_1,\pi^\mathbb Z_2}$ satisfies our standing hypotheses.  First of all $\pi^\mathbb Z$ is not in general connected.  Again $\pi_1^\mathbb Z,\pi_2^\mathbb Z$ could be rather small. Nevertheless we have the

\begin {lemma}

\

(i)  The canonical truncation of $\hat{\mathfrak a}_\mathbb Z$ is contained in $\mathfrak a_\mathbb Z$.

\

(ii) $Y(\hat{\mathfrak a}_\mathbb Z)$ is reduced to scalars.

\

(iii) $\pi^\mathbb Z_1 \cup \pi^\mathbb Z_2= \pi^\mathbb Z$.

\begin {proof}  By definition, $\mathfrak a$ is the canonical truncation of $\hat{\mathfrak a}$, which is a biparabolic subalgebra of $\mathfrak g$ and as such satisfies our standing hypothesis. Thus $Y(\mathfrak a)^\mathfrak h$ which equals $Sy(\hat{\mathfrak a})^\mathfrak h = Y(\hat{\mathfrak a})$ is reduced to scalars. Since $\mathfrak h \subset \mathfrak a_\mathbb Z$, the projection $\mathscr P:Y(\mathfrak a)\rightarrow Y(\mathfrak a_\mathbb Z)$ defined in \ref {2.6} is an $\mathfrak h$ module map.  Then by Lemma \ref {2.6} we conclude that $Y(\mathfrak a_\mathbb Z)$ is isomorphic to $Y(\mathfrak a)$ as an $\mathfrak h$ module.  In particular $Y(\hat{\mathfrak a}_\mathbb Z)=Y(\mathfrak a_\mathbb Z)^\mathfrak h$ is reduced to scalars. Hence (ii).

Since $\mathfrak a_\mathbb Z$ obtains from $\hat{\mathfrak a}_\mathbb Z$ on replacing $\mathfrak h$ by the common kernel of the weights of $Y(\mathfrak a_\mathbb Z)$, (i) also results.

The orthogonal in $\mathfrak h$ of the left hand side (iii) belongs to $Y(\hat{\mathfrak a}_\mathbb Z)$, so is zero by (ii).  Hence (iii).
\end {proof}
\end {lemma}

\subsection{}\label{6.4}


Let $(\pi,\pi^\mathbb Z)$ be a regular integral pair and recall that $\pi^\mathbb Z \subset \mathbb N \pi$ and that $|\pi^\mathbb Z|=|\pi|$.  If $\pi$ is of type $A$, a special case of the Dynkin theory implies that $\pi^\mathbb Z =\pi$ and so an adapted pair for a truncated parabolic subalgebra of $\mathfrak g_\pi$ has the integrality property.  This is just the result noted in \cite [9.10]{FJ2}.

We show below that this result extends to the case when both $\pi_1$ and $\pi_2$ are of type $A$, but $\pi$ need not be.

Surprisingly the proof is not so straightforward and we are seemingly forced to consider separately the case when $\pi$ (which we assume connected) is, or is not, simply-laced.  In this we recall the distinguished simple root $\alpha_*$ defined outside type $A$ described for example in \cite [5.5]{J8} and whose definition will be recalled in \ref {6.5} and \ref {6.6} below.

\subsection{}\label{6.5}

Assume that (the Dynkin diagram of) $\pi$ is connected and not simply-laced, that is $\pi$ is of type $B,C,F$, or $G$. Then $\alpha_*$ is defined to be the unique long root with a short root nearest neighbour.  This short root is also unique and will be denoted by $\alpha_*^s$.

\begin {lemma} Assume $\pi$ of type $B,C,F$, or $G$ and $\pi_1,\pi_2$ are both of type $A$.  Then an adapted pair $(h,\eta)$ for the canonical truncation of the biparabolic subalgebra $\mathfrak q_{\pi_1,\pi_2}$ of $\mathfrak g_\pi$ has the integrality property.
\end {lemma}

\begin {proof}  By our standing hypothesis $\alpha_*,\alpha_*^s$ must belong to $\pi_1\cup \pi_2$.  However the hypothesis of the lemma implies that they cannot both belong to $\pi_1$ or to $\pi_2$.  Thus just one (say $\alpha_*$) belongs to $\pi_2$ and just the other belongs to $\pi_1$.  This has the following consequence.  Consider the simple root system $\pi'$ obtained from $\pi$ by deleting the lines joining $\alpha_*$ and $\alpha_*^s$.  Then $\mathfrak q_{\pi_1,\pi_2}$ can be identified with a biparabolic subalgebra of $\mathfrak g_{\pi'}$.  Since $\pi'$ is of type $A$, the assertion follows by the corresponding result in type $A$.
\end {proof}

\textbf{Remark}.  One may construct an adapted pair for the canonical truncation of $\mathfrak q_{\pi_1,\pi_2}$ by considering it as a subalgebra of $\mathfrak g_{\pi'}$.  Since we have shown (\cite {J5} that an adapted pair always exists for truncated biparabolics in type $A$, it follows that an adapted pair always exist for truncated biparabolics satisfying the hypothesis of the lemma.

\subsection{}\label{6.6}

Now suppose $\pi$ is connected, simply-laced but not of type $A$, that is $\pi$ is of types $D,E$.  Then $\alpha_*$ is the unique simple root with three nearest neighbours in the Dynkin diagram of $\pi$.

\begin {lemma} Assume $\pi$ of type $A,D$, or $E$ and $\pi_1,\pi_2$ are both of type $A$.  Then an adapted pair $(h,\eta)$ for the canonical truncation of the biparabolic subalgebra $\mathfrak q_{\pi_1,\pi_2}$ of $\mathfrak g_\pi$ has the integrality property.
\end {lemma}

\begin {proof}   Take $\gamma \in \pi^\mathbb Z$. By the hypothesis and \ref {6.2}(iii), $\gamma$ written as a sum of elements of $\pi$ must have coefficients in $\{0,1\}$. Those $\alpha \in \pi$ occurring with coefficient $1$ are said to form the support of $\gamma$. Then $\gamma$ is determined by its support which must be connected subset of $\pi$.  Set $\pi':=\{\alpha \in \pi|h(\alpha) \notin \mathbb Z\}$.  Obviously $\pi\setminus \pi' \subset \pi^\mathbb Z$.  In particular we can assume that $\pi'$ is non-empty.  Since $\pi^\mathbb Z$ consists of the extremal elements of $\Delta^+\cap \Delta^\mathbb Z$, the support of $\gamma \in \pi^\mathbb Z$ is minimal under the condition that $h(\gamma) \in \mathbb Z$.

Assume further that $\gamma \notin \pi$. Call an element of the support of $\gamma$ an end-point if it has just one nearest neighbour in the support of $\gamma$.  Minimality of support implies that an end-point lies in $\pi'$. It then follows that $\gamma$ must have at least two end-points.  However it cannot have three (or more) since otherwise by \ref {6.3}(iii), $\pi_1$ or $\pi_2$ could not be of type $A$.

With respect to the total ordering of $\pi$ given by the numbering in Bourbaki \cite [Planches I-IX]{B}, let $\varphi(\gamma) \in \pi'$ be the smaller of the two end-points of $\gamma$.

Suppose $\pi$ is of type $A$ and that $\pi'$ is not empty. Then $\varphi$ is an injection of $\pi^\mathbb Z \setminus (\pi \setminus \pi')$ onto a proper subset of $\pi'$, omitting in particular the largest of all end-points. This implies that $|\pi^\mathbb Z \setminus (\pi \setminus \pi')| < |\pi'|$, giving the contradiction $|\pi^\mathbb Z| <|\pi|$. We conclude that $\pi'$ is empty, reproducing what we already know, namely that an adapted pair has the integrality condition in type $A$.  (We have included this case for completeness and for illustration.)

Suppose $\pi$ is of type $D$ or of type $E$. Then there is just one way to avoid a similar contradiction.  Namely that there are three elements of $\pi^\mathbb Z \setminus (\pi \setminus \pi')$ whose end-points are the three different possible pairs of the set  $\{\alpha^1,\alpha^2,\alpha^3\}$ whose members lie on the three different branches of the Dynkin diagram of $\pi$ emanating from $\alpha_*$ and distinct from $\alpha_*$.   Since $|\pi^\mathbb Z|=|\pi|$, this forces the members of at least two of these pairs to lie in $\pi_1$ (or $\pi_2$).  However in this case $\pi_1$ (or $\pi_2$) must be of type $D$ or of type $E$ in contradiction to the hypothesis.
\end {proof}

\subsection{}\label{6.7}

Let us recall that the first step of Dynkin's construction of a regular pair $(\pi,\pi')$ is to adjoin the negative $-\beta$ of the highest root and delete one root, say $\alpha$ from $\pi$, that is to say we take $\pi'=(\pi\setminus \{\alpha\})\cup \{-\beta\}$.  Since we would further like that $\pi' \subset \mathbb N\pi$, we replace $\pi'$ by $-w_{\pi\setminus \{\alpha\}}\pi'=\{w_{\pi\setminus \alpha}\beta\}\cup(\pi\setminus \{\alpha\})$.
In the above we call $\alpha$ the deleted root and $w_{\pi\setminus \alpha}\beta$ the added root.  Observe that if the coefficient of $\alpha$ is a positive integer and that if it equals $1$, then $\pi'=\pi$.  Indeed the first statement follows from linear independence and the second that if we define $h \in \mathfrak h$ by the condition that $h(\alpha)=1$ for all $\alpha \in \pi'$, then the second condition implies that $h$ takes integer values on $\Delta$.  This in turn forces $\pi'$ to be conjugate to $\pi$, but since we have constructed the latter to lie in $\mathbb N\pi$, we must have the asserted equality.

Then this construction (of Dynkin) is repeated to the connected components of $\pi'$ not of type $A$.

\subsection{}\label{6.8}

Assume $\pi$ connected.  Let $m_\pi$ denote the largest coefficient that occurs in the highest root (expressed as a sum of elements of $\pi$.  If $\pi'$ is a subset of $\pi$, then clearly $m_{\pi'} \leq m_\pi$.  One may check that  $M_\pi:=\{1,2,\ldots, m_\pi\}$ is just the set of coefficients in the highest root.

Define an order relation $>$ on the set of positive roots by $\beta\geq\gamma$, if $\varpi_\alpha(\beta) \geq  \varpi_\alpha(\gamma)$, for all $\alpha \in \pi$.

\begin {lemma} Let $\mathfrak a$ be the canonical truncation of a biparabolic subalgebra $\mathfrak q_{\pi_1,\pi_2}$ of $\mathfrak g_\pi$.  Assume that $Y(\mathfrak a)$ is polynomial (so then $\mathfrak a$ is regular). Let $(h,\eta)$ be an adapted pair for $\mathfrak a$.  Then the eigenvalues of $\ad h$ on $\mathfrak a$ lie in $\{\frac{1}{m}\mathbb Z\}$ for some $m \in M_\pi$.
\end {lemma}

\begin {proof}

The proof is by induction on the number of steps in the Dynkin construction.  In this we can assume that the assertion holds for root systems of smaller cardinality than $|\Delta|$.

Consider the first step of the Dynkin construction giving the regular pair $\pi,\pi'$. Let $m$ be the coefficient of the deleted root $\alpha$ in the added root $\beta$.  Then $m \in M_\pi$ and as explained above we can assume $m>1$.  It is clear that if $h \in \mathfrak h$ is integer on $\pi'$, then $h(\pi)\subset \{\frac{1}{m}\mathbb Z\}$.

 For subsequent steps in the Dynkin construction, assume first that the connected component $\pi''$ of $\pi'$ containing the added root $\beta$ of type $A$.  Then we need not modify this component in a subsequent step.

 Assume next that $\pi''$ of $\pi'$ is not of type $A$.  Since $m \geq 2$, it easily follows that it generates a root subsystem of smaller cardinality than $|\Delta|$.

 We claim that there is no positive root $\gamma$ such that $\gamma \geq 2\beta$.  This is obvious for $\pi$ classical, since all the coefficients of a positive root are at most $2$.  For $\pi$ exceptional it can be easily checked.  We conclude from the remarks preceding the lemma that an added root never becomes a deleted root in a subsequent step.  Moreover the coefficient of the deleted root $\alpha$ is trivially a multiple $m$ in any of $\Delta':=\Delta \cap \mathbb Z\pi'$.  Finally we claim that $\pi''$ is always classical and $m=2$.  This is immediate if $\pi$ itself is classical.   For $\pi$ exceptional this can be easily checked. Thus $m_{\pi''}=2$.

Finally suppose that $\pi'''$ is a connected component of $\pi$ different than $\pi''$. Then it must have strictly smaller cardinality so we may assume that the conclusion of the Lemma applies to $\pi'''$.  We claim that the least common multiple of $m$ and an element of $M_{\pi'''}$ lies in $M_\pi$. This is obvious if $\pi$ is classical.   For $\pi$ exceptional this can be easily checked, though a priori it is surprising.

The conclusion of the lemma results.

\end {proof}

\textbf{Remark}. If $\pi$ is of classical type and the truncated biparabolic $\mathfrak a$ is regular, then the eigenvalues of the semisimple element of an adapted pair are all half-integer.  This is all we shall need of the Dynkin theory in the next two sections.

 \section {Truncated Biparabolics - The Classical Case}\label{7}

 Let us emphasize again that we now abandon the shorthand notation $i_j$ for $i_{\pi_j}$ in this and the next section introduction in \ref {5.7} for comparison with the notation of \cite [4.5] {J2}.

 \subsection{}\label{7.1}

 Let $(\pi,\pi^\mathbb Z)$ be a regular integral pair with $\pi$ connected and of classical type but not of type $A$. Set $n=|\pi|$.

 By Lemma \ref {6.8}, one has $h(\alpha) \in \frac{1}{2}\mathbb Z$, for all $\alpha \in \pi$.

  Label $\pi$ as in Bourbaki \cite [Planches II-IV]{B}. (Thus $\pi$ inherits a total order from $\mathbb N^+$.)

  Let $\pi^{\frac{1}{2}}:=\{\alpha_{i_1},\alpha_{i_2},\ldots,\alpha_{i_r}:1\leq i_1<i_2< \ldots <i_r\leq n\}$ be the set of roots on which the value of $h$ is half-integer (and not integer).

  Having fixed $\pi^{\frac{1}{2}}$, then $\pi^\mathbb Z$ is uniquely determined as the set of extremal elements of $\mathbb N\pi^\mathbb Z\cap \Delta$ which we describe below.  However not for all choices of $\pi^{\frac{1}{2}}$ will we have $|\pi^\mathbb Z|=|\pi|$ and the choices for which this fails will also be described.

  We describe below just all cases when $\pi^{\frac{1}{2}}$ is non-empty. Of course otherwise  $\pi^\mathbb Z =\pi$.

   \subsubsection{}\label{7.1.1}

  Recall the above labelling of $\pi^{\frac{1}{2}}$.  A large part of $\pi^\mathbb Z$ obtains rather quickly.   In this, type $D_n$ is rather exceptional so we assume till \ref {7.1.5} that $i_r <n$ in type $D_n$.  Set $\gamma_j:= \alpha_{i_{j}}+\alpha_{i_{j}+1}+\ldots +\alpha_{i_{j+1}}:j=1,2,\ldots,r-1$ and $\Gamma:=\{\gamma_j\}_{j=1}^{r-1}, \hat{\Gamma}=\Gamma\cup (\pi \setminus \pi^{\frac{1}{2}})$.  It is immediate that the elements of $\hat{\Gamma}$ lie in $\Delta^+\cap \Delta^\mathbb Z$ and are extremal.   Thus $\hat{\Gamma} \subset \pi^\mathbb Z$.  Moreover $\hat{\Gamma}$ is a simple root system whose cardinality is $|\pi|-1$.  It is connected if $r=1$, or if $r=2$ and $i_1,i_2$ are consecutive integers, and has exactly two connected components otherwise.

   Given a further extremal root $\gamma \in \Delta^+\cap \Delta^\mathbb Z$,  then $\pi^\mathbb Z=\hat{\Gamma} \cup\{\gamma\}$.

  Let us examine the possible candidates for $\gamma$ written as an element of $\mathbb N\pi$.

  One checks that there are no further extremal roots with all coefficients being $\leq 1$.  Thus in the Bourbaki notation we can write $\gamma =\varepsilon_i+\varepsilon_j$ with $i\leq j$ there being equality only in type $C$.

   \subsubsection{}\label{7.1.2}

  Suppose that $\pi$ is of type $B_n: n \geq 3$.

  \

  Then $\gamma:=\varepsilon_i+\varepsilon_j:i<j$ lies in $\Delta^+\cap \Delta^\mathbb Z$ and is extremal, if and only if $j=i_r$ and either $i=j-1\notin \pi^{\frac{1}{2}}$  or $i=i_{r-2}, j-1 =i_{r-1}$.  Here the first condition presupposes that $i_r >1$ and the second that $r \geq 3$.

   We conclude that there are exactly two choices of $\pi^{\frac{1}{2}}$ which do not admit a regular integral pair.   Either $r=1$ and $i_1=1$, or $r=2$ and $i_1,i_2$ are consecutive integers.

   \subsubsection{}\label{7.1.3}

   Suppose that $\pi$ is of type $C_n: n \geq 2$.

   \

    Suppose $\gamma=\varepsilon_i+\varepsilon_j:i\leq j$ lies in $\Delta^+\cap \Delta^\mathbb Z$ and is extremal.  If $i<j$, then $\alpha_i \in \pi^{\frac{1}{2}}$, say $i=i_t$ and is the unique integer $t:i\leq t <j$ with this property.  Again  $\alpha_j \in \pi^{\frac{1}{2}}$  and is the unique integer $t:j\leq t <n$ with this property. This in turn implies that $\alpha_n \notin \pi^{\frac{1}{2}}$ and similarly $i=j$.  Thus $t=r$.  We conclude that $\pi^{\frac{1}{2}}$ admits a regular integral pair if and only if $i_r<n$.  Moreover the additional root $\gamma$ is $2\varepsilon_{i_r}$.

     \subsubsection{}\label{7.1.4}

    Suppose that $\pi$ is of type $D_n: n \geq 4$ and that $i_r <n$.

   \

    Suppose $\gamma=\varepsilon_i+\varepsilon_j:i<j$ lies in $\Delta^+\cap \Delta^\mathbb Z$ and is extremal. Then  $\alpha_j \in \pi^{\frac{1}{2}}$  and is the unique integer $t:j\leq t <n-1$ with this property.

    Just as in type $C_n$ we conclude that $\alpha_{n-1} \notin \pi^{\frac{1}{2}}$. The case $\alpha_{n-1} \notin \pi^{\frac{1}{2}}$ is similar.

    We conclude that neither $\alpha_{n-1}$ nor $\alpha_n$ can belong to $\pi^{\frac{1}{2}}$.

    Thus $j=i_r$.

     Then just as in type $B_n$, we deduce that $\gamma:=\varepsilon_i+\varepsilon_j:i<j$ lies in $\Delta^+\cap \Delta^\mathbb Z$ and is extremal, if and only if $j=i_r$ and either $i=j-1\notin \pi^{\frac{1}{2}}$  or $i=i_{r-2}$.  Here the first condition presupposes that $i_r >1$ and the second that $r \geq 3$.

   We conclude that there are two further exclusions on $\pi^{\frac{1}{2}}$ not admitting a regular integral pair.   Either $r=1$ and $i_1=1$, or $r=2$ and $i_1,i_2$ are consecutive integers.

    \subsubsection{}\label{7.1.5}

    Suppose that $\pi$ is of type $D_n: n \geq 4$.

    Suppose $i_{r-1}<n-1$.  Then by interchanging the roots $\alpha_{n-1},\alpha_n$ we are reduced to the case considered in \ref {7.1.4} in which case $i_r<n-1$.

   We are reduced to the case when $\alpha_{n-1}, \alpha_n$ both belong to $\pi^{\frac{1}{2}}$, that is $i_{r-1}=n-1, i_r=n$.  Then in the definition of $\Gamma$ we should take $\gamma_{r-1}=\alpha_{i_{r-2}}+\alpha_{i_{r-2}+1}+\dots +\alpha_{n-2}+\alpha_n$.  As before we obtain $|\hat{\Gamma}|=|\pi|-1$.  To this we may adjoin a further extremal root $\gamma \in \Delta^+\cap \Delta^\mathbb Z$ exactly when $r\geq 3$, namely $\gamma:=\alpha_{i_{r-2}}+\alpha_{i_{r-2}+1} +\ldots +\alpha_{n}=\varepsilon_{i_{r-2}}+\varepsilon_n$.

 \subsection{}\label{7.2}

 We now have the utterly daunting task of showing that all possible solutions for $\pi^\mathbb Z \neq \pi$ given in \ref {7.1} are incompatible with the conclusions in Sections \ref {2}, \ref {6}.  Nevertheless there is one important simplification that can be introduced.

 Let $\pi_{1,u}:u \in U$ be the set of connected components of $\pi_1$ and $\pi_{2,v}:v \in V$ the set of connected components of $\pi_2$.  Then if non-empty $\pi_{1,u}\cap \pi_{2,v}$ is a connected component $\pi_\cap=\pi_1\cap \pi_2$ and all such connected components are so obtained. We denote it by $\pi_{\cap(u,v)}$ and call it a double component.

 Let $W_{\pi_{\cap(u,v)}}$ denote the subgroup of the Weyl group generated by the reflections $s_\alpha:\alpha \in \pi_{\cap(u,v)}$.  Since $\mathfrak q_{\pi_1,\pi_2}$ is stable under $W_{\pi_{\cap(u,v)}}$, it follows that a translate of the adapted pair $(h,\eta)$ (which defined $\pi^{\frac{1}{2}}$) under $W_{\pi_{\cap(u,v)}}$ is still an adapted pair for $q_{\pi_1,\pi_2}$.

 Following \ref {6.2}, set $\pi^\mathbb Z_{\cap(u,v)}=\mathbb N\pi_{\cap(u,v)}\cap \pi^\mathbb Z$.

 \begin {lemma} Up to translation by $W_{\pi_{\cap(u,v)}}$ one has $|\pi^{\frac{1}{2}}\cap \pi_{\cap(u,v)}|\leq 1$, for every connected component $\pi_{\cap(u,v)}$ of $\pi_\cap$.
 \end {lemma}

 \begin {proof}   Since the action of $W_{\pi_{\cap,(u,v)}}$ also alters $\pi$ the proof will not be entirely direct.  What we do is to apply a suitable element $w$ of this group to $\pi^\mathbb Z$ and from its image $w\pi^\mathbb Z$ infer the change in $\pi^{\frac{1}{2}}$ defined with respect to the image of $\pi$ under $w$ which one can recover $w\pi^\mathbb Z$. The example in the remark following the proof of the lemma shows how this process works.

 Let $j$ is the smallest integer $1\leq j \leq r$ such that $\alpha_{i_j}\in \pi_{\cap(u,v)}$ and set $k:=|\pi_{\cap (u,v)}\cap [1,i_j-1]|$.  If $j=r$ or if $\alpha_{i_{j+1}}\notin \pi_{\cap(u,v)}$, there is nothing to prove.  Otherwise we apply the product of reflections $s_{i_{j+1}-1}s_{i_{j+1}-2}\ldots s_{i_j}$ to $\pi_\cap^\mathbb Z$.
 Only $\pi^\mathbb Z_{\cap(u,v)}$ is changed and we claim that from the change in its form we may infer how $\pi^{\frac{1}{2}}$ should be altered.  In detail:  if $k>0$, then $k$ is reduced by $1$ with $j$ being unchanged, whereas if $k=0$, then $j$ is increased by $1$. Hence the assertion of the lemma (up to the above claim).

 To prove our claim consider the roots in $\pi^\mathbb Z_{\cap(u,v)}$ which are changed by applying the above element of $W$.  Adopt the hypothesis of \ref {7.1.1}. The easiest case is when the additional extremal root $\gamma$ described in \ref {7.1.1} does not lie in $\pi^\mathbb Z_{\cap(u,v)}$ or if $j<r-2$ which means that
  $\gamma$ does not change.  Then the only roots which are changed are the pair $(s[i_j,i_{j+1}],i_{j+1}-1)$ and $i_j-1$ if it lies in $\pi^\mathbb Z_{\cap(u,v)}$.  The former become the pair $(i_{j+1},i_{j+1}-2)$ and the latter $s[i_j-1,i_{j+1}-1]$. This means that $i_{j+1}$ has been replaced by $i_{j+1}-1$ in $\pi^{\frac{1}{2}}$, that $i_j$ has been deleted from $\pi^{\frac{1}{2}}$ and replaced by $i_j-1$ if the latter belongs to $\pi^\mathbb Z_{\cap(u,v)}$ (which is the case when $k$ is reduced by $1$).

 This conclusion remains the same when the additional extremal root $\gamma$ lies in $\pi^\mathbb Z_{\cap(u,v)}$ and is changed as checked in detail below.

  Suppose $j=r-1$. In types $B,D$ one has $\gamma=\varepsilon_{i_r-1}+\varepsilon_{i_r}$, which becomes $\varepsilon_{i_r-2}+\varepsilon_{i_r-1}$ corresponding again to $i_{j+1}$ being replaced by $i_{j+1}-1$ in $\pi^{\frac{1}{2}}$. In type $C$ one has $\gamma=2\varepsilon_{i_r}$, which becomes $2\varepsilon_{i_r-1}$ leading to a similar conclusion.

  Suppose $j=r-2$ and $i_{r-1}=i_{r-2}+1$.  In types $B,D$ one has $\gamma=\varepsilon_{i_{r-2}}+\varepsilon_{i_r}$, which becomes $\varepsilon_{i_r-1}+\varepsilon_{i_r}$ corresponding again to $i_{j+1}=i_{r-1}$ being replaced by $i_{j+1}-1$ in $\pi^{\frac{1}{2}}$ and $i_j$ deleted from $\pi^{\frac{1}{2}}$. However since $i_{j+1}-1=i_{r-1}-1=i_{r-2}=i_j$ in this case the overall effect is to just delete $i_{j+1}$ from $\pi^{\frac{1}{2}}$ (and to replace it by $i_j-1$ if the latter belongs to $\pi^\mathbb Z_{\cap(u,v)}$).  In type $C$ there is nothing new to consider.

  Finally as in \ref {7.1.5} it remains to consider the case of type $D_n$ when $\alpha_{n-1},\alpha_n \in \pi^{\frac{1}{2}}$.  Again one checks that applying the above element of $W$ to $\pi^\mathbb Z_{\cap(u,v)}$ replaces $i_{j+1}$ by $i_{j+1}-1$ in $\pi^{\frac{1}{2}}$, deletes $i_j$ from $\pi^{\frac{1}{2}}$ and replaces it by $i_j-1$ if the latter belongs to $\pi^\mathbb Z_{\cap(u,v)}$.

 \end {proof}

  \textbf{Example}.  Let us give an example in type $C_6$.  Suppose that $\pi^{\frac{1}{2}}=\{3,5\}, \pi_1=[2,6],\pi_2=\pi$. In this we may simply write $\pi_{u,v}^\mathbb Z$ as $\pi_1^\mathbb Z$.

  In the notation of the Lemma one has $k=1,1=j<r=2$. One checks that $\pi_1^\mathbb Z = \{2,s[3,5],6;4,\beta_5\}$, where the semi-colon separates the two connected components of $\pi_1^\mathbb Z$.  The recipe of the Lemma is to apply $s_4s_3$ to this expression which replaces it by $\{s[2,4],5,6;3,\beta_4\}$.  Of course $\{\pi,\pi_1,\pi_2\}$ is also altered but as the biparabolic is not we can assume these bases to be re-chosen so that their images under $s_4s_3$ is again the set $\{\pi,\pi_1,\pi_2\}$.  Then we infer from the expression for the new $\pi_1^\mathbb Z$, that $\pi^{\frac{1}{2}}=\{2,4\}$. In this $k=0,1=j<r=2$ in accordance with the claim in the lemma.   The recipe of the Lemma is to apply $s_3s_2$ to $\{s[2,4],5,6;3,\beta_4\}$ which replaces it by $\{4,5,6;2,\beta_3\}$, from which we infer that $\pi^{\frac{1}{2}}=\{3\}$.  In this $k=0,1=j=r'=1$ and we are done.

 \subsection{}\label{7.3}

  Before tackling the task outlined in the first part of \ref {7.2}, we start with some easy general considerations which are independent of type.

  Recall again that we are assuming that $\mathfrak q$ admits and adapted pair $(h,\eta)$ which then defines $\pi^\mathbb Z$.

 First of all $\pi^\mathbb Z$ must satisfy \ref {6.3}(iii).  This excludes some choices of the pair $\pi_1,\pi_2$.  However no choices are excluded in the parabolic case, that is when $\pi_2=\pi$.

 \subsection{}\label{7.4}

 Fix $\pi_1, \pi_2 \subset \pi$.  Recall the notation of \ref {1.2}.   One may calculate $r\ell:=r\ell(\mathfrak q_{\pi_1,\pi_2})$, $r\ell^\mathbb Z:= r\ell(\mathfrak q_{\pi^\mathbb Z_1,\pi^\mathbb Z_2})$ using \cite [5.9,7.16,7.17]{J2}.  It is immediate from Lemma \ref {2.6} that we must have $r\ell^\mathbb Z \geq r\ell$.

 Suppose that
 $$r\ell^\mathbb Z = r\ell. \eqno{(*)}$$

Let $\mathfrak a$ denote the canonical truncation of $\mathfrak q_{\pi_1,\pi_2}$.  Then by construction $S(\mathfrak a)$ admits no proper semi-invariants.  Then by \cite [Thm. 1.11(i)]{DDV} we conclude that $\mathfrak a$ is unimodular.  On the other hand our hypothesis combined with Theorem \ref {2.12} implies that $S(\mathfrak a_\mathbb Z)$ has no proper semi-invariants and hence by \cite [Thm. 1.11(i)]{DDV} we conclude that $\mathfrak a_\mathbb Z$ is also unimodular.  On the other hand this is obviously incompatible with the conclusion of Lemma \ref {2.5} unless $\mathfrak c=0$, that is $\mathfrak a_\mathbb Z = \mathfrak a$.   (These considerations which apply to any regular Lie algebra $\mathfrak a$ were our original motivation for proving Theorem \ref {2.12}.)

 One can easily find examples when $(*)$ holds (indeed infinitely many examples exist if that gives one any joy).  A particularly simple situation in which this occurs is when already $r\ell$ takes its maximal value for a biparabolic subalgebra of $\mathfrak g_\pi$, namely $\dim \mathfrak h$.  For example one can take $\pi=\pi_2$ of type $B_n$ and $\pi_1$ of type $B_{n-1}$. For $n=3$ an adapted pair for the corresponding truncated parabolic was described in \cite [8.16]{J1.2}.  P. Lamprou \cite {L} reported that she had generalized this construction for all $n$.

 \subsection{}\label{7.5}

 Let us explain in general terms how the question of integrability of an adapted pair should be settled in the case of the canonical truncation of a biparabolic subalgebra $\mathfrak q_{\pi_1,\pi_2}$.  We write the latter simply as $\mathfrak q$ and its canonical truncation as $\mathfrak q_\Lambda$ .  It is assumed that $\mathfrak q_\Lambda$ is regular and to admit an adapted pair $(h,\eta)$.

 There are two sets of data which are used.  The first is the set $\{\alpha \in \pi|h(\alpha)\notin \mathbb Z\}$. For $\mathfrak g$ classical, this is just $\pi^{\frac{1}{2}}$ which can be nearly any subset of $\pi$ (the possible excluded cases are given in \ref {7.1}).  From this we may calculate $\pi^\mathbb Z$, which can assumed to lie in $\Delta \cap \mathbb N\pi$ and then is uniquely determined (Lemma \ref {6.1}).

 The second set of data is just the pair $(\pi_1,\pi_2)$ of subsets of $\pi$.  Using \ref {6.2} we may then calculate the pair $(\pi_1^\mathbb Z, \pi_2^\mathbb Z)$.

 By Lemma \ref {2.6} and Theorem \ref {2.12}, one has $Sy(\mathfrak q_\mathbb Z)[q_j^{-1}:j \in \mathscr J] = \mathscr P (Sy(\mathfrak q))[q_j^{\pm}:j \in \mathscr J]$.  Dropping absolute precision of language we shall refer to the $\{q_j\}_{j \in \mathscr J}$ as the additional generators of $Sy(\mathfrak q_\mathbb Z)$ over $Sy(\mathfrak q)$.
By Lemma \ref {2.11} their weights $\{\gamma_j\}_{j \in \mathscr J}$ freely generate an additive subgroup $\Gamma$ of $\mathfrak h^*$.

Let $\mathfrak q_\Lambda$ denote the canonical truncation of $\mathfrak q$.
Thus in keeping with our notation in \ref {2.12} we should denote the canonical truncation of $(\mathfrak q_\Lambda)_\mathbb Z$ by $((\mathfrak q_\Lambda)_\mathbb Z)_\Gamma$. However the latter is also the canonical truncation of $\mathfrak q_\mathbb Z$ which we denote by $\mathfrak q_{\mathbb Z,\Lambda^\mathbb Z}$. Set $\mathfrak h_\Lambda=\mathfrak h \cap \mathfrak q_\Lambda$ (resp. $\mathfrak h^\mathbb Z=\mathfrak h \cap \mathfrak q_{\mathbb Z,\Lambda^\mathbb Z}$). As in \ref {2.12} we define $\mathfrak h_\Gamma$ to be a complement to $\mathfrak h_{\Lambda^\mathbb Z}$ in $\mathfrak h_\Lambda$. Recall \ref {2.12} that  $\mathfrak h_\Gamma$ is non-degenerately paired to $\Gamma$.  A choice of $\mathfrak h_\Gamma$ may be calculated from the datum $(\pi_1,\pi_2,\pi_1^\mathbb Z, \pi_2^\mathbb Z)$, through the formula in \cite [5.9(i)]{J2}.  In the simplest case $\mathfrak h_\Lambda$ (resp. $\mathfrak h_{\Lambda^\mathbb Z}$) is the span of the coroots corresponding to the roots in $\pi_\cap$ (resp. $\pi_\cap^\mathbb Z$). In general the formulae are more complicated.

 For most truncated biparabolic subalgebras \cite [Thm. 6.7]{J4} asserts that their invariants (for co-adjoint action) generate a polynomial algebra and moreover we can describe the weights and degrees of each generator.  We apply this result to $\mathfrak q_{\mathbb Z,\Gamma}$ whose invariant algebra we already know to be polynomial by Theorem \ref {2.12}.  In many cases we are thus able to calculate the weights and degrees of the generators.  Then it remains to compute the generators of $Y(\mathfrak q_{\mathbb Z,\Gamma})^{\mathfrak h_\Gamma}$ and to obtain a contradiction using \ref {2.14}.

 \subsection{}\label{7.6}

 Before going further let us summarize what the above stated contradiction should involve.  First retaining the notation and hypotheses of \ref {7.5}, it follows from Theorem \ref {2.12}, Lemma \ref {2.6}, \ref {2.13} and \ref {7.4} that

 \

 (i)  $Sy(\mathfrak q_\mathbb Z)=Y(\mathfrak q_{\mathbb Z,\Lambda^\mathbb Z})$ is polynomial on strictly more generators than the polynomial algebra $Sy(\mathfrak q)$.

 \

 (ii)  The additive subgroup $\Gamma$ of $\mathfrak h^*$ freely generated by weights of the additional generators of $Sy(\mathfrak q_\mathbb Z)$ over $Sy(\mathfrak q)$ is non-degenerately paired to a complement $\mathfrak h_\Gamma$ to $\mathfrak h_{\Lambda^\mathbb Z}$ in $\mathfrak h_\Lambda$.  Moreover $Y(\mathfrak q_{\mathbb Z,\Lambda^\mathbb Z})^{\mathfrak h_\Gamma}=Y(\mathfrak q_\Lambda)$ and so is polynomial.

 \

 (iii) The homogeneous generators $p_i:i \in I$ of $Y(\mathfrak q_{\mathbb Z,\Lambda^\mathbb Z})^{\mathfrak h_\Gamma}$ (resp. $\hat{p}_i:i\in I, q_j:j \in J$) of  $Y(\mathfrak q_{\mathbb Z,\Lambda^\mathbb Z}$) can be chosen so that $p_i$ is a product of $\hat{p}_i$ times a product of the $q_j:j \in J$.

 \

 (iv) The set of weights of the $p_i:i \in I$ coincide with the set of weights of the generators of $Sy(\mathfrak q)$.

 \

 We call (iii) the factorisation property.  For parabolic subalgebras in type $C$ studied in the next section we shall use (i)-(iii).  However (iv) is not needed.

 \subsection{}\label{7.7}

 It is worthwhile to give some examples to show how the method described above works.  Take $\pi:=\{\alpha_i:i=1,2,\ldots,6\}$ of type $C_6$ using the Bourbaki labelling \cite [Planche III]{B} - in particular $\alpha_6$ is the unique long simple root.  Take $\pi_1= \pi \setminus \{\alpha_3, \alpha_6\}, \pi_2=\pi$, so then $\mathfrak q =\mathfrak q_{\pi_1,\pi_2}$ is a parabolic subalgebra.  By \cite [Thm. 6.7]{J4} its canonical truncation $\mathfrak q_\Lambda$ is regular.  (This is a general fact for truncated biparabolic subalgebras of $\mathfrak g_\pi$ with $\pi$ of type $A$ or type $C$.)

 It is not easy to show that $\mathfrak q_\Lambda$ admits an adapted pair $(h,\eta)$ and even more difficult to find all of them.  Here we just assume that it does and try to show that such a pair must satisfy integrality.
 Define $\pi^{\frac{1}{2}}$ as in \ref {7.1}.  As noted above this can be any subset of $\pi \setminus \{\alpha_6\}$.  Here we just consider the case when $\pi^{\frac{1}{2}}=\{\alpha_1,\alpha_3,\alpha_5\}$, which is one of the most delicate cases.

 From the recipe in \ref {7.1} we obtain $\pi^\mathbb Z= \{\alpha_1+\alpha_2+\alpha_3,\alpha_2,\alpha_3+\alpha_4+\alpha_5,\alpha_4,2\alpha_5+\alpha_6,\alpha_6\}$.  Then from \ref {6.2} we obtain $\pi_1^\mathbb Z=\{\alpha_2,\alpha_4\}, \pi^\mathbb Z_2 =\pi^\mathbb Z$.  Since $\pi^\mathbb Z$ is of type $C$, the truncation of $\mathfrak q_\mathbb Z$ is regular and so \cite [Thm. 6.7]{J4} determines the weights and degrees of the generators of $Sy(\mathfrak q_\mathbb Z)$.  They are parametrised by the $<i_{\pi_1^\mathbb Z}i_{\pi_2^\mathbb Z}>$ orbits in $\pi^\mathbb Z$.  In this case all the orbits are singletons and we denote by $p_i$ the semi-invariant generator  defined by $\{\alpha_i\}:i=1,2,\ldots, 6$.

 The weight $\varpi_{p_i}$ of $p_i$ is given as the sum of weights of generators of $Sy(\mathfrak b_{\pi^\mathbb Z_2})$ and of $Sy(\mathfrak b^-_{\pi^\mathbb Z_1})$ as determined \cite [3.5]{J4} by the $<i_{\pi_1^\mathbb Z}i_{\pi_2^\mathbb Z}>$ orbit which defines $p_i$.  The latter are in turn sums of elements in the Kostant cascades $B_{\pi_2^\mathbb Z},B_{\pi_1^\mathbb Z}$.

 Set $\beta_i=2\alpha_i+\ldots+2\alpha_5+\alpha_6$.  The set $\{\beta_i:i=1,2,\ldots,6\}$ is just the Kostant cascade $B_\pi$ for $\pi$.  It turns out that it is also the Kostant cascade $B_{\pi^\mathbb Z}$ for $\pi^\mathbb Z$, a general property in type $C$.  There is however a subtle difference - the contribution of the Kostant cascade to the weights of $Sy(\mathfrak b_\pi)$ differs from that of $Sy(\mathfrak b_{\pi^\mathbb Z_2})$ though of course it follows the same general rule, namely that one takes partial sums of the element of the Kostant cascade according to the structure of the Dynkin diagram which of course differs for $\pi$ and $\pi^\mathbb Z$. The precise rule (for type $C$) is given by line $2$ of \cite [Table I]{J1}.

 Similarly one must calculate the weights of $Sy(\mathfrak b^-_{\pi^\mathbb Z_1})$, which in this case is rather trivial since $\pi^\mathbb Z_1$ is of type $A_1\times A_1$.

 Finally in the present case $\mathfrak h_\Gamma$ as defined in \ref {2.12} can be chosen to be the linear span of $\alpha^\vee_1,\alpha^\vee_5$.

 In the table below we list the resulting generators and their weights.  The last column describes their values on the pair $(\alpha^\vee_1,\alpha^\vee_5)$.

\bigskip
\begin {center}

\begin{tabular}{|l|l|l|}
\hline
Generator $p_i$&Weight $\varpi_{p_i}$&$(\alpha^\vee_1(\varpi_{p_i}),\alpha^\vee_5(\varpi_{p_i}))$\\

\hline
$p_1$& $\beta_1$&$(2,0)$\\
$p_2$& $\beta_1+\beta_4-\alpha_4$&$(2,1)$\\
$p_3$& $\beta_1+\beta_4+\beta_5$&$(2,2)$\\
$p_4$& $\beta_2-\alpha_2$&$(-1,0)$\\
$p_5$& $\beta_2+\beta_3$&$(-2,0)$\\
$p_6$& $\beta_2+\beta_3+\beta_6$&$(-2,-2)$\\

\hline

\end{tabular}

\end {center}

\bigskip

\begin {center} Table III

\end {center}

\bigskip

 From this table we may now calculate the generators of $Sy(\mathfrak q_\mathbb Z)^{\mathfrak h_\Gamma}$.  In this case the algebra is polynomial on generators $p_1p_5,p_1p_4^2,p_3p_6,p_2^2p_5p_6$.

 Obviously these do not satisfy the factorisation property of \ref {7.6}(iii).  The easiest way to see this is to consider their product which is $p^2_1p^2_5p_4^2p_3p^2_6p_2^2$, whilst the factorisation property implies that at least four generators should appear with exponent exactly one.  Finally we remark that in this example condition (iv) of \ref {7.6} \textit{is} satisfied!

 This excludes the case $\pi^{\frac{1}{2}}=\{\alpha_1,\alpha_3,\alpha_5\}$.  To complete the calculation for this particular example all other cases must be similarly excluded!

 \

 A second example is provided in type $C_5$ by taking $\pi_2=\{\alpha_4,\alpha_5\}, \pi_1=\pi\setminus \{\alpha_5\}$.  It this case $\mathfrak q_{\pi_1,\pi_2}$ is properly a biparabolic.  Here we shall take $\pi^{\frac{1}{2}}=\{\alpha_1,\alpha_4,\}$.

 One checks from \ref {7.1} that $\pi^\mathbb Z= \{\alpha_1+\alpha_2+\alpha_3+\alpha_4,\alpha_2,\alpha_3,\alpha_5,2\alpha_4+\alpha_5\}$.  Then from \ref {6.2} we obtain $\pi_1^\mathbb Z=\{\alpha_2,\alpha_3,\alpha_1+\alpha_2+\alpha_3+\alpha_4\}, \pi^\mathbb Z_2 = \{\alpha_5,2\alpha_4+\alpha_5\}$.  In this we may note that \ref {6.2}(iii) is satisfied.

 There are three $<i_{\pi_1^\mathbb Z}i_{\pi_2^\mathbb Z}>$  orbits which are singletons and a further pair of orbits lying in a single $<i_{\pi_1^\mathbb Z},i_{\pi_2^\mathbb Z}>$ orbit. The latter gives rise to two semi-invariants generators of the same weight (but differing degrees).

In this case $\mathfrak h_\Gamma$ as defined in \ref {7.6} is spanned by $\alpha^\vee_4$.

Using the same conventions and notations as before we tabulate the generators of $Sy(\mathfrak q_\mathbb Z)$, their weights and values on $\alpha^\vee_4$.

\bigskip
\begin {center}

\begin{tabular}{|l|l|l|}
\hline
Generator $p_i$&Weight $\varpi_{p_i}$&$\alpha^\vee_4(\varpi_{p_i})$\\

\hline
$p_1$& $2\alpha_4+\alpha_5$&$2$\\
$p_2$& $\alpha_5$&$-2$\\
$p_3$& $-(\alpha_2+\alpha_3)$&$1$\\
$p_3'$& $-(\alpha_2+\alpha_3)$&$1$\\
$p_4$& $-(\alpha_1+\alpha_2+\alpha_3+\alpha_4)$&$-1$\\

\hline

\end{tabular}

\end {center}

\bigskip

\begin {center} Table IV

\end {center}

\bigskip

From this table we may now calculate the generators of $Sy(\mathfrak q_\mathbb Z)^{\mathfrak h_\Gamma}$.  In this case the algebra is not even polynomial.


  By condition (ii) of \ref {7.6} this excludes the case $\pi^{\frac{1}{2}}=\{\alpha_1,\alpha_4\}$.  To complete the calculation for this particular example all other cases must be similarly excluded.

  \subsection{}\label{7.8}

  It is worthwhile to give an example which we can verify actually admits an adapted pair. Thus take $\pi=\{\alpha_1,\alpha_2,\alpha_3\}$ which is of the type $C_3$.  Set $\pi_1=\{\alpha_1,\alpha_2\}$ which is of type $A_2$ and $\pi_2=\pi$.  Thus $\mathfrak q:=\mathfrak q_{\pi_1,\pi_2}$ is a parabolic subalgebra of $\mathfrak g_\pi$.  The canonical truncation $\mathfrak q_\Lambda$ of $\mathfrak q$ has index $2$ and its Cartan subalgebra $\mathfrak h_\Lambda$  is spanned by $\alpha^\vee_1,\alpha^\vee_2$. We look for a presentation of an adapted pair $(h,\eta)$ with $\eta$ given by a subset $S$ of the roots of $\mathfrak q^*$ as in \ref {4.1}.  As a guess we take $-(\alpha_1+\alpha_2+\alpha_3) \in S$.  One checks that it is then enough to add one of the roots $-(2\alpha_2+\alpha_3),-\alpha_2,\alpha_1$ to $S$.  If we add the first of these then (because of the factor of two) we have a chance of finding an adapted pair not satisfying integrality. However the equations $h(\alpha_1+\alpha_2+\alpha_3)=1,h(2\alpha_2+\alpha_3)=1$, have no solution for $h \in \mathfrak h_\Lambda$!  On the other hand if we include instead $-\alpha_2$, then the corresponding equations yield the unique solution $h=3\alpha^\vee_1+2\alpha^\vee_2$ and obviously $\ad h$ has only integer eigenvalues.

  Even in this baby example it is not at all obvious if this procedure constructs all adapted pairs (up to equivalence) so this does not prove integrality.  However we can try out the technique developed above.  There are three possible choices for $\pi^{\frac{1}{2}}$ and this gives three possible choices for $\pi^\mathbb Z$.  However one may check that the latter are all conjugated under the action of the Weyl group of the Levi factor of $\mathfrak q$, hence are equivalent and we need only the consider the case $\pi^{\frac{1}{2}}=\alpha_1$.  Table V below is computed by the same method used in \ref {7.7}.  It shows that the factorisation property is not satisfied. (By contrast condition (iv) of \ref {7.6} is satisfied.)  We conclude that any adapted pair for $\mathfrak q_\Lambda$ must satisfy the integrality property.
  \bigskip
\begin {center}

\begin{tabular}{|l|l|l|}
\hline
Generator $p_i$&Weight $\varpi_{p_i}$&$\alpha^\vee_1(\varpi_{p_i})$\\

\hline
$p_1$& $\beta_1$&$2$\\
$p_2$& $\beta_2-\alpha_2$&$-1$\\
$p_3$& $\beta_2+\beta_3$&$-2$\\

\hline

\end{tabular}

\end {center}

\bigskip

\begin {center} Table V

The generators of $Sy(\mathfrak q_\mathbb Z)^{\mathfrak h_\Gamma}$ are $p_1p_2^2,p_1p_3$.

\end {center}

\bigskip

 \section {Integrality for Truncated Parabolics in Type $C$}\label{8}

\textbf{Notation.}  Recall that the Weyl group $W$ is generated by the simple reflections $s_\alpha: \alpha \in \pi$.  Given $\alpha=\alpha_j$, we shall write $s_{\alpha_j}$ simply as $s_j$.

\subsection{}\label{8.1}

   The author tried out a host of possible arguments to prove the integrality of an adapted pair the best of which was the algorithm given in \ref {7.5}, \ref {7.6}.  The fact that this only just goes through in even quite simple examples \ref {7.7}-\ref {7.8} seems to indicate that the question is a really delicate one.

  In this section we assume that $\pi$ is of type $C$.  This has the advantage that $\pi^\mathbb Z$ has only components of type $A$ and of type $C$.  Consequently via \cite [Thm.6.7]{J4} we can conclude that $Sy(\mathfrak q_\mathbb Z)$ is polynomial with the generators having known weights (and degrees).

  In addition to the above we assume that $\mathfrak q=\mathfrak q_{\pi_1,\pi_2}$ is a parabolic subalgebra of $\mathfrak g_\pi$, that is $\pi_2=\pi$.  Considering that the weights of generators are given through orbit sums \cite [4.6]{J3.5}, \cite [Thm. 6.7]{J4}, this has the technical advantage that these orbits are reduced to at most two elements in the parabolic case for type $C$, whilst in the biparabolic case they are meanders and essentially indescribable.   Further advantages is that it is simpler to compute $\mathfrak h_\Gamma$ (cf \ref {7.5}, \ref {8.2.3}) and the overset $\tilde{\pi}$ introduced in \cite [4.5]{J2} can just be taken to be $\pi$ itself.

  We may conclude that in the above setting there is a bijection $\mathscr P$ (resp. $\mathscr P_\mathbb Z$) from the set of $<i_{\pi_1}i_{\pi_2}>$ (resp. $<i_{\pi_1}^\mathbb Z i_{\pi_2}^\mathbb Z >$) orbits in $\pi$ (resp. $\pi^\mathbb Z$) to a set of generators of the polynomial algebra $Sy(\mathfrak q)$ (resp. $Sy(\mathfrak q_\mathbb Z)$). Moreover the weights of the generators are given by $<i_{\pi_1}i_{\pi_2}>$ (resp. $<i_{\pi_1}^\mathbb Z i_{\pi_2}^\mathbb Z >$) orbit sums with at most two terms.

  Our aim is to prove the following theorem in this section.

  \begin {thm}  Let $\mathfrak q$ be a parabolic subalgebra in type $C$.  An adapted pair for $\mathfrak q$ (if it exists) satisfies integrality.
 \end {thm}
   The proof is concluded in subsection \ref {8.12}.   In some sense the key to the proof is an identification of the $\dim \mathfrak h_\Gamma$ the additional generators of $Sy(\mathfrak q_\mathbb Z)$ over $Sy(\mathfrak q)$. These are generators which must map to non-zero scalars under the restriction map $\psi$ defined in \ref {2.4}.  A priori it is completely unclear how to distinguish them from those generators of $Sy(\mathfrak q_\mathbb Z)$ which map to linear elements under $\psi$.  Of course were we to knew the ad-semisimple element of the adapted pair $(h,\eta)$ then this distinction would become easy.  Namely a generator $q$ in the first set has eigenvalue $\deg q$ by Lemma \ref {2.7} whilst a generator $\hat{p}$ in the second set has a non-positive eigenvalue by Remark \ref {2.8}.

 \subsection{}\label{8.2}

 Assume that $\pi$ is of type $C_n$, let $\pi_1$ be a proper subset of $\pi$ and set $\pi_2=\pi$.  Set $\mathfrak q=\mathfrak q_{\pi_1,\pi_2}$ which is a parabolic subalgebra of $\mathfrak g_\pi$ in standard form, that is containing the Borel subalgebra with roots in $\mathbb N\pi$.

 Set $I=\{1,2,\ldots,n\}$. Given $i\leq j \in I$, let $[i,j]$ denote the subset $\{\alpha_i,\alpha_{i+1},\ldots,\alpha_j\}$ of $\pi$ and let $s[i,j]$ denote the sum $\alpha_i+\alpha_{i+1}+\ldots+\alpha_j$ which is a positive root.   It may also be convenient to denote the simple root $\alpha_i: i \in I$ by $i$.  Set $\beta_i:=2\alpha_i+2\alpha_{i+1}+\ldots + 2\alpha_{n-1}+\alpha_n:i \in I$. These are all long roots (in type $C$). Again $\beta_i$ is the $i^{th}$ element of the Kostant cascade $B_\pi$, itself a maximal set of strongly orthogonal roots \cite [Sect. 2, Table I]{J1}.

 Define $\pi^{\frac{1}{2}}$ as in \ref {7.1} and recall (\ref {7.1.3}) that in type $C$ it may be any subset of $[1,n-1]$. We write $\pi^{\frac{1}{2}}= \{i_1,i_2,\ldots,i_r\}$, with $i_j \in [1,n-1]$ and strictly increasing.

  As in \ref {7.2}, let $\pi_{1,u}:u \in U$ denote the set of connected components of $\pi_1$. By Lemma \ref {7.2} we can and do assume that

 $$|\pi^{\frac{1}{2}}\cap \pi_{1,u}|\leq 1, \forall u \in U. \eqno{(*)}$$

  \subsubsection{}\label{8.2.1}

   Through \ref {7.1.1} and \ref {7.1.3} we may compute the subset $\pi^\mathbb Z \subset \mathbb N\pi$ of simple roots for $\mathfrak g_\mathbb Z$ lying in $\mathbb N \pi$.  One finds that it has two connected components.  They are denoted by $\pi^{\mathbb Z,\ell}, \pi^{\mathbb Z,r}$ and are described below.

 Set $o_r=1,e_r=0$
 (resp. $o_r=0,e_r=1$) if $r$ is odd (resp. even). Then
 $$\pi^{\mathbb Z,\ell}:=\{[1,i_1-1],s[i_1,i_2],[i_2+1,i_3-1],\ldots,e_r[i_r+1,n],o_r \beta_{i_r} \},$$
  and
  $$\pi^{\mathbb Z,r}:=\{[i_1+1,i_2-1],s[i_2,i_3],[i_3+1,i_4-1],\ldots,o_r[i_r+1,n],e_r\beta_{i_r}\}.$$  Thus $\pi^{\mathbb Z,\ell},\pi^{\mathbb Z,r}$ both contain exactly one long root, which is either $\alpha_n$ or $\beta_{i_r}$, and so both are of type $C$ (possibly including type $C_1$) and hence admit a total order $<$ through the numbering of roots in the Dynkin diagram given as in Bourbaki \cite [Planche III]{B}. This numbering makes the long root the largest root for this ordering.   Again the $i_{\pi^\mathbb Z}$ orbits in $\pi^\mathbb Z$ are all reduced to singletons.

 One may compute the form of the Kostant cascade for each of the above two connected subsets of $\pi^\mathbb Z$, using the general form given (for type $C$) in \cite [Table I]{J1}.

 For this we set $i_0=0,i_{r+1}=n$ and write $I$ as the union of the two disjoint subsets
 $$I^\ell:=\bigsqcup_{j=0}^{[r/2]}[i_{2j}+1,i_{2j+1}], \quad I^r:=\bigsqcup_{j=0}^{[(r-1)/2]}[i_{2j+1}+1,i_{2(j+1)}].\eqno {(*)}$$

 One checks that
 $$B_{\pi^{\mathbb Z,\ell}} =\{\beta_t\}_{t\in I^\ell}, \quad B_{\pi^{\mathbb Z,r}} =\{\beta_t\}_{t\in I^r}. \eqno {(**)}$$

 In particular if $r$ is even (resp. odd), then $B_{\pi^{\mathbb Z,\ell}}$ terminates in $\beta_n$ (resp. $\beta_{i_r}$) and $B_{\pi^{\mathbb Z,r}}$ terminates in $\beta_{i_r}$ (resp. $\beta_n$), compatible with our previous description of the connected components of $\pi^\mathbb Z$.

 As in our examples (\ref {7.7}, \ref {7.8}) we find that $B_\pi$ is their disjoint union.

  \subsubsection{}\label{8.2.2}

 Recall (\ref {6.2}) that $\pi^\mathbb Z_1=\pi^\mathbb Z \cap \mathbb N\pi_1$.  One checks that $$\pi^\mathbb Z_1=(\pi_1\cap (\pi\setminus \pi^{\frac{1}{2}}))\cup \{s[i_t,i_{t+1}] \ | \ [i_t,i_{t+1}] \subset \pi_1\}_{t=1}^r\cup \{\beta_{i_r} \ | \ [i_r,n] \subset \pi_1\}. $$

 In view of \ref {8.2}$(*)$, this expression simplifies to give

 $$\pi^\mathbb Z_1=(\pi_1\cap (\pi\setminus \pi^{\frac{1}{2}})) \cup \{\beta_{i_r} \ | \ [i_r,n] \subset \pi_1\}. \eqno {(***)} $$

  Recall that we denote by $\{\pi_{1,u}\}_{u \in U}$ the set of connected components of $\pi_1$. The linearity of the Dynkin diagram (in type $C$) induces a linear order $<$ on $U$.  Since an interval $[i_t,i_{t+1}]$ can belong to at most one connected component of $\pi_1$ it follows that if we  replace $\pi_1$ by $\pi_{1,u}$ in the right hand side of $(***)$, then the left hand side defines a subset $\pi_{1,u}^\mathbb Z$ of $\pi_1^\mathbb Z$ which in turn is a disjoint union of the $\pi_{1,u}^\mathbb Z: u \in U$.  Finally for all $u \in U$, set $$\pi_{1,u}^{\mathbb Z, \ell} =  \pi^{\mathbb Z, \ell} \cap \pi_{1,u}^\mathbb Z,\quad \pi_{1,u}^{\mathbb Z,r} =  \pi^{\mathbb Z, r} \cap \pi_{1,u}^\mathbb Z.$$

 One checks that these subsets of $\pi^\mathbb Z_1$ are connected whilst by the preceding paragraph their union is $\pi_1^\mathbb Z$.  On the other hand roots in different subsets are clearly orthogonal and so the non-empty sets $\pi_{1,u}^{\mathbb Z, \ell},\pi_{1,u}^{\mathbb Z,r}: u \in U$ form the set of connected components of $\pi_1^\mathbb Z$.

   \subsubsection{}\label{8.2.3}

Define $\mathfrak h_\Gamma$ as in \ref {7.5}.

\begin {lemma} $\mathfrak h_\Gamma$ can be chosen to be any complement to $\mathfrak h_{\pi_1^\mathbb Z}$ in $\mathfrak h_{\pi_1}$
\end {lemma}

\begin {proof}   Since $\pi_2$ is of type $C$, the involution $i_{\pi_2}$ is trivial.  Thus the $<i_{\pi_1}i_{\pi_2}>$ orbits (which determine the generators of $Sy(\mathfrak q)$, their weights and degrees) are just orbits under the involution $i_{\pi_1}$ and are reduced to one or to two elements.  Moreover they are all $<i_{\pi_1},i_{\pi_2}>$ orbits, so by \cite [Prop. 5.9(i)]{J3.5} the Cartan subalgebra of the truncation of $\mathfrak q$ is just $\mathfrak h_{\pi_1}$.

Since $i_{\pi_2^\mathbb Z}$ is trivial, the $<i_{\pi_1^\mathbb Z}i_{\pi_2^\mathbb Z}>$ orbits are just orbits under the involution $i_{\pi_1^\mathbb Z}$ and are reduced to one or to two elements.  Moreover they are all $<i_{\pi_1^\mathbb Z},i_{\pi_2^\mathbb Z}>$ orbits, so by \cite [Prop. 5.9(i)]{J3.5} the Cartan subalgebra of the truncation of $\mathfrak q_\mathbb Z$ is just $\mathfrak h_{\pi_1^\mathbb Z}$.

Hence the assertion.

\end {proof}

 \subsection{}\label{8.3}

 Recall the discussion preceding Theorem \ref {8.1}.

 A generator of $Sy(\mathfrak q)$ (resp. $Sy(\mathfrak g_\mathbb Z)$) is given by an $<i_{\pi_1}i_{\pi_2}>$  orbit $\Gamma$ in $\pi$ (resp. an $<i_{\pi^\mathbb Z_1}i_{\pi^\mathbb Z_2}>$ orbit $\Gamma^\mathbb Z$ in $\pi^\mathbb Z$). We denote it by $p_\Gamma$ (resp. $p_{\Gamma^\mathbb Z}$) and its weight by $\delta_\Gamma$ (resp. $\delta_{\Gamma^\mathbb Z}$).   There is a description of these generators as the image of special elements in the Hopf dual of $U(\mathfrak q)$ under a linear isomorphism of the latter onto $S(\mathfrak q)$ (see \cite [Prop. 7.5, Cor. 7.6]{J2}). However only their weights will be needed here.

 Our present goal is to calculate the (integers) $\alpha^\vee (\delta_{\Gamma^\mathbb Z})$, for all $\alpha \in \pi_1\cap \pi^{\frac{1}{2}}$ and for all $<i_{\pi_1^\mathbb Z}i_{\pi_2^\mathbb Z}>$ orbits $\Gamma^\mathbb Z$ in $\pi^\mathbb Z$. In Tables III-V, these are just the entries in the last column.  As in the examples we shall show that the criteria of \ref {7.6} are not satisfied.  Up to some mild combinatorics this is all perfectly straightforward though a little tedious.  Regrettably we did not find any shortcuts.

  For all $i\in \{1,2\}$, take $\alpha \in \pi_i$ (resp. $\alpha \in \pi_i^\mathbb Z$) and let $\varpi_{\alpha}^{\pi_i}$ (resp. $\varpi_{\alpha}^{\pi_i^\mathbb Z}$) denote the fundamental weight in the linear span of the elements of $\pi_i$ (resp. $\pi_i^\mathbb Z$).  Of course since we are taking $\pi_2=\pi$ one has $\varpi_\alpha^{\pi_2}=\varpi_\alpha$, for all $\alpha \in \pi$.  In case $\alpha \notin \pi_i$ (resp. $\alpha \notin \pi_i^\mathbb Z$) we set $\varpi_{\alpha}^{\pi_i}=0$ (resp. $\varpi_{\alpha}^{\pi_i^\mathbb Z}=0$).

  Let $\Gamma$ be an $<i_{\pi_1}i_{\pi_2}>$ orbit. Since all connected components are of type $A$ or of type $C$ it follows from \cite [Thm. 6.7]{J4} that  $\delta_\Gamma$ is given by \cite [4.6]{J3.5}.  However it is more convenient to effect the transformation on this expression carried out in \cite [Lemma 3.4]{J4}.  In addition there is a simplification here since either $\Gamma$ is contained in $\pi_1\subset \pi_2=\pi$, or $\Gamma$ is a singleton orbit contained in $\pi_2\setminus \pi_1$ and in the latter case we are setting $\varpi_\alpha^{\pi_1}=0$ - see above.  A similar remark applies if $\Gamma^\mathbb Z$ is an $<i_{\pi^\mathbb Z_1}i_{\pi^\mathbb Z_2}>$  orbit. With these hints one checks that

  $$\delta_\Gamma= \sum_{\alpha \in \Gamma} 2(\varpi_\alpha-\varpi_\alpha^{\pi_1}), \quad \delta_{\Gamma^\mathbb Z}= \sum_{\alpha \in \Gamma^\mathbb Z} 2(\varpi_\alpha^{\pi_2^\mathbb Z}-\varpi_\alpha^{\pi_1^\mathbb Z}). \eqno{(*)}$$

   One may remark that by $(*)$ the weight of a semi-invariant of $S(\mathfrak q)$ (resp. $S(\mathfrak q_\mathbb Z)$) vanishes on $\alpha^\vee:\alpha \in \pi_\cap$ (resp. $\alpha^\vee:\alpha \in \pi^\mathbb Z_\cap$) as it should.

  \subsubsection{}\label{8.3.1}

  Recall the definition of $I$ given in \ref {8.2}. Take $J \subset I$.  Call $j \in I \setminus J$ a left (resp. right) neighbour of $J$ if $j+1 \in J$ (resp. $j-1 \in J$) and a neighbour if either (or both) hold.

   Recall the decomposition of $I$ as a disjoint union $I=I^\ell\sqcup I^r$ into two parts given in \ref {8.2.1}$(*)$.  Given $i \in I$, let $I(i)$ denote the part of I containing $i$.

  Take $\alpha \in \pi_2^\mathbb Z \setminus \{\beta_{i_r}\}$.  We may write $\alpha=s[i,j]$ with $i\leq j$ and we set $I(\alpha)=I(i)$.
  Slightly surprisingly $2\varpi_\alpha^{\pi_2^\mathbb Z}$ is the subsum of elements in the Kostant cascade indexed by $I(i)$ terminating in $\beta_i$. More explicitly
  $$2\varpi_{s[i,j]}^{\pi_2^\mathbb Z}=\sum_{t \in I(i), t \leq i}\beta_t , \eqno {(*)}$$
  that is to say compared to $2\varpi_i=\sum_{t=1}^i\beta_t$, we omit those $\beta_t: t \in I\setminus I(i)$.

  Recall that we set $i_{r+1}=n$.  We may include the case $\alpha = \beta_r$ in the above by viewing $\beta_r$ as $s[i_r,i_{r+1}]$, noting that $(*)$ still holds.  The possible confusion that this can cause does not arise because $s[i_r,n] \notin \pi_2^\mathbb Z$.

  We shall need to calculate $\alpha_k^\vee(2\varpi_{s[i,j]}^{\pi_2^\mathbb Z})$ for all $\alpha_k \in \pi^{\frac{1}{2}}$ and for all $s[i,j] \in \pi_2^\mathbb Z$.

   \begin {lemma} Suppose $\alpha_k \in \pi^{\frac{1}{2}}$. Then
   $$\alpha^\vee_k(2\varpi_{s[i,j]}^{\pi_2^\mathbb Z})=\left\{\begin{array}{ll}-2& : \text {If} \ k \ \text{is a left neighbour to} \  I(i), \\
   2&:k=i,\\
0& :otherwise.\\
\end{array}\right.$$
   \end {lemma}

   \begin {proof}  For all $k,\ell \in \{1,2,\ldots,n\}$, one has
   $$\alpha^\vee_k(\beta_\ell)=\left\{\begin{array}{ll}2& :\ell=k,\\
-2& :\ell=k+1,\\
0&:otherwise,\\
\end{array}\right.$$
from which the assertion follows.

   \end {proof}

\textbf{Remark}.  Since $\pi_2^\mathbb Z \cap \pi^{\frac{1}{2}}=\phi$, the situation $k=i$ in the above exactly arises if $k=i=i_t,j=i_{t+1}$ for some $t \in \{1,2,\ldots,r\}$.

\subsubsection{}\label{8.3.2}

 For all $\alpha \in \pi_1^\mathbb Z$, set

  $$\varpi_{s(\alpha)}^{\pi_1^\mathbb Z}=\left\{\begin{array}{ll}\varpi_{\alpha}^{\pi_1^\mathbb Z}& :\alpha=i_{\pi_1^\mathbb Z}(\alpha),\\
\varpi_{\alpha}^{\pi_1^\mathbb Z}+\varpi_{i_{\pi_1^\mathbb Z}(\alpha)}^{\pi_1^\mathbb Z}& :\alpha \neq i_{\pi_1^\mathbb Z}(\alpha).\\
\end{array}\right.$$

By \ref {8.3}$(*)$ such an expression occurs in the description of $\delta_{\Gamma^\mathbb Z}$.

 Recall (\ref {8.2.2}) and let $\pi_{1,u}^{\mathbb Z,v}:v \in \{\ell,r\}$ be a connected component of $\pi_{1,u}^\mathbb Z$ of which there are at most two.

  Both are of type $A$ if $\alpha_n \notin \pi_{1,u}$ and both are of type $C$ (possibly of type $C_1$) otherwise, that is if $\alpha_n \in \pi_{1,u}$, equivalently if $\pi_{1,u}$ is of type $C$.
  For type $C$ the $i_{\pi_1^\mathbb Z}$ orbits are singletons.

  Recall again that $\pi_2^\mathbb Z \cap \pi^{\frac{1}{2}}=\phi$.

 \begin {lemma}

 Take $\alpha_k \in \pi^{\frac{1}{2}}, \alpha \in \pi_{1,u}^{\mathbb Z,v} $.

 \

 (i)  If $\alpha_k$ is not a neighbour of $\pi_{1,u}^{\mathbb Z,v}$, then $\alpha^\vee_k(\varpi_{s(\alpha)}^{\pi_1^\mathbb Z})=0$.

\

 (ii) Suppose $\alpha_n \notin \pi_{1,u}$.  If $\alpha_k$ is a neighbour of $\pi_{1,u}^{\mathbb Z,v}$, then

$$\alpha^\vee_k(2\varpi_{s(\alpha)}^{\pi_1^\mathbb Z})=\left\{\begin{array}{ll}-1& :\alpha=i_{\pi_1^\mathbb Z}(\alpha), \\
-2& :\alpha \neq i_{\pi_1^\mathbb Z}(\alpha).\\
\end{array}\right.$$

\

(iii)  Suppose $\alpha_n \in \pi_{1,u}$.
If $\alpha_k$ is a left neighbour of $\pi_{1,u}^{\mathbb Z,v}$, then

$$\alpha^\vee_k(2\varpi_{\alpha}^{\pi_1^\mathbb Z})=-2.$$

\

\end {lemma}

\begin {proof}  If $\alpha_n \notin \pi_{1,n}$, then $\pi_{1,u}^{\mathbb Z,v}$ is of type $A$ and moreover $\pi_{1,u}^{\mathbb Z,v} \subset \pi_{1,u}$. One checks that $\varpi_{s(\alpha)}^{\pi_1^\mathbb Z}$ (resp. $2\varpi_{s(\alpha)}^{\pi_1^\mathbb Z}$) is just the sum of the elements in the Kostant cascade for this component starting from its unique highest root and ending with the sum of the simple roots lying between $\alpha$ and $i_{\pi_1^\mathbb Z}(\alpha)$ if $\alpha \neq i_{\pi_1^\mathbb Z}(\alpha)$ (resp. $\alpha = i_{\pi_1^\mathbb Z}(\alpha)$).  From this the assertion is easily verified.

If $\alpha_n \in \pi_{1,n}$, then $\pi_{1,u}^{\mathbb Z,v}$ is of type $C$ and contains either $\alpha_n$ or $\beta_{i_r}$. In the first case $\pi_{1,u}^{\mathbb Z,v}=[i_r+1,n]$ and the hypothesis of (ii) holds exactly when $k=i_r$. Then $2\varpi_{\alpha_t}^{\pi_1^\mathbb Z}:t \in [i_r+1,n]$ is the sum of the elements in the Kostant cascade for this component starting from the unique highest root $\beta_{i_r+1}$ and ending in $\beta_t$. Since the value of $\alpha^\vee_{i_r}$ equals $-2$ on the first of these and is zero on the remainder, the assertion follows.
In the second case $\pi_{1,u}^{\mathbb Z,v}=[s,i_r-1]\cup \{\beta_{i_r}\}$, where $s$ is the unique smallest element of $\pi_{1,u}$.   Then $2\varpi_{\alpha_t}^{\pi_1^\mathbb Z}:t \in [s,i_r-1]$ is the sum of the elements in the Kostant cascade for this component starting from the unique highest root $\beta_{i_s}$ and ending in $\beta_t$, whilst $2\varpi_{\beta_{i_r}}^{\pi_1^\mathbb Z}$ is the sum of the elements in the Kostant cascade for this component starting from the unique highest root $\beta_{i_s}$ and ending in $\beta_{i_r}$.  In this case the hypothesis of (ii) exactly holds when $k=i_{r-1},s=i_{r-1}+1$, from which (iii) obtains.

\end {proof}

\textbf{Remark}.  When $\alpha_n \in \pi_{1,u}$ and $\pi_{1,u}^{\mathbb Z,v}=[s,i_r-1]\cup \{\beta_{i_r}\}$ some possible cases have not been considered.  However by Lemma \ref {8.4} these will not be needed.

\subsection{}\label{8.4}

Observe that $\beta_{i_r} \in \pi_1^\mathbb Z$ if and only if there is a connected component of $\pi_1$ containing $[i_r,n]$.  Moreover
$$\text{If} \ \beta_{i_r} \in \pi_1^\mathbb Z, \ \text {then} \ \alpha_{i_r}^\vee=\beta_{i_r}^\vee, \mod \mathfrak h_{\pi_1^\mathbb Z}. \eqno {(*)}$$

Since it is always true that $\beta_{i_r} \in \pi_2^\mathbb Z$ we obtain $\beta_{i_r}^\vee \in \mathfrak h_{\pi_1^\mathbb Z}$ and hence by $(*)$ that $\alpha_{i_r}^\vee \in \mathfrak h_{\pi_1^\mathbb Z}$ in this case.

Set
$$\overline{\pi}^{\frac{1}{2}}=\left\{\begin{array}{ll}\pi^{\frac{1}{2}}\setminus \{i_r\}& :\beta_{i_r} \in \pi_1^\mathbb Z ,\\
\pi^{\frac{1}{2}}& :otherwise.\\
\end{array}\right.$$


 \begin {lemma}  One may take $\mathfrak h_\Gamma$ to be the linear span of the $\alpha^\vee: \alpha \in \overline{\pi}^{\frac{1}{2}}\cap \pi_1$.
 \end {lemma}

 \begin {proof}  By \ref {8.2.2}$(***)$.
$$\pi_1^\mathbb Z= \pi_1 \setminus (\pi^{\frac{1}{2}}\cap \pi_1) \cup (\pi_1^\mathbb Z \cap\{\beta_{i_r}\}).\eqno {(**)}$$

 Then the assertion follows from the above remarks combined with  \ref {8.2.3}.
 \end {proof}

 \subsection{}\label{8.5}

 Recall the notation of \ref {1.2}.  Let $s$ (resp. $s^\mathbb Z$) denote the number of $i_{\pi_1}$ (resp. $i_{\pi_1^\mathbb Z}$) orbits in $\pi_1$ (resp. $\pi_1^\mathbb Z$).

 \begin {lemma}  $r\ell(\mathfrak q_\mathbb Z)-r\ell(\mathfrak q)=(s^\mathbb Z -s) +|\overline{\pi}^{\frac{1}{2}}\cap \pi_1|$.
 \end {lemma}

 \begin {proof}  Recall that $i_{\pi_2},i_{\pi_2}^\mathbb Z$ are both trivial. Thus by the last paragraph preceding Theorem \ref {8.1} the left hand side above is just the number of  $i_{\pi_1^\mathbb Z}$ orbits in $\pi^\mathbb Z$ minus the number of $i_{\pi_1}$ orbits in $\pi$.  Recall further that $i_{\pi_1^\mathbb Z}$ (resp. $i_{\pi_1}$) is extended by the identity on $\pi^\mathbb Z\setminus \pi_1^\mathbb Z$ (resp.  $\pi\setminus \pi_1$) where a priori it is not defined.  Thus the first term in the right hand side of must be supplemented by the term $|\pi^\mathbb Z\setminus \pi_1^\mathbb Z|-|\pi\setminus \pi_1|$.  Since $|\pi^\mathbb Z|=|\pi|$ by definition of a regular integral pair, this expression equals $|\pi_1| - |\pi_1^\mathbb Z|$ which by \ref {8.2}$(***)$ is just $|\overline{\pi}^{\frac{1}{2}}\cap \pi_1|$, as required.
 \end {proof}

 \subsection{}\label{8.6}

 We now compute the first term $(s^\mathbb Z -s)$ occurring in the right hand side above.  Here we compute the contribution $(s_u^\mathbb Z -s_u)$ from each connected component $\pi_{1,u}:u \in U$.

Suppose that $\pi_{1,u}\cap \pi^{\frac{1}{2}} = \phi$, for some $u \in U$.  In this case one checks that $\pi_{1,u}^\mathbb Z= \pi_{1,u}$.  Thus obviously $(s_u^\mathbb Z -s_u)=0$.
Again the $<i_{\pi_1^\mathbb Z}i_{\pi_2^\mathbb Z}>$ and $<i_{\pi_1}i_{\pi_2}>$ orbits in $\pi_{1,u}$ coincide.

 Thus the number of generators of $Sy(\mathfrak q)$ and of $Sy(\mathfrak q_\mathbb Z)$ originating from $\pi_{1,u}$ coincide.  We remark that the weights do not unless $\alpha_1 \in \pi_{1,u}$.  However this will not be used.

Suppose that $\pi_{1,u}\cap \pi^{\frac{1}{2}} \neq \phi$.  Then by our assumption in \ref {8.2}$(*)$ one has $$|\pi_{1,u}\cap \pi^{\frac{1}{2}}|=1.\eqno {(*)}$$

It follows from the description of $\pi^\mathbb Z$ given in \ref {8.2} that $|\pi_{1,u}|-|\pi^\mathbb Z_{1,u}|\leq 1$.  Moreover one checks that $\pi^\mathbb Z_{1,u}$ admits at most two connected components.

In the above there are two cases to consider.  This will be done in the next two subsections.  Set $\overline{U}:=\{u \in U| \ |\pi_{1,u}\cap \pi^{\frac{1}{2}}|=1\}$.

\subsubsection{}\label{8.6.1}

 Suppose $\alpha_n \in \pi_{1,u}$. Then $\pi_{1,u}$ is of type $C$. The hypothesis of \ref {8.5} implies that $\alpha_{i_r} \in \pi_{1,u}$.  Thus $\beta_{i_r} \in \pi^\mathbb Z_{1,u}$.

 The fact that the ``extra" root $\beta_{i_r}$ belongs to $\pi^\mathbb Z_{1,u}$ has the consequence that $|\pi_{1,u}^\mathbb Z|= |\pi_{1,u}|$.  In this case the $i_{\pi_1}$ (resp. $i_{\pi^\mathbb Z_1}$) orbits on $\pi_1$ (resp. $\pi_1^\mathbb Z$) are trivial. Thus again $(s_u^\mathbb Z -s_u)=0$.

 Thus the number of generators of $Sy(\mathfrak q)$ and of $Sy(\mathfrak q_\mathbb Z)$ originating from $\pi_{1,u}$ coincide. We remark that their weights do not but again this will not be used.

\subsubsection{}\label{8.6.2}

 Suppose $\alpha_n \notin \pi_{1,u}$. Then $\pi_{1,u}$ is of type $A$ and $\pi_{1,u}^\mathbb Z$ has at most two connected components and both are of type $A$, whilst $|\pi_{1,u}|-|\pi^\mathbb Z_{1,u}|= 1$.

 By our assumption \ref {8.5}$(*)$,  $\pi^{\frac{1}{2}}\cap \pi_{1,u}$ is a singleton $\alpha_{t}$.  Moreover
 $$\pi_{1,u}^\mathbb Z= \pi_{1,u}\setminus \{\alpha_t\}. \eqno{(*)}$$

  Set $k =|\pi_{1,u}|$. Then $\pi_{1,u}$ is of type $A_k$. Again by $(*)$ it follows that $\pi_{1,u}^\mathbb Z$ is of type $A_m\times A_n$ with $m+n=k-1$.

  \begin {lemma}  For all $u \in \overline{U}$ with $\pi_{1,u}$ of type $A$ one has
  $$s_u^\mathbb Z-s_u=\left\{\begin{array}{ll}-1& :m,n \in 2\mathbb N ,\\
0& :otherwise.\\
\end{array}\right.$$

  \end {lemma}

  \begin {proof}

   Recall that if $\pi'$ is of type $A_t$, then $i_{\pi'}$ has $\frac {[t+1}{2}]$ orbits on $\pi'$. Thus $\pi_{1,u}$ admits $[\frac{k+1}{2}]$ orbits for the action of $i_{\pi_1}$ and $\pi_{1,u}^\mathbb Z$ admits $[\frac{m+1}{2}]+[\frac{n+1}{2}]$ orbits for the action of $i_{\pi_1^\mathbb Z}$. These two expressions are equal if $m$ and $n$ are not both even.  Otherwise the former exceeds the latter by $1$. Hence the assertion of the lemma.

 \end {proof}

 \subsection{}\label{8.7}

 We conclude from \ref {8.5} and {8.6} above that
 $$r\ell(\mathfrak q_\mathbb Z)-r\ell(\mathfrak q)\leq \dim \mathfrak h_\Gamma,\eqno {(*)}$$
 with equality if and only if the case $m,n \in 2\mathbb N$ is excluded in Lemma \ref {8.6.2}, for all $u \in \overline{U}$ with $\pi_{1,u}$ of type $A$.

 On the other hand the $\rk \Gamma \leq r\ell(\mathfrak q_\mathbb Z)-r\ell(\mathfrak q)$ so \ref {7.6}(ii) enforces $m,n \in 2\mathbb N$ to be excluded in Lemma \ref {8.6.2} and that the $\rk \Gamma=r\ell(\mathfrak q_\mathbb Z)-r\ell(\mathfrak q)$.  (We shall not use this latter fact.)

   Again by \ref {7.6} (i) the overall number of generators must strictly rise, that is we require
  $$\dim \mathfrak h_\Gamma =r\ell(\mathfrak q_\mathbb Z) -r\ell(\mathfrak q) >0. \eqno{(**)}$$

\subsection{}\label{8.8}

Recall that $\overline{U}:=\{u \in U| \ |\pi_{1,u}\cap \pi^{\frac{1}{2}}|=1\}$.

Given $u \in U$ we may write $\pi_{1,u}$ as $[\ell_u,k_u]$, for some $1\leq \ell_u\leq k_u \leq n$. If $u \in \overline{U}$ then in addition $\ell_u\leq i_u\leq k_u$. Moreover by the penultimate observation in \ref {8.7} we may assume that $i_u-\ell_u, k_u-i_u$ are not both even.  This forces $k_u-\ell_u \geq 1$.

If $u \in U$ is the last element of $U$, we set $u^+:=n+1$.  Otherwise we let $u^+$ denote its subsequent element of $U$.   Then $k_u < \ell_{u^+}$ unless $u$ is the largest element of $U$ and $\alpha_n \in \pi_{1,u}$.

Recall \ref {8.1} and let $\mathscr {IP}_{\mathbb Z,u}$ denote the image of $\mathscr P_\mathbb Z$ restricted to the set of
 $<i_{\pi_2^\mathbb Z}i_{\pi_1^\mathbb Z}>$ orbits in $\pi_{1,u}=\pi_{1,u}^\mathbb Z$.  These orbits have at most two elements. Both lie in the same subset (either $I^\ell$ or $I^r$) of $I$ and given $p \in \mathscr I \mathscr P_{\mathbb Z,u}$ we denote this subset by $I(p)$ and by $|p|$ denote the number of elements in that orbit.  If $k$ is the smaller element in the orbit we denote the corresponding element of $\mathscr I \mathscr P_{\mathbb Z,u}$ by $p_u^k$.

Suppose $u \in \overline{U}$.

Let $\mathscr I \mathscr P_{\mathbb Z,u}$ denote the set of generators $Sy(\mathfrak q_\mathbb Z)$ defined by the $<i_{\pi_2^\mathbb Z}i_{\pi_1^\mathbb Z}>$ orbits lying in $\pi_{1,u}^\mathbb Z$, together with the singleton orbits $\{s[i_u,i_{u+1}],\{\alpha_k\}_{k=k_u+1}^{\ell_{u^+}-1}\}$.

The generator defined by the singleton orbit $\{s[i_u,i_{u+1}]\}$ will be denoted simply as $p_u^{i_u}$. The remaining elements of $\mathscr I\mathscr P_{u,\mathbb Z}$ are in bijection with the $i_{\pi_1^\mathbb Z}$ orbits in $\{\ell_u, \ell_u+1, \ldots, \ell_{u+1}-1\}\setminus\{i_u\}$.  Both lie in the same subset  of $I$ and given $p \in \mathscr I \mathscr P_{\mathbb Z,u}$ we denote this subset by $I(p)$ and by $|p|$ denote the number of elements in that orbit which is at most two.  If $k$ is the smaller element in the orbit we denote the corresponding element of $\mathscr I \mathscr P_{\mathbb Z,u}$ by $p_u^k$.

Given $p \in \mathscr I \mathscr P_{\mathbb Z,u}: u \in U$, let $\varpi_p$ denote its weight.

In order to apply \ref {7.6} we must compute the matrix $\alpha_{i_u}^\vee(\varpi_p): u \in \overline{U}, p \in \mathscr {IP}_{\mathbb Z,u'}, u' \in U$.

Recall \ref {8.3}$(*)$.  We may write $\varpi_p=\varpi^+_p - \varpi^-_p$, where the first (resp. second) term comes from the fundamental weights corresponding to $\pi^\mathbb Z_2$ (resp. $\pi^\mathbb Z_1$) as given by decomposition in the right hand side \ref {8.3}$(*)$.

\begin {lemma} Take $u \in \overline{U}, u' \in U$ distinct.  Then for all $p \in \mathscr {IP}_{\mathbb Z,u'}$ one has

\

(i)  \quad  $\alpha^\vee_{i_u}(\varpi^+_p)=\left\{\begin{array}{ll}-2|p|& : \text{If} \ i_u \  \text{is a left neighbour of} \ I(p),\\
0& :\text{otherwise}.\\
\end{array}\right.$

\

(ii) \quad  $\alpha^\vee_{i_u}(\varpi^-_p)=0$.
\end {lemma}

\begin {proof} By \ref {8.3}$(*)$ one has $\varpi_p^+=\sum_{\alpha \in \mathscr P^{-1}(p)}2\varpi_\alpha^{\pi_2^\mathbb Z}$.
Then (i) follows from Lemma \ref {8.3.1} and the definition of $I(p)$.

 By \ref {8.3}$(*)$ one has $\varpi_p^-=\sum_{\alpha \in \mathscr P^{-1}(p)\cap \pi_1^\mathbb Z}2\varpi_\alpha^{\pi_1^\mathbb Z}$.  This equals zero if $\mathscr P^{-1}(p)\cap \pi_1^\mathbb Z=\phi$.  Otherwise $\mathscr P^{-1}(p)\subset \pi_1^\mathbb Z$ and for any $\alpha \in \mathscr P^{-1}(p)$ this expression equals $2\varpi_{s(\alpha)}$ by the definition (\ref {8.3.2}) of the latter.

 By Lemma \ref {8.3.2}(i) it suffices to show that $i_u$ is not a neighbour to either one of the connected components of $\pi_{1,u}^\mathbb Z$.  If it were, then we would have $i_u \in \pi_{1,u'}$ contradicting that $u,u'$ are distinct.
\end {proof}

\subsection{}\label{8.9}

We complement the above result by the following

\begin {lemma}  For all $u \in U, p_u^k \in \mathscr I \mathscr P_{\mathbb Z,u}: k \in \{\ell_u, \ell_u+1, \ldots, \ell_{u^+}-1\}$ one has

\

(i)  \quad  $\alpha^\vee_{i_u}(\varpi^+_{p_u^k})=\left\{\begin{array}{ll}0& : k <i_u ,\\
2&:k=i_u,\\
-2|p_u^k|& :k>i_u.\\
\end{array}\right.$

\

(ii)  \quad  If $\alpha_n \in \pi_{1,u}$, then
$\alpha^\vee_{i_u}(\varpi^-_{p_u^k})=\left\{\begin{array}{ll}0& : k <i_u ,\\
2&:k=i_u,\\
-2|p_u^k|& :k>i_u.\\
\end{array}\right.$

\

(iii) \quad  If $\alpha_n \notin \pi_{1,u}$, then
$\alpha^\vee_{i_u}(\varpi^-_{p_u^k})=\left\{\begin{array}{ll}0& : k = i_u, \text{or if} \  k>k_u,\\-1& : k \neq i_u, \alpha_k = i_{\pi^\mathbb Z_1}(\alpha_k), k\leq k_u ,\\

-2& : k \neq i_u, \alpha_k \neq i_{\pi^\mathbb Z_1}(\alpha_k), k\leq k_u .\\
\end{array}\right.$
\end {lemma}

\begin {proof} (i) and (ii) follow from Lemma \ref {8.3.1} noting that if $i_u < k$, then it is necessarily a left neighbour to $I(p_u^k)$.  (iii) follows from Lemma \ref {8.3.2}(ii) since $i_u$ is a neighbour to the two connected components of $\pi_{1,u}^\mathbb Z$ (one of which is empty if $i_u$ is on the boundary of $\pi_{1,n}$).  In this note that $\varpi_{p_u^k}^-=0$ if $k=i_u$ or if $k >k_u$, because the corresponding orbit is not in $\pi_{1,u}^\mathbb Z$.
\end {proof}

\textbf{Remark}.  When $k=i_u$ one has $|p_u^k|=1$.  This also holds under the hypothesis of (ii) but we have kept $|p_u^k|$ for comparison with (i).

\subsection{}\label{8.10}

\begin {cor}    Suppose $\alpha_n \in \pi_{1,u}$. Then $\alpha_{i_u}^\vee(\varpi_p)=0$, for all $p \in \mathscr {IP}_{\mathbb Z,u}$.
\end {cor}

\textbf{Remark.} This is not too surprising since as noted in \ref {8.6} under the hypothesis, the components $\pi_{1,n}$ and of $\pi_{1,n}^\mathbb Z$ provide the same number of generators.

\subsection{}\label{8.11}  Comparison of (i) and (iii) of Lemma \ref {8.9} gives

\begin {cor}  Suppose $\alpha_n \notin \pi_{1,u}$.  Then for all $ k \in \{\ell_u, \ell_u+1, \ldots, \ell_{u^+}-1\}$ one has
$\alpha^\vee_{i_u}(\varpi_{p_u^k})=\left\{\begin{array}{ll}1& : k < i_u, \alpha_k = i_{\pi^\mathbb Z_1}(\alpha_k) ,\\

2& : k < i_u, \alpha_k \neq i_{\pi^\mathbb Z_1}(\alpha_k),\\ 2&:k=i_u,\\-1&:k_u \geq k>i_u,\alpha_k = i_{\pi^\mathbb Z_1}(\alpha_k),\\ -2&:k_u \geq k>i_u,\alpha_k \neq i_{\pi^\mathbb Z_1}(\alpha_k),\\-2&:k>k_u.\\

\end{array}\right.$
\end {cor}

\subsection{}\label{8.12}  Recall that we are assuming that $\pi^{\frac{1}{2}}$ is defined by an adapted pair. We now apply the considerations of \ref {7.6} to prove integrality of that pair.

Take $u \in \overline{U}$ and assume for the moment that $\alpha_n \notin \pi_{1,u}$.

  Recall \ref {8.7}, which excludes having both $m,n \in 2\mathbb N$, implies that either $\alpha_k = i_{\pi^\mathbb Z_1}(\alpha_k)$ for some $k \in I$ satisfying $\ell_u\leq k <i_u$ (corresponding to $m \notin 2\mathbb Z$) or $\alpha_k = i_{\pi^\mathbb Z_1}(\alpha_k)$ for some $k \in I$ satisfying $i_u\leq k <i_u \leq k_u$ (corresponding to $n \notin 2\mathbb Z$) in Corollary \ref {8.11}.

 Again the set $\{k_u+1,\ldots,\ell_{u^+}-1\}$ is not empty.  Thus by Corollary \ref {8.11} we conclude that $\alpha^\vee_{i_u}(\varpi_p):p \in \mathscr I \mathscr P_{\mathbb Z,u}$ takes \textit{all} the values in $\{2,-2\}\cup \{1 \ \text{and/or} \ -1\}$.  Let $q \in \mathscr I \mathscr P_{\mathbb Z,u}$ be an element in which this value is $\pm 1$ and $p \in \mathscr I \mathscr P_{\mathbb Z,u}$ an element in which this value is $\mp 2$.   Then $q^2p \in \textbf{k}[p':p' \in \mathscr I \mathscr P_{\mathbb Z,u}]^{\alpha^\vee_{i_u}}$.

 By Lemma \ref {8.8} for all $u' \in \overline{U}$, the expression  $\alpha_{i_{u'}}^\vee(q^2p)$ is a non-negative integer multiple of $-2$ if $u' <u$ and is zero otherwise.  On the other hand by Lemma \ref {8.9} one has, for all $u', u'' \in \overline {U}$, that
 $$\alpha^\vee_{i_u''}(\varpi_{p_{u'}^{i_{u'}}})=\left\{\begin{array}{ll}-2& : \ \text{If} \ u'' \ \text{is a left neighbour of} \ I(p_{ u'}^{i_{u'}}) ,\\2&:u''=u',\\
0& :\text{otherwise}.\\
\end{array}\right.\eqno{(*)}$$

As a consequence there exists a unique up to scalars monomial $\hat{p}$ in the $p_{u'}^{i_{u'}}:u' \in \overline{U}, u'<u$ such that $q^2p \hat{p} \in Sy(\mathfrak q_\mathbb Z)^{\mathfrak h_\Gamma}$.  Moreover it is clear that $q^2p \hat{p}$ cannot be expressed in terms of the remaining generators of $Sy(\mathfrak q_\mathbb Z)^{\mathfrak h_\Gamma}$.

Finally suppose that $\alpha_n \notin \pi_{1,u}$.  Then as noted in \ref {8.6} the number of generators provided by the components $\pi_{1,u}$ and by $\pi_{1,u}^\mathbb Z$ is the same.

\begin {lemma} Assume that the adapted pair satisfies integrality.  Then for all $u \in \overline{U}$ such that $\alpha_n \notin \pi_{1,u}$, there exists exactly one element $q \in \mathscr {IP}_{\mathbb Z,u}$ such that $\alpha_{i_u}^\vee(q)\in \{\pm 1\}$.  Moreover these form the additional generators of $Sy(\mathfrak q_\mathbb Z)$ as a polynomial algebra over $Sy(\mathfrak q)$.
\end {lemma}

\begin {proof}  This follows from the above observation, the factorization property (\ref {7.6}(iii)) and the remarks in the last part of \ref {8.7} concerning the number of additional generators $s$ as expressed by \ref {8.7}$(**)$.
\end {proof}

 It is now quite easy to complete the proof of Theorem \ref {8.1}.   By \ref {8.7}$(*),(**)$ there exists $u \in \overline {U}$ such that $\alpha_n \notin \pi_{1,u}$.  Then by \ref {8.12}$(*)$, there exists for this choice of $u$ elements $p_1,p_2 \in \mathscr {IP}_{\mathbb Z,u}$ \textit{distinct} from the element in the conclusion of Lemma \ref {8.12} such that $\alpha_{i_u}^\vee(\varpi_{p_1p_2})=0$.  Then as in \ref {8.12} there exists a unique up to scalars monomial $\hat{p}$ in the $p_{u'}^{i_{u'}}:u' \in \overline{U}, u'<u$ such that $p_1p_2 \hat{p} \in Sy(\mathfrak q_\mathbb Z)^{\mathfrak h_\Gamma}$.  Moreover t is clear that $p_1p_2 \hat{p}$ cannot be expressed in terms of the remaining generators of $Sy(\mathfrak q_\mathbb Z)^{\mathfrak h_\Gamma}$.  This clearly contradicts the factorisation property (\ref {7.6}(iii)) and concludes the proof of integrality and hence of Theorem \ref {8.1}.

\section{Index of Notation.}

Symbols occurring frequently are given below in the paragraph where they are first defined.

\

1.1  $\mathfrak a,\textbf{A},  S(\mathfrak a), Y(\mathfrak a), \mathscr N(\mathfrak a),\mathfrak a^*, \mathfrak a^\xi,\ell(\mathfrak a), \mathfrak a^\xi,, \mathfrak a^*_{reg},I(\mathfrak a)$.

1.2  $F(\mathfrak a), C(\mathfrak a), Sy(\mathfrak a), \Lambda, \mathfrak a_\Lambda, r\ell(\mathfrak a)$.

1.3  $\Delta, \pi, \mathfrak g_\pi, x_\alpha, \kappa, s_\alpha, \varpi_\alpha, \mathfrak b_{\pi},\mathfrak p_{\pi_1},\mathfrak q_{\pi_1,\pi_2}$.

1.4  $\mathfrak a_i, \mathfrak a_\mathbb Z$.

1.6  $\mathfrak r,\mathfrak m, \mathfrak h, \varphi$.

2.2  $m_i, d_i, \mathfrak z$.

2.4  $\psi$.

2.5  $\mathfrak c$.

2.6  $\mathscr P$.

2.7  $q_j,J$.

2.8  $p_i.\hat{p_i}$.

2.11 $\mathscr J, \Gamma$.

4.1  $\mathfrak a_\alpha, S$.

4.3  $\mathfrak a_\mathbb Q, \mathfrak a_\mathscr J$.

5.2  $B_\pi$.

5.3  $W_\pi, w_\pi, i_\pi$.

6.1  $\alpha^\vee, \mathfrak g_\mathbb Z, \Delta^\mathbb Z, \pi^\mathbb Z, \pi_\cap, W_\cap, \Delta^+$.

6.2  $\pi_1^\mathbb,\pi_2^\mathbb Z$.

6.8  $M_\pi$.

7.1  $\pi^{\frac{1}{2}}$.

7.2  $\pi_{1,u},\pi_{2,v},\pi_{\cap(u,v)}$.

7.5  $\mathfrak h_\Lambda, \mathfrak h_\Gamma$.

8.1  $\tilde{\pi}$.

8.2  $[i,j],s[i,j], \beta_i$.

8.2.1  $\pi^{\mathbb Z, \ell}, \pi^{\mathbb Z,r}, I^\ell, I^r$.

8.3  $\delta_\Gamma, \delta_{\Gamma^\mathbb Z}$.

8.3.1  $I(i)$.

8.6  $\overline {\pi}^{\frac{1}{2}}, \overline{U}$.

8.8  $\mathscr P_\mathbb Z, \mathscr I \mathscr P_{\mathbb Z,u}, |p|, \varpi_p, \varpi_p^+, \varpi_p^-$.

\end{document}